\documentclass[a4paper,10pt]{article}
\usepackage[T2A]{fontenc}
\usepackage[utf8]{inputenc}

\usepackage[margin=2.5cm]{geometry}

\usepackage{amsmath, amssymb, amsthm}
\usepackage{hyperref}
\usepackage{xcolor}
\usepackage{color, colortbl}
\usepackage{comment}
\usepackage{tikz}

\usepackage[shortlabels]{enumitem}

\makeatletter
\newcommand{\customitem}[1]{%
\item[#1]\protected@edef\@currentlabel{#1}%
}
\makeatother

\usepackage{hhline}
\usepackage{diagbox}
\usepackage{multirow}

\definecolor{myred}{rgb}{0.81, 0.06, 0.13}
\definecolor{myblue}{rgb}{0,0.5,0.6}

\def\lrc{\lambda_{rc}}
\def\lrr{\lambda_{rr}}
\def\lcc{\lambda_{cc}}

\def\Z{{\mathbb Z}}
\def\SNTA{S_{\mathrm{NTA}}}
\def\STA{S_{\mathrm{TA}}}
\def\SNBG{S_{\mathrm{NBG}}}
\def\SBG{S_{\mathrm{BG}}}

\DeclareMathOperator{\Aut}{Aut}

\newtheorem{theorem}{Theorem}[section]
\newtheorem{remark}[theorem]{Remark}

\newtheorem{corollary}[theorem]{Corollary}
\newtheorem{proposition}[theorem]{Proposition}
\newtheorem{lemma}[theorem]{Lemma}
\newtheorem{definition}[theorem]{Definition}

\newtheorem{question}[theorem]{Question}
\newtheorem{conjecture}[theorem]{Conjecture}

\begin{document}

\title{Near Triple Arrays}
\author{Alexey Gordeev, Klas Markstr{\"o}m, Lars-Daniel {\"O}hman\\
\small{Department of Mathematics and Mathematical Statistics, Ume\r{a} University, Sweden}}
\date{}

\maketitle

\begin{abstract}
We introduce \textit{near triple arrays} as binary row-column designs with at most two consecutive values for the replication numbers of symbols, for the intersection sizes of pairs of rows, pairs of columns and pairs of a row and a column.
Near triple arrays form a common generalization of such well-studied classes of designs as triple arrays, (near) Youden rectangles and Latin squares.

We enumerate near triple arrays for a range of small parameter sets and show that they exist in the vast majority of the cases considered.
As a byproduct, we obtain the first complete enumerations  of $6 \times 10$ triple arrays on $15$ symbols, $7 \times 8$ triple arrays on $14$ symbols and $5 \times 16$ triple arrays on $20$ symbols.

Next, we give several constructions for families of near triple arrays, and e.g. show that near triple arrays with 3 rows and at least 6 columns exist for any number of symbols.
Finally, we investigate a duality between row and column intersection sizes of a row-column design, and covering numbers for pairs of symbols by rows and columns. These duality results are used to obtain necessary conditions for the existence of near triple arrays. This duality also provides a new unified approach to earlier results on triple arrays and balanced grids.
\end{abstract}

\section{Introduction}\label{sec:intro}
The study of experimental designs that allow for eliminating the influence 
of multiple factors on an experiment began with the early works of Fisher and was developed further by among others Agrawal in the 1950:s and 60:s. 
Agrawal~\cite{agrawalMethodsConstructionDesigns1966} 
introduced a class of experimental designs 
that would later be known as \emph{triple arrays}. A triple array is 
a binary (no repeated symbols in any row or column), equireplicate (each symbol occurs the same number of times) array such that the following three intersection conditions hold:

\begin{enumerate}
    \customitem{(RC)}\label{TA:rc} Any row and column have a constant number of symbols in common,
	\customitem{(RR)}\label{TA:rr} Each pair of distinct rows has a constant number of symbols in common,
	\customitem{(CC)}\label{TA:cc} Each pair of distinct columns has a constant number of symbols in common.
\end{enumerate}

Condition~\ref{TA:rc} is often referred to as \textit{adjusted orthogonality}.
Bagchi and Shah~\cite{bagchiOptimalityClassRowcolumn1989} showed that triple arrays are statistically optimal among all binary equireplicate arrays with regard to a large class of optimality criteria.

Unfortunately, the intersection conditions~\ref{TA:rc}, \ref{TA:rr}, \ref{TA:cc} lead to some very restrictive constraints on the possible sizes of such designs.
For example, among arrays with at most $50$ cells, excluding Latin squares and assuming that at least one symbol is repeated, triple arrays exist only for sizes $3 \times 7$, $4 \times 7$, $5 \times 6$ and $4 \times 9$.
In the more recent literature, for example McSorley, Phillips, Wallis and Yucas~\cite{mcsorleyDoubleArraysTriple2005a}, these types of designs have mainly been studied without any relaxations on the three conditions stated above.
When relaxed versions have been studied, it has been done by completely removing one of the three conditions. For example, Bailey, Cameron and Nilson~\cite{baileySesquiarraysGeneralisationTriple2018} studied  \emph{sesqui arrays}, which are arrays where condition~\ref{TA:cc} is removed. 
This sacrifices the possibility of eliminating one factor affecting the experiment, while unfortunately not expanding the range of possible sizes very much.

In previous work by some of the present authors~\cite{jagerSmallYoudenRectangles2023}, relaxations
of \emph{Youden rectangles}, which likewise suffer from a restricted number of possible sizes, were studied.
Youden rectangles can be seen as $r \times v$ triple arrays on $v$ symbols, that is, the number of symbols is the same as the number of columns.
Note that for Youden rectangles conditions~\ref{TA:rc} and~\ref{TA:rr} hold trivially.
In~\cite{jagerSmallYoudenRectangles2023}, \emph{near Youden rectangles} were defined by relaxing condition~\ref{TA:cc} to allow two consecutive column intersection sizes as close as possible to the average intersection size.
It turns out that such arrays exist for a large majority of small sizes.

Our general approach in the present paper, which is similar to that of~\cite{jagerSmallYoudenRectangles2023}, is to define \textit{near triple arrays} by relaxing all three intersection properties to allow two values concentrated around the average intersection size, and also allowing two consecutive values as replication numbers.
This definition includes Latin squares, (near) Youden rectangles and triple arrays, which all are near triple arrays with additional restrictions placed on the number of rows, columns and symbols: for example, Latin squares are precisely $n \times n$ near triple arrays on $n$ symbols.
The class of near triple arrays is thus a common generalization of these well-known types of designs.

\begin{figure}[!ht]
\begin{center}
\begin{tabular}{c c c}
    \begin{tabular}{|*{9}{c}|}
        \hline
        0 & 1 & 2 & 3 & 4 & 5 & 6 & 7 & 8\\
        1 & 2 & 3 & 4 & 5 & 9 & 7 & 10 & 11\\
        4 & 6 & 0 & 9 & 10 & 7 & 11 & 8 & 2\\
        11 & 9 & 10 & 8 & 6 & 0 & 3 & 1 & 5\\
        \hline
    \end{tabular}
    & &
    \begin{tabular}{|*{6}{c}|}
        \hline
        0 & 1 & 2 & 3 & 4 & 5\\
        1 & 2 & 3 & 6 & 7 & 8\\
        4 & 0 & 6 & 8 & 5 & 7\\
        6 & 7 & 5 & 0 & 3 & 1\\
        \hline
    \end{tabular}    
\end{tabular}
\caption{A $4 \times 9$ triple array on $12$ symbols and a $4 \times 6$ near triple array on $9$ symbols.
In the first, each symbol occurs 3 times, any two rows share 6 symbols, any two columns share 1 symbol, and any row and column share 3 symbols.
In the second, each symbol occurs 2 or 3 times, any two rows share 3 or 4 symbols, any two columns share 1 or 2 symbols, and any row and column share 2 or 3 symbols.}
\label{fig:NTA}
\end{center}
\end{figure}

The \textit{column design} of a triple array is a block design with points and blocks corresponding to, respectively, columns and symbols of the array, with a point appearing in a block when the corresponding column contains the corresponding symbol.
The \textit{row design} can be defined similarly.
In the column design of a triple array, all blocks have the same size and each pair of points is covered by the same number of blocks, so it is a \textit{balanced incomplete block design (BIBD)}.
The column design of a near triple array belongs to a more general class of \textit{maximally balanced maximally uniform designs (MBMUDs)}, which were studied by Bofill and Torras~\cite{MBMUD}.
When the array is equireplicate, the column design belongs to the more well-known class of \textit{regular graph designs}, first introduced by John and Mitchell~\cite{johnOptimalIncompleteBlock1977}.
Our relaxation of triple arrays is thus also parallel to earlier efforts to generalize BIBDs.
We give a necessary condition for the existence of the column design of a near triple array and use it to show that there are no near triple arrays for several families of parameter sets. Conversely, we also give constructions for some families of near triple arrays, and thus prove existence for certain parameter sets.

Near triple arrays will by definition minimize the deviation of the size of the pairwise column intersections from their average.
In \cite{jagerSmallYoudenRectangles2023} it was shown that for equireplicate arrays with the same number of columns and symbols this will also make the covering numbers, i.e. the number of columns which contain a given pair of symbols, deviate as little as possible from their average.  As we will see later, for near triple arrays in general this is no longer true. Given that many classical block designs are defined in terms of their covering numbers,  it also becomes natural to consider arrays where the covering numbers are restricted to two consecutive values. When taking into account covering symbol pairs by both columns and rows, this leads to a family of arrays we call \textit{near balanced grids}.

Near balanced grids are in turn a generalization of \textit{balanced grids}, introduced by McSorley, Phillips, Wallis and Yucas~\cite{mcsorleyDoubleArraysTriple2005a}. Balanced grids are defined as binary arrays in which the number of columns and rows containing a given pair of symbols is constant, i.e. does not depend on the choice of the pair.
They showed that any $r \times c$ triple array on $r + c - 1$ symbols is a balanced grid, and later McSorley~\cite{mcsorleyDoubleArraysTriple2005} showed that the converse also holds.
We give an alternative proof of this result and show that it is in fact a special case of a more general relationship between near triple arrays and near balanced grids. More precisely, we show that for any given numbers of rows, columns and symbols, either near triple arrays on these parameters are precisely near balanced grids, or only one of these two classes of designs are possible for these parameters.

In addition to the results on the theory of near triple arrays, we also  completely enumerate near triple arrays for a range of small parameter sets, and show that they exist for a vast majority of combinations of small sizes of an array and the number of symbols.
As a byproduct of our enumeration efforts, we also obtain the first classification up to isotopism of $6 \times 10$ triple arrays on $15$ symbols, $7 \times 8$ triple arrays on $14$ symbols and $5 \times 16$ triple arrays on $20$ symbols.
We also extend previous enumerative results for (near) Youden rectangles, sesqui arrays and several other types of row-column designs.

\subsection{Overview}
The rest of the paper is structured as follows.
In Section~\ref{sec:defs}, we give the central definitions and discuss connections between near triple arrays and other types of row-column designs.
In Section~\ref{sec:enum}, we describe the algorithm we used to enumerate near triple arrays.
In Section~\ref{sec:res}, we present and discuss the results of the enumeration.
In Section~\ref{sec:exist}, we give general constructions and existence proofs.
In Section~\ref{sec:comp}, we study row and column designs of near triple arrays and derive non-existence conditions.
In Section~\ref{sec:near_balance}, we investigate the relationship between near triple arrays and near balanced grids.
In Section~\ref{sec:correct}, we describe measures taken to ensure correctness of our enumerative results.
Section~\ref{sec:concl} concludes with a list of open questions.

\section{Definitions and basic properties}\label{sec:defs}

An $r \times c$ \textit{row-column design} on $v$ symbols is a two-dimensional array with $r$ rows and $c$ columns, each cell of which is filled with one of $v$ symbols.
It is \textit{binary} if no symbol appears more than once in any row or column. 
In order for a row-column design to be binary, clearly the number of
symbols must be at least as large as the number of rows and columns,
that is, $\max(r,c) \leq v$.
On the other hand, if the number of symbols is larger than the number of cells, then some symbols are not even in use.
We therefore restrict our investigation to $\max(r,c) \leq v \leq rc$. To avoid rather trivial examples, we also generally only consider $r,c \geq 3$.

The \textit{average replication number} of a row-column design is $e := rc / v$.
Let $e^- := \lfloor e \rfloor$ and $e^+ := \lceil e \rceil$.
If $e$ is an integer and every symbol occurs $e$ times in the array, the row-column design is called \textit{equireplicate} with \textit{replication number} $e$. 
We will call a row-column design \textit{near equireplicate} if $e$ is not an integer and every symbol occurs in the array either $e^-$ or $e^+$ times.
For near equireplicate designs, we can count the number of occurrences of
the symbols as in the following lemma.

\begin{lemma}\label{lm:k-k+}
In a near equireplicate $r \times c$ row-column design on $v$ symbols, there are $v_- := v(e^+ - e)$ symbols occurring $e^-$ times and $v_+ := v(e - e^-)$ symbols occurring $e^+$ times.
\end{lemma}
\begin{proof}
Counting the total number of symbols and the total number of cells in the array, we get
\[
\begin{cases}
v_- + v_+ = v,\\
v_-e^- + v_+e^+ = rc.
\end{cases}
\]
Remembering that $e=rc/v$, the unique solution of this system is $v_- = v(e^+ - e)$, $v_+ = v(e - e^-)$.
\end{proof}

For a binary $r \times c$ row-column design on $v$ symbols, we denote by $\lrc$, $\lrr$, $\lcc$ the average number of common symbols between a row and a column, two rows, and two columns, respectively. More formally, if $R_i$ is the set of symbols occurring in row $i$ and $C_j$ is the set of symbols occurring in column $j$, then
\[
    \lrc := \frac{1}{rc}\sum_{i = 1}^r\sum_{j = 1}^c |R_i \cap C_j|; 
    \quad \lrr := \frac{1}{\binom{r}{2}}
    \sum_{1 \leq i < j \leq r} |R_i \cap R_j|; 
    \quad \lcc := \frac{1}{\binom{c}{2}}
    \sum_{1 \leq i < j \leq c} |C_i \cap C_j|.
\]

For a proof of the next lemma in a more restrictive setting, see Theorem 2.2 and Theorem 3.1 in~\cite{mcsorleyDoubleArraysTriple2005a}.

\begin{lemma}\label{lm:equiavg}
For a binary equireplicate $r \times c$ row-column design on $v$ symbols,
\begin{enumerate}[(a)]
    \item\label{eqavg:lrc} $\lrc = e$,
    \item\label{eqavg:lrr} $\lrr = \frac{c(e - 1)}{r - 1} = \frac{c(\lrc - 1)}{r - 1}$,
    \item\label{eqavg:lcc} $\lcc = \frac{r(e - 1)}{c - 1} = \frac{r(\lrc - 1)}{c - 1}$.
\end{enumerate}
\end{lemma}

\begin{proof}
Each of the $v$ symbols occurs $e$ times, and, since the design is binary, each symbol is common to $e^2$ pairs of a row and a column.
Thus, the number of common symbols summed over all pairs of a row and a column is $rc\lrc = ve^2$, which implies \ref{eqavg:lrc}.

Similarly, each symbol is a common symbol of $\binom{e}{2}$ pairs of rows.
Thus, the number of common symbols summed over all pairs of two rows is
\[
\binom{r}{2}\lrr = v\binom{e}{2} = \frac{ve(e - 1)}{2} = \frac{rc(e - 1)}{2},
\]
which implies \ref{eqavg:lrr}.
A similar count gives \ref{eqavg:lcc}.
\end{proof}

We will need to extend Lemma~\ref{lm:equiavg} to the near equireplicate case.
Let $\lrc^- := \lfloor\lrc\rfloor$, $\lrc^+ := \lceil\lrc\rceil$, and define $\lrr^-$, $\lrr^+$, $\lcc^-$ and $\lcc^+$ similarly.

\begin{lemma}\label{lm:avg}
For a binary (near) equireplicate $r \times c$ row-column design on $v$ symbols,
\begin{enumerate}[(a)]
\item\label{avg:lrc} $\lrc = e^- + e^+ - \frac{e^-e^+}{e} = e + \frac{(e^+ - e)(e - e^-)}{e}$,
\item\label{avg:lrc+-} $\lrc^- = e^-$ and $\lrc^+ = e^+$,
\item\label{avg:lrr} $\lrr = \frac{c(\lrc - 1)}{r - 1}$,
\item\label{avg:lcc} $\lcc = \frac{r(\lrc - 1)}{c - 1}$.
\end{enumerate}
\end{lemma}

\begin{remark}\label{rem:elrc}
Lemma~\ref{lm:avg}~\ref{avg:lrc} implies that $\lrc \geq e$, and that the equality holds if and only if $e$ is an integer, i.e. the row-column design is equireplicate.
\end{remark}

\begin{proof}[Proof of Lemma~\ref{lm:avg}]
For an equireplicate design all claims hold due to Lemma~\ref{lm:equiavg}, so from now on we assume that it is near equireplicate, that is, $e^- < e < e^+$.
Noting that $v=v_-+v_+$, a double counting argument similar to the proof of Lemma~\ref{lm:equiavg} together with Lemma~\ref{lm:k-k+} gives that the number of common symbols summed over all pairs of a row and a column is
\[
rc\lrc = v_-(e^-)^2 + v_+(e^+)^2 = v(e^+ - e)(e^-)^2 + v(e - e^-)(e^+)^2.
\]
Dividing by $v$ and using $e=rc/v$, we get
\begin{equation}\label{eq:elrc}
e\lrc = (e^+ - e)(e^-)^2 + (e - e^-)(e^+)^2 = e((e^+)^2 - (e^-)^2) + e^-e^+(e^- - e^+).
\end{equation}
Note that $e^+ - e^- = 1$, so $(e^+)^2 - (e^-)^2 = e^+ + e^-$, thus
\[
e\lrc= e(e^+ + e^-) - e^-e^+,
\]
from which \ref{avg:lrc} follows.
In order to prove \ref{avg:lrc+-} it suffices to show that $e^- < \lrc < e^+$, which follows from \ref{avg:lrc}:
\[
e^- < e < e + \frac{(e^+ - e)(e - e^-)}{e} = \lrc,
\]
\[
\lrc = e + \frac{(e^+ - e)(e - e^-)}{e} = e + (e^+ - e) \cdot \frac{e - e^-}{e} < e + (e^+ - e) = e^+.
\]
Again using double counting and Lemma~\ref{lm:k-k+}, the number of common symbols summed over all pairs of rows is
\[
\binom{r}{2}\lrr = v_- \binom{e^-}{2} + v_+ \binom{e^+}{2} = v(e^+ - e) \binom{e^-}{2} + v(e - e^-) \binom{e^+}{2}.
\]
Dividing by $v / 2$, we get
\begin{align*}
\frac{r(r - 1)\lrr}{v} &= (e^+ - e)e^-(e^- - 1) + (e - e^-)e^+(e^+ - 1)\\
&= e((e^+)^2 - (e^-)^2) + e^+e^-(e^- - e^+) + e(e^- - e^+)\\
&= e\lrc - e,
\end{align*}
where the last equality follows from~\eqref{eq:elrc}.
This implies \ref{avg:lrr}, and a similar count gives \ref{avg:lcc}.
\end{proof}

Having done this ground work, we now formally define our main object of study.

\begin{definition}\label{def:NTA}
An $(r \times c, v)$-near triple array is a binary (near) equireplicate $r \times c$ row-column design on $v$ symbols in which
\begin{enumerate}
\item any row and column have either $\lrc^-$ or $\lrc^+$ common symbols,
\item any two rows have either $\lrr^-$ or $\lrr^+$ common symbols,
\item any two columns have either $\lcc^-$ or $\lcc^+$ common symbols.
\end{enumerate}
\end{definition}

Note that the definition allows a near triple array to be equireplicate. To avoid having to calculate the quantities involved in the definition
explicitly, it will often be convenient to instead use the following alternative characterization of near triple arrays.

\begin{proposition}\label{prop:altdefNTA}
A binary $r \times c$ row-column design on $v$ symbols in which, for some integers $x$, $x_{rc}$, $x_{rr}$ and $x_{cc}$,
\begin{enumerate}
\item any symbol occurs $x$ or $x + 1$ times,
\item any row and column have either $x_{rc}$ or $x_{rc} + 1$ common symbols,
\item any two rows have either $x_{rr}$ or $x_{rr} + 1$ common symbols,
\item any two columns have either $x_{cc}$ or $x_{cc} + 1$ common symbols,
\end{enumerate}
is an $(r \times c, v)$-near triple array. Conversely, any near triple
array satisfies all these conditions.
\end{proposition}
\begin{proof}
Since $e$ is the average replication number, it holds that $x \leq e \leq x + 1$, 
by elementary properties of averages. 
If $e = x$ (or $e = x + 1$), then each symbol occurs $x$ (or $x + 1$) times, in which case the design is equireplicate. If $x < e < x + 1$, it follows that $e^{-} = x$ and $e^{+} = x+1$, and the design is near equireplicate.
By similar arguments, all other conditions from Definition~\ref{def:NTA} hold. The converse follows trivially from Definition~\ref{def:NTA}.
\end{proof}

In the remainder of this section, we describe how several well-studied classes of row-column designs are in fact near triple arrays with parameters $r, c, v$ restricted.

An \textit{$(r \times c, v)$-triple array} is a binary equireplicate $r \times c$ row-column design on $v$ symbols in which any row and column have a constant number of symbols in common, and the same is true for any pair of rows and for any pair of columns.
Clearly, a necessary condition for the existence of a triple array is that the corresponding $e$, $\lrc$, $\lrr$, $\lcc$ are integers.
We will call parameters $(r \times c, v)$ \textit{admissible for triple arrays} when this is the case.
Note that for admissible parameters $(r \times c, v)$, near triple arrays are precisely triple arrays.

For a near triple array in general, each of $\lrc$, $\lrr$ and $\lcc$ can either be or not be an integer, as can be seen by considering arrays in Figures~\ref{fig:NTA} and~\ref{fig:exnta} and their transposes.
Due to Remark~\ref{rem:elrc}, we know that $\lrc$ is an integer precisely when the design is equireplicate.

\begin{figure}[!ht]
\begin{center}
    \begin{tabular}{c c c}
	\begin{tabular}{|*{6}{c}|}
        \hline
        0 & 1 & 2 & 3 & 4 & 5\\
        1 & 2 & 0 & 6 & 7 & 8\\
        3 & 9 & 7 & 5 & 8 & 0\\
        6 & 5 & 9 & 7 & 1 & 4\\
        \hline
        \multicolumn{6}{c}{}\\[-0.2cm]
        \multicolumn{6}{c}{(a)}
    \end{tabular}
    & &
    \begin{tabular}{|*{6}{c}|}
        \hline
        0 & 1 & 2 & 3 & 4 & 5\\
        1 & 0 & 3 & 4 & 2 & 6\\
        2 & 3 & 0 & 5 & 6 & 1\\
        4 & 6 & 5 & 1 & 3 & 2\\
        \hline
        \multicolumn{6}{c}{}\\[-0.2cm]
        \multicolumn{6}{c}{(b)}
    \end{tabular}\\
	& &\\
    \begin{tabular}{|*{7}{c}|}
        \hline
        0 & 1 & 2 & 3 & 4 & 5 & 6\\
        1 & 2 & 7 & 4 & 8 & 9 & 10\\
        9 & 11 & 12 & 13 & 6 & 3 & 7\\
        12 & 13 & 5 & 10 & 11 & 8 & 0\\
        \hline
        \multicolumn{7}{c}{}\\[-0.2cm]
        \multicolumn{7}{c}{(c)}
    \end{tabular}
    & &
    \begin{tabular}{|*{6}{c}|}
        \hline
        0 & 1 & 2 & 3 & 4 & 5\\
        1 & 6 & 3 & 7 & 8 & 9\\
        8 & 5 & 10 & 11 & 2 & 7\\
        10 & 11 & 9 & 4 & 6 & 0\\
        \hline
        \multicolumn{6}{c}{}\\[-0.2cm]
        \multicolumn{6}{c}{(d)}
    \end{tabular}
	\end{tabular}
	\end{center}
\caption{Near triple arrays on parameters (a) $(4 \times 6, 10)$, (b) $(4 \times 6, 7)$, (c) $(4 \times 7, 14)$ and (d) $(4 \times 6, 12)$.
Designs (c) and (d) are equireplicate, and so have integer $\lrc$.
Designs (a), (b) and (d) have integer $\lrr$, and design (b) has integer $\lcc$.}
\label{fig:exnta}
\end{figure}

A binary equireplicate $r \times n$ row-column design on $n$ symbols is a \textit{Latin rectangle}, and if $r=n$ it is called a 
\textit{Latin square}. We note that for any Latin square, the conditions of Definition~\ref{def:NTA} are satisfied trivially, so $(n \times n, n)$-near triple arrays are precisely $n \times n$ Latin squares.
A \textit{Youden rectangle} is a Latin rectangle in which any pair of columns has $\lambda$ common symbols for some fixed $\lambda$, and a \textit{near Youden rectangle} is a Latin rectangle in which any pair of columns has either $\lcc^-$ or $\lcc^+$ symbols in common.
We see that $(r \times n, n)$-near triple arrays are precisely $r \times n$ (near) Youden rectangles.

\section{Enumeration method}\label{sec:enum}

The group $G_{v, r, c} := S_v \times S_r \times S_c$ of \textit{isotopisms} acts on the set of $(r \times c, v)$-near triple arrays, where $S_v$, $S_r$ and $S_c$ correspond to permutations of symbols, rows and columns, respectively.
Two $(r \times c, v)$-near triple arrays $T$ and $T'$ are \textit{isotopic} (belong to the same \textit{isotopism class}) if there exists $\varphi \in G_{v, r, c}$ such that $\varphi(T) = T'$.
An \textit{autotopism} of $T$ is an isotopism $\varphi \in G_{v, r, c}$ such that $\varphi(T) = T$, and all autotopisms of $T$ form its \textit{autotopism group} $\Aut(T) \leq G_{v, r, c}$. We will generally only report the size of the autotopism group, but note that this does not uniquely identify the group from an abstract point of view, and in turn, knowing the structure of the group does not determine how it acts on the array.

For a given set of parameters $r, c, v$, we generate all non-isotopic $(r \times c, v)$-near triple arrays on the set of symbols $\{0, 1, \dots, v - 1\}$ by starting from an empty $r \times c$ array and filling cells with symbols, one cell at a time, in the order
\begin{equation}\label{eq:order}
(1, 1), (1, 2), \dots, (1, c), (2, 1), \dots, (2, c), \dots, (r, 1), \dots, (r, c),
\end{equation}
that is, row by row, and within each row in order of increasing column number.
The partial objects we consider during the generation are thus \textit{partially filled arrays} where only the cells corresponding to a prefix of the sequence~\eqref{eq:order} are filled.
We define a subgroup of isotopisms $G_{v,r,c}^i \leq G_{v,r,c}$ that acts on the set of partially filled arrays with the first $i$ cells filled, such that isotopisms in $G_{v,r,c}^i$ consist only of permutations of rows and columns which map the set of the first $i$ cells of the sequence~\eqref{eq:order} onto itself.
A \textit{canonical representative of an isotopism class}, or simply a \textit{canonical} array, is the array which is lexicographically minimal in its isotopism class when viewed as a sequence of symbols written in order~\eqref{eq:order}.

At each step of the enumeration process, we exclude a partially filled array from further consideration if we can deduce that it cannot be completed to a near triple array (see Section~\ref{ssec:compl}) or if it is not canonical (see Section~\ref{ssec:canon}).
At the end of the process, we reject a complete object if it is not canonical.
Note that a consequence of this is that the complete objects produced are normalized, in the sense that their first row is $0, 1, \ldots, c-1$ and cell $(2,1)$ has symbol $1$.

Every non-rejected complete object is a canonical near triple array, so no two arrays from the same isotopism class are generated.
On the other hand, any partially filled array which can be completed to a canonical near triple array must itself be canonical, thus all non-isotopic near triple arrays are generated.
This approach is often called \textit{orderly generation}, and was introduced independently by Farad\v{z}ev~\cite{faradvzev1978constructive} and Read~\cite{read1978every}.
See chapter 4 of Kaski and \"Osterg{\aa}rd's book~\cite{kaskiClassificationAlgorithmsCodes2006} for an overview of such algorithms in a more general setting.

Similar to many other design enumeration problems, the number of partial objects during our search often exceeds the number of complete near triple arrays by many orders of magnitude.
This has been a major bottleneck in several recent enumeration efforts for other row column designs~\cite{jagerSmallYoudenRectangles2023,jagerEnumerationRowColumnDesigns2024}.
We address this problem in two ways.

First, a major advantage of our approach is that different branches of the search tree are completely independent.
In particular, no additional cross-checks of the results to remove isotopic copies of the same array are required, so there is no need to simultaneously store all partial objects obtained at any given step.
It means that, even though the number of partial objects may be huge, we are not limited by storage space constraints.
The independence of different branches of the search tree also makes the parallelization of tasks straightforward.

Second, as stated earlier, we generate near triple arrays by filling one cell at a time.
This is contrary to a more common approach, when row-column designs, or, more generally, incidence matrices of combinatorial structures, are generated one row at a time.
Making the steps of the generation procedure as small as possible allows us to detect and reject non-completable partial objects much sooner, reducing the total size of the search tree.

The algorithms used were implemented in C++ and run in a parallelized version on the Kebnekaise supercomputer at High Performance Computing Center North (HPC2N).
The source code is available at~\cite{zenodo}.
The most computationally heavy case, $(7 \times 7, 11)$, required roughly 17 core years.
The total run time for all cases was around a hundred core years.

\subsection{Completability of partial objects}\label{ssec:compl}

Let $T$ be a partially filled array.
If $T$ is \textit{non-completable}, that is, cannot be completed to a near triple array, we would like to be able to detect this as early as possible, so that we can drop $T$ from consideration in our search and thus avoid doing unnecessary work.
We try to deduce if $T$ is non-completable via the following procedure.

Denote by $R_i$ and $C_j$, respectively, sets of symbols appearing in row $i$ and column $j$ of $T$.
For each row $i$ of $T$, we maintain two sets of symbols $A_{R_i}$ and $B_{R_i}$ which ``estimate'' the possible content of row $i$ in the completed array in the following way: For any near triple array $T'$ completed from $T$, if the set of symbols appearing in row $i$ of $T'$ is $R'_i$, then
\[
A_{R_i} \subseteq R'_i \subseteq B_{R_i}.
\]
Since for any such $T'$, the set $R'_i$ would contain the set $R_i$,
a trivial way to get these estimates is to set $A_{R_i}$ to be $R_i$, and $B_{R_i}$ to be the set of all $v$ symbols.

As an example of how the sets $A_{R_i}$ and $B_{R_i}$ can be used, note that the set of common symbols $R'_i \cap R'_j$ between two rows $i$ and $j$ in a completed array $T'$ must satisfy
\begin{equation}\label{eq:RRintr}
A_{R_i} \cap A_{R_j} \subseteq R'_i \cap R'_j \subseteq B_{R_i} \cap B_{R_j}.
\end{equation}
In particular, if $|A_{R_i} \cap A_{R_j}| > \lrr^+$ or $|B_{R_i} \cap B_{R_j}| < \lrr^-$, then $T$ is non-completable.
Thus, the better estimate sets we have, that is, the larger sets $A_{R_i}$ and the smaller sets $B_{R_i}$, the more we can say about the non-completability of $T$.

Fortunately, intersection conditions like~\eqref{eq:RRintr} can also sometimes be used to improve the estimate sets. For example, if $|A_{R_i} \cap A_{R_j}| = \lrr^+$, then the contents $R'_i$ and $R'_j$ of rows $i$ and $j$ of any near triple array $T'$ completed from $T$ must satisfy
\[
    R'_i \subseteq B_{R_i} \setminus (A_{R_j} \setminus A_{R_i}),
    \quad R'_j \subseteq B_{R_j} \setminus (A_{R_i} \setminus A_{R_j}),
\]
so we can replace the estimate sets $B_{R_i}$ and $B_{R_j}$ with the smaller sets $B_{R_i} \setminus (A_{R_j} \setminus A_{R_i})$ and $B_{R_j} \setminus (A_{R_i} \setminus A_{R_j})$, respectively.
Similarly, if $|B_{R_i} \cap B_{R_j}| = \lrr^-$, then it must hold that
\[
    A_{R_i} \cup (B_{R_i} \cap B_{R_j}) \subseteq R'_i,
    \quad A_{R_j} \cup (B_{R_i} \cap B_{R_j}) \subseteq R'_j,
\]
so in this case we can replace $A_{R_i}$ and $A_{R_j}$ with the larger sets $A_{R_i} \cup (B_{R_i} \cap B_{R_j})$ and $A_{R_j} \cup (B_{R_i} \cap B_{R_j})$, respectively.

In a similar manner, we also maintain estimate sets $A_{C_j}$ and $B_{C_j}$ satisfying
\[
    A_{C_j} \subseteq C'_j \subseteq B_{C_j}
\]
for the content $C'_j$ of column $j$ in any near triple array $T'$ completed from $T$, and use intersection conditions analogous to~\eqref{eq:RRintr} to estimate the sets of common symbols between a pair of columns, or between a row and a column in a completed array, as well as to update the estimate sets.

In addition, for each empty cell $(i, j)$ of $T$ we maintain an estimate set $B_{i, j}$ satisfying $T'_{i, j} \in B_{i, j}$ for any near triple array $T'$ completed from $T$.
A trivial estimate set $B_{i, j}$ is the set of all symbols not yet used in row $i$ and column $j$ of $T$.
For convenience, we also set $B_{i, j} = \{T_{i, j}\}$ for cells $(i,j)$ of $T$ which are already filled.
We use the conditions
\begin{equation}\label{eq:cellCond}
    B_{i, j} \subseteq (B_{R_i} \setminus R_i) \cap (B_{C_j} \setminus C_j),\quad
    B_{R_i} \subseteq \bigcup_{j = 1}^c B_{i, j},\quad
    B_{C_j} \subseteq \bigcup_{i = 1}^r B_{i, j},
\end{equation}
where the first condition is only for empty cells $(i, j)$ of $T$,
to try to deduce that $T$ is non-completable, or to further improve the estimate sets.
For example, $(B_{R_i} \setminus R_i) \cap (B_{C_j} \setminus C_j) = \emptyset$ would imply non-completability.

Starting from trivial estimate sets, we iteratively check all row-row, row-column and column-column intersection conditions as well as the conditions in~\eqref{eq:cellCond} trying to either deduce that $T$ is non-completable, or to improve the estimate sets.
We stop either when we have successfully shown that $T$ is non-completable, in which case we remove $T$ from further consideration in our search, or when we can no longer improve any of the estimate sets.
As a byproduct of this procedure, for a non-rejected array we in particular obtain a set of symbols $B_{i, j}$ which may be put into the next empty cell $(i,j)$ of $T$ without directly violating properties of near triple arrays. A complete and more rigorous description of all checks performed on estimate sets is given in Appendix~\ref{ap:compl}.

Given a set of at most $64$ elements, any subset of it can be interpreted as a $64$-bit integer. Set operations on such subsets correspond to bitwise operations on standard integer types and can thus be computed very efficiently. In our search, we only considered cases where the number of symbols is less than $64$, so we used such operations when working with estimate sets.

\subsection{Canonicity}\label{ssec:canon}

Let $T$ be either a partially filled array with $i$ cells filled, or a complete near triple array, and let $G$ be the set of isotopisms acting on $T$, that is, $G = G_{v,r,c}^i$ if $T$ is a partial object with $i$ cells filled, or $G = G_{v,r,c}$ if $T$ is a completed near triple array. Recall that an isotopism $\varphi \in G$ consists of permutations $\pi_r$, $\pi_c$ and $\pi_v$ of rows, columns and symbols, respectively. If $T$ is a partial object, $\varphi$ maps the set of first $i$ cells of~\eqref{eq:order} onto itself.
The basic approach to determine whether $T$ is canonical is performing a brute-force search for an isotopism $\varphi \in G$ with $\varphi(T)$ lexicographically smaller than $T$.

We can, however, do slightly better than a pure brute-force search that checks every combination of $\pi_r$, $\pi_c$ and $\pi_v$. Note that given $\pi_r$ and $\pi_c$, it is easy to find the unique $\pi_v$ which lexicographically minimizes $\varphi(T)$, by observing that $\pi_v$ should reorder symbols so that they appear for the first time in increasing order when looking at cells of $\varphi(T)$ in order~\eqref{eq:order}.
Thus it is sufficient to generate permutations $\pi_r$ and $\pi_c$ and determine the optimal $\pi_v$ from them.

Additionally, if for partially generated permutations $\pi_r$ and $\pi_c$ we can already determine that $T$ is lexicographically larger than $\varphi(T)$, then $T$ is not canonical and can be excluded from further search.
Similarly, if we can already see that $\varphi(T)$ is lexicographically larger than $T$, we can exclude the current partially generated $\pi_r$ and $\pi_c$ from further consideration.
As a further optimization during the canonicity checking procedure, we bunch some sets of partial permutations of columns together when we can apply the same arguments to all of them.
For example, if for some subset of columns all symbols appearing in these columns appear in $T$ only once, we can bunch together all permutations $\pi_c$ mapping this subset onto itself.

As a byproduct of the canonicity checking procedure, for a completed near triple array $T$ we determine the size of its autotopism group.

Since each isotopism class has a unique canonical representative, the canonicity check can be seen as eliminating isotopic mates of the (partial) objects produced.
A very common way of performing such an elimination is by translating the objects to an equivalent set of graphs and using a package such as nauty.
We have not employed this method, because in such packages it is usually hard to control which particular member of the automorphism class would be chosen as canonical, and our generation procedure relies heavily on our particular definition of the canonical representative.

\section{Enumeration results}\label{sec:res}

\subsection{Near triple arrays}\label{ssec:res_nta}

In Tables~\ref{tbl:nta3} through~\ref{tbl:nta67} in Appendix~\ref{ap:nta_tables}, we present results on enumeration of near triple arrays for $r=3,4, 5, 6, 7$, for as large $c$ as we could manage.
With some exceptions due to size restrictions, the data we generated is available at~\cite{row_col_data}.
For some parameter sets, presented in Table~\ref{tbl:nonta}, our exhaustive search did not find any complete objects, so there are no near triple arrays on these parameters.

\begin{table}[!ht]
\begin{center}
\begin{tabular}{|c|c||l|l|l|l||c||c|c|c|}
\hline
$r \times c$ & $v$ & $e$ & $\lrc$ & $\lrr$ & $\lcc$ & Comment & $\omega_{rc}$ & $\omega_{rr}$ & $\omega_{cc}$ \\
\hline
$3 \times 3$ & $6$ & $1.5$ & $1.66...$ & $1$ & $1$ &  &  $+$ & $+$ & $+$ \\
\hline
$3 \times 4$ & $6$ & $2$ & $2$ & $2$ & $1$ & R\ref{rmrk:false_ptrn} &  $+$ & $-$ & $+$ \\
\hline
$3 \times 5$ & $8$ & $1.875$ & $1.93...$ & $2.33...$ & $0.7$ &  &  $+$ & $-$ & $+$ \\
\hline
$4 \times 4$ & $9$ & $1.77...$ & $1.875$ & $1.16...$ & $1.16...$ &  &  $+$ & $+$ & $+$ \\
\hline
$4 \times 5$ & $7$ & $2.85...$ & $2.9$ & $3.16...$ & $1.9$ &  &  $+$ & $-$ & $+$ \\
\hline
$4 \times 5$ & $10$ & $2$ & $2$ & $1.66...$ & $1$ & R\ref{rmrk:false_ptrn},T\ref{thm:dual_nta-nbg} &  $+$ & $-$ & $+$ \\
\hline
$4 \times 6$ & $8$ & $3$ & $3$ & $4$ & $1.6$ & T\ref{thm:dual_nta-nbg} &  $+$ & $-$ & $+$ \\
\hline
$4 \times 7$ & $9$ & $3.11...$ & $3.14...$ & $5$ & $1.42...$ &  &  $+$ & $-$ & $+$ \\
\hline
$4 \times 10$ & $12$ & $3.33...$ & $3.4$ & $8$ & $1.06...$ &  &  $+$ & $-$ & $+$ \\
\hline
$4 \times 11$ & $13$ & $3.38...$ & $3.45...$ & $9$ & $0.98...$ & T\ref{thm:nonta_column} &  $-$ & $-$ & $+$ \\
\hline
$5 \times 5$ & $8$ & $3.125$ & $3.16$ & $2.7$ & $2.7$ &  &  $+$ & $+$ & $+$ \\
\hline
$5 \times 6$ & $8$ & $3.75$ & $3.8$ & $4.2$ & $2.8$ &  &  $+$ & $+$ & $+$ \\
\hline
$5 \times 6$ & $15$ & $2$ & $2$ & $1.5$ & $1$ & R\ref{rmrk:false_ptrn} &  $+$ & $-$ & $+$ \\
\hline
$5 \times 7$ & $9$ & $3.88...$ & $3.91...$ & $5.1$ & $2.42...$ &  &  $+$ & $-$ & $+$ \\
\hline
$5 \times 8$ & $10$ & $4$ & $4$ & $6$ & $2.14...$ & T\ref{thm:dual_nta-nbg} &  $+$ & $-$ & $+$ \\
\hline
$5 \times 8$ & $11$ & $3.63...$ & $3.7$ & $5.4$ & $1.92...$ &  &  $+$ & $-$ & $+$ \\
\hline
$5 \times 9$ & $11$ & $4.09...$ & $4.11...$ & $7$ & $1.94...$ & T\ref{thm:dual_bd} &  $?$ & $-$ & $+$ \\
\hline
$5 \times 9$ & $12$ & $3.75$ & $3.8$ & $6.3$ & $1.75$ &  &  $+$ & $?$ & $+$ \\
\hline
$5 \times 9$ & $15$ & $3$ & $3$ & $4.5$ & $1.25$ & T\ref{thm:dual_nta-nbg} &  $+$ & $-$ & $+$ \\
\hline
$5 \times 10$ & $12$ & $4.16...$ & $4.2$ & $8$ & $1.77...$ &  &  $+$ & $-$ & $+$ \\
\hline
$5 \times 10$ & $13$ & $3.84...$ & $3.88$ & $7.2$ & $1.6$ &  &  $+$ & $?$ & $+$ \\
\hline
$5 \times 10$ & $17$ & $2.94...$ & $2.96$ & $4.9$ & $1.08...$ &  &  $?$ & $?$ & $+$ \\
\hline
$5 \times 10$ & $18$ & $2.77...$ & $2.84$ & $4.6$ & $1.02...$ &  &  $?$ & $?$ & $+$ \\
\hline
$6 \times 7$ & $10$ & $4.2$ & $4.23...$ & $4.53...$ & $3.23...$ &  &  $+$ & $+$ & $+$ \\
\hline
$6 \times 7$ & $11$ & $3.81...$ & $3.85...$ & $4$ & $2.85...$ &  &  $+$ & $+$ & $-$ \\
\hline
$6 \times 7$ & $14$ & $3$ & $3$ & $2.8$ & $2$ & T\ref{thm:dual_nta-nbg} &  $+$ & $+$ & $+$ \\
\hline
$6 \times 7$ & $21$ & $2$ & $2$ & $1.4$ & $1$ & R\ref{rmrk:false_ptrn} &  $+$ & $+$ & $+$ \\
\hline
$6 \times 8$ & $10$ & $4.8$ & $4.83...$ & $6.13...$ & $3.28...$ &  &  $+$ & $+$ & $+$ \\
\hline
$6 \times 8$ & $11$ & $4.36...$ & $4.41...$ & $5.46...$ & $2.92...$ &  &  $+$ & $?$ & $+$ \\
\hline
$6 \times 8$ & $12$ & $4$ & $4$ & $4.8$ & $2.57...$ & T\ref{thm:dual_nta-nbg} &  $+$ & $?$ & $?$ \\
\hline
$7 \times 7$ & $10$ & $4.9$ & $4.91...$ & $4.57...$ & $4.57...$ &  &  $+$ & $?$ & $?$ \\
\hline
$7 \times 7$ & $11$ & $4.45...$ & $4.51...$ & $4.09...$ & $4.09...$ &  &  $+$ & $+$ & $+$ \\
\hline
$7 \times 7$ & $12$ & $4.08...$ & $4.10...$ & $3.61...$ & $3.61...$ &  &  $+$ & $?$ & $?$ \\
\hline
\end{tabular}
\caption{Parameters for which no near triple arrays were found.
The cell in the column $\omega_*$ contains $+$ if we found an example of a generalized triple array on corresponding parameters with $\omega_* = 3$ and the other two from $\omega_{rc}, \omega_{rr}, \omega_{cc}$ equal to $2$.
If that is not the case, the cell contains $-$ if we confirmed that there are no such arrays via complete enumeration, and $?$ otherwise.}
\label{tbl:nonta}
\end{center}
\end{table}

\begin{remark}\label{rmrk:false_ptrn}
As can be seen in Table~\ref{tbl:nonta}, there are no near triple arrays for parameters $(3 \times 4, 6)$, $(4 \times 5, 10)$, $(5 \times 6, 15)$ and $(6 \times 7, 21)$, which leads to the suspicion that perhaps there are no $(r \times (r + 1), \binom{r + 1}{2})$-near triple arrays for every $r \geq 3$.
This turns out to be false, and we give examples of such arrays for $r = 7, 8$ in Figure~\ref{fig:false-ptrn}.
\end{remark}

\begin{figure}[!ht]
\begin{center}
    \begin{tabular}{c c c}
	\begin{tabular}{|*{8}{c}|}
    \hline
    0 & 1 & 2 & 3 & 4 & 5 & 6 & 7\\
    1 & 8 & 3 & 9 & 10 & 11 & 12 & 13\\
    9 & 4 & 14 & 15 & 16 & 17 & 18 & 19\\
    17 & 20 & 16 & 21 & 22 & 2 & 13 & 23\\
    23 & 24 & 25 & 26 & 6 & 8 & 15 & 10\\
    25 & 19 & 12 & 22 & 27 & 26 & 20 & 5\\
    27 & 21 & 24 & 7 & 11 & 18 & 0 & 14\\
    \hline
    \end{tabular}
	& &
    \begin{tabular}{|*{9}{c}|}
    \hline
    0 & 1 & 2 & 3 & 4 & 5 & 6 & 7 & 8\\
    1 & 2 & 9 & 10 & 11 & 12 & 13 & 14 & 15\\
    10 & 16 & 4 & 5 & 13 & 17 & 18 & 19 & 20\\
    17 & 21 & 22 & 23 & 24 & 6 & 25 & 11 & 26\\
    26 & 27 & 28 & 21 & 29 & 30 & 9 & 18 & 7\\
    28 & 31 & 23 & 32 & 30 & 14 & 8 & 33 & 16\\
    33 & 25 & 20 & 29 & 31 & 34 & 35 & 3 & 12\\
    35 & 19 & 34 & 15 & 0 & 27 & 32 & 22 & 24\\
    \hline
    \end{tabular}
    \end{tabular}
	\end{center}
    \caption{Examples of a $(7 \times 8, 28)$ and a $(8 \times 9, 36)$-near triple array.}
	\label{fig:false-ptrn}
\end{figure}

We discuss the question of non-existence in greater detail in Sections~\ref{sec:comp} and~\ref{sec:near_balance}.

When there are no $(r \times c, v)$-near triple arrays, a natural question is what can one get that is the closest thing to a near triple array.
To make this question more precise, we make the following definition.

\begin{definition}\label{def:GNTA}
A binary (near) equireplicate $r \times c$ row-column design on $v$ symbols in which, for some integers $x_{rc}$, $x_{rr}$, $x_{cc}$, $\omega_{rc}$, $\omega_{rr}$ and $\omega_{cc}$,
\begin{enumerate}
    \item any row and column have at least $x_{rc}$ and at most $x_{rc} + \omega_{rc}-1$ common symbols,
    \item any two rows have at least $x_{rr}$ and at most $x_{rr} + \omega_{rr}-1$ common symbols,
    \item any two columns have at least $x_{cc}$ and at most $x_{cc} + \omega_{cc}-1$ common symbols,
\end{enumerate}
is called an $(r \times c, v; \omega_{rc}, \omega_{rr}, \omega_{cc})$-generalized triple array.
\end{definition}

Note that the number of possible intersection sizes for pairs of rows and columns is then $\omega_{rc}$ (similarly for $\omega_{rr}$ and $\omega_{cc}$), and that $(r \times c, v)$-triple arrays are $(r \times c, v; 1,1,1)$-generalized triple arrays.
By Proposition~\ref{prop:altdefNTA}, $(r \times c, v)$-near triple arrays are $(r \times c, v; 2,2,2)$-generalized triple arrays.
In the cases where we could find no near triple arrays, we ran a modified version of the code to look for generalized triple arrays.
In every case considered, we were able to find examples of $(r \times c, v; \omega_{rc}, \omega_{rr}, \omega_{cc})$-generalized triple arrays with just one of $\omega_{rc}, \omega_{rr}, \omega_{cc}$ relaxed to $3$, and the other two equal to $2$.
A summary of these cases is given in the last three columns of Table~\ref{tbl:nonta}.

As a concrete example, consider the parameter set $(4 \times 11, 13)$, where $\lambda_{cc} < 1$.
Below we prove Theorem~\ref{thm:nonta_column}, which implies that there are no $(4 \times 11, 13)$-near triple arrays.
The proof of the theorem in fact shows something stronger, namely, that there is no binary near equireplicate $4 \times 11$ row-column design with any pair of columns having at most one symbol in common.
It follows that any $(4 \times 11, 13; \omega_{rc}, \omega_{rr}, \omega_{cc})$-generalized triple array must have $\omega_{cc} \geq 3$, and we give an example with $\omega_{rc} = \omega_{rr} = 2$ and $\omega_{cc} = 3$ in Figure~\ref{fig:GNTA41113}.

\begin{figure}[!ht]
\begin{center}
    \begin{tabular}{|*{11}{c}|}
    \hline
    0 & 1 & 2 & 3 & 4 & 5 & 6 & 7 & 8 & 9 & 10\\
    1 & 0 & 3 & 4 & 2 & 6 & 7 & 5 & 9 & 11 & 12\\
    2 & 3 & 0 & 5 & 6 & 1 & 8 & 11 & 12 & 10 & 4\\
    5 & 6 & 7 & 8 & 11 & 9 & 12 & 10 & 2 & 3 & 0\\
    \hline
    \end{tabular}
    \end{center}
    \caption{A $(4 \times 11, 13; 2,2,3)$-generalized triple array.}
	\label{fig:GNTA41113}
\end{figure}

In Figure~\ref{fig:GNTA479}, we give three further examples of generalized triple arrays for parameters where there exist no near triple arrays.
The first two examples illustrate that the range of permitted values of the intersection sizes, which of course must contain the average value, may in principle be skewed to either side of the average.
The third example shows that, perhaps counter-intuitively, sometimes one has to skew the center of the range of permitted values away from the average: for the parameter set $(4 \times 6, 8)$, the average $\lcc = 1.6$ is closer to $(1 + 3) / 2 = 2$ than to $(0 + 2) / 2 = 1$.
However, there are $(4 \times 6, 8; 2, 2, 3)$-generalized triple arrays with two columns sharing 0, 1, or 2 symbols but no such arrays with columns sharing 1, 2, or 3 symbols.

\begin{figure}[!ht]
\begin{center}
    \begin{tabular}{c c c c c}
    \begin{tabular}{|*{7}{c}|}
    \hline
    0 & 1 & 2 & 3 & 4 & 5 & 6\\
    1 & 0 & 3 & 2 & 5 & 7 & 8\\
    2 & 4 & 0 & 7 & 8 & 6 & 5\\
    8 & 7 & 6 & 4 & 3 & 0 & 1\\
    \hline
    \end{tabular}
    & &
    \begin{tabular}{|*{7}{c}|}
    \hline
    0 & 1 & 2 & 3 & 4 & 5 & 6\\
    1 & 2 & 0 & 4 & 5 & 7 & 8\\
    2 & 6 & 3 & 7 & 8 & 0 & 4\\
    3 & 7 & 8 & 5 & 2 & 6 & 1\\
    \hline
    \end{tabular}
    & &
    \begin{tabular}{|*{6}{c}|}
    \hline
    0 & 1 & 2 & 3 & 4 & 5\\
    1 & 0 & 3 & 2 & 6 & 7\\
    4 & 6 & 5 & 7 & 0 & 2\\
    7 & 5 & 6 & 4 & 3 & 1\\
    \hline
    \end{tabular}
    \end{tabular}
	\end{center}
    \caption{Two $(4 \times 7, 9; 2,2,3)$-generalized triple arrays with $\lambda_{cc} = 1.42...$, with two columns sharing 0, 1 or 2 symbols in the first, and 1, 2 or 3 symbols in the second array, and a $(4 \times 6, 8; 2, 2, 3)$-generalized triple array with $\lambda_{cc} = 1.6$ and two columns sharing 0, 1, or 2 symbols.}
	\label{fig:GNTA479}
\end{figure}

\subsection{Triple arrays}\label{ssec:res_ta}

Recall that $(r \times c, v)$-near triple arrays with all three $\lrc$, $\lrr$ and $\lcc$ being integer are precisely $(r \times c, v)$-triple arrays.
Since the intersection conditions imposed on an array in this special case are much more restrictive, our search works more efficiently.
Thanks to this, we have been able to complete the enumeration of triple arrays for some parameter sets beyond the range where we enumerated near triple arrays in general, namely, for $(7 \times 8, 14)$, $(6 \times 10, 15)$ and $(5 \times 16, 20)$.
We note that this is the first classification of triple arrays up to isotopism for these three parameters sets.
The generated arrays are available at~\cite{row_col_data}.

Regarding smaller parameter sets, it seems to be folklore knowledge that $(3 \times 4, 6)$-triple arrays do not exist, and that the $(4 \times 9, 12)$-triple array is unique up to isotopism.
Additionally, $(5 \times 6, 10)$-triple arrays have previously been enumerated by Phillips, Preece and Wallis in~\cite{phillipsSevenClassesTriple2005}, and, more recently, triple arrays for all three parameter sets mentioned were enumerated by J\"ager, Markstr\"om, Shcherbak and \"Ohman~\cite{jagerEnumerationRowColumnDesigns2024}.
Our results for these parameters also match what was previously known.
We present all results on triple arrays together in Table~\ref{tbl:triple_enum}.

The most symmetric example in this range is the $(6 \times 10, 15)$-triple array with autotopism group size $720$. This design was described by Nilson in \cite{Nilson_intercalates}, and it has the exceptional property that any pair of occurrences of a symbol lies in an \textit{intercalate}, that is, a $2 \times 2$ Latin subsquare.
Its autotopism group is isomorphic to the symmetric group $S_6$, and it acts on the design by permuting rows, that is, for every row permutation there is precisely one autotopism which permutes rows in this way.
We give two further rather symmetric examples in Figures~\ref{fig:ex120} and~\ref{fig:Fano}, both of which have similar properties.

\begin{figure}[!ht]
\begin{center}
    \begin{tabular}{|*{10}{c}|}
    \hline
    0 & 1 & 2 & 3 & 4 & 5 & 6 & 7 & 8 & 9\\
    1 & 0 & 3 & 4 & 5 & 2 & 10 & 11 & 12 & 13\\
    6 & 11 & 0 & 8 & 14 & 13 & 9 & 12 & 3 & 2\\
    8 & 13 & 14 & 11 & 0 & 6 & 5 & 4 & 7 & 10\\
    10 & 9 & 7 & 14 & 12 & 1 & 11 & 2 & 5 & 8\\
    12 & 7 & 10 & 1 & 9 & 14 & 3 & 6 & 13 & 4\\
    \hline
    \end{tabular}
\end{center}
\caption{The $(6 \times 10, 15)$-triple array with autotopism group size 120.}
\label{fig:ex120}
\end{figure}

In the $(6 \times 10, 15)$-triple array with autotopism group size 120, each cell belongs to exactly one intercalate, that is, 15 intercalates form a partition of the table, and the symbols can be grouped into triplets 
$\{0,1,14\},\{2,7,13\},\{3,8,10\},\{4,9,11\},\{5,6,12\}$ such
that all intercalates use only symbols from one triplet, and each pair of
symbols from a triplet forms a unique intercalate.
Each autotopism of the array permutes this set of five triplets, and for every permutation of triplets there is a unique autotopism permuting triplets in this way.
The autotopism group of the array is thus isomorphic to the permutation group $S_5$.

\begin{figure}[!ht]
\begin{center}
\begin{minipage}{0.45\textwidth}
	\begin{tabular}{|*{8}{c}|}
    \hline
    0 & 1 & 2 & 3 & 4 & 5 & 6 & 7\\
    1 & 8 & 3 & 9 & 5 & 10 & 7 & 11\\
    2 & 12 & 13 & 8 & 3 & 1 & 11 & 6\\
    7 & 5 & 9 & 2 & 6 & 12 & 13 & 10\\
    8 & 0 & 5 & 7 & 13 & 11 & 4 & 12\\
    10 & 6 & 11 & 4 & 8 & 2 & 9 & 0\\
    13 & 9 & 0 & 12 & 10 & 4 & 1 & 3\\
    \hline
    \end{tabular}
\end{minipage}
\begin{minipage}{0.45\textwidth}
    \begin{tikzpicture}
    \tikzset{every node/.style={draw,circle,thick,minimum size=1.25cm,fill=white}, every edge/.style={draw,thick}}
    \draw[thick] (0,0) circle (1.5cm);
    \node (v7) at (0,0) {\small $5,7$};
    \node (v1) at (90:3cm) {\small $0,4$};
    \node (v2) at (210:3cm) {\small $8,11$};
    \node (v4) at (330:3cm) {\small $12,13$};
    \node (v3) at (150:1.5cm) {\small $1,3$};
    \node (v6) at (270:1.5cm) {\small $9,10$};
    \node (v5) at (30:1.5cm) {\small $2,6$};
    \draw (v1) edge (v3)
               edge (v5)
               edge (v7)
          (v2) edge (v6)
               edge (v7)
               edge (v3)
          (v3) edge (v7)
          (v4) edge (v5)
               edge (v6)
               edge (v7)
          (v5) edge (v7)
          (v6) edge (v7);
    \end{tikzpicture}
\end{minipage}
\end{center}
\caption{The $(7 \times 8, 14)$-triple array with autotopism group size 168 and the corresponding Fano plane.}
\label{fig:Fano}
\end{figure}
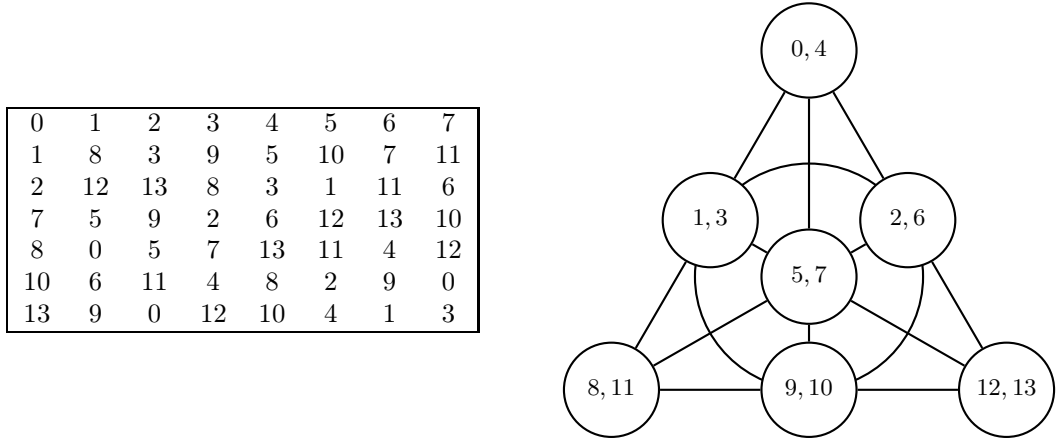

In the $(7 \times 8, 14)$-triple array with autotopism group size 168, any two occurrences of a single symbol lie in a $3 \times 2$ Latin subrectangle, and any pair of distinct symbols that appear together in some column likewise appear in one common $3 \times 2$ Latin subrectangle. 
That is, every pair of symbols except $(0,4),(1,3),(2,6),(5,7),(8,11),(9,10),(12,13)$ appear together in a $3\times 2$ Latin rectangle.
The autotopism group of the array is isomorphic to the automorphism group of the Fano plane, $\mathrm{PSL}(3, 2)$.
The action of the autotopism group on the array can be described using the labelled Fano plane given in Figure~\ref{fig:Fano} on the right.
The vertex labels are the seven pairs of symbols mentioned above, and each of the seven edges (six lines and the circle in the illustration) corresponds to four of the $3 \times 2$ Latin subrectangles using one symbol from each vertex in the edge.
The group isomorphism between the autotopism group of the array and $\mathrm{PSL}(3, 2)$ is given by the action of autotopisms on the pairs of symbols corresponding to the labels on the vertices of the Fano plane.

\subsection{Near Youden rectangles}\label{ssec:res_nyr}

Recall that for $r \leq v$, $(r \times v, v)$-near triple arrays are precisely $r \times v$ (near) Youden rectangles.
In~\cite{jagerSmallYoudenRectangles2023}, J\"ager, Markstr\"om, Shcherbak and \"Ohman enumerated (near) Youden rectangles for a variety of small parameter sets.
We replicated these results, including counts for specific autotopism group sizes, in all cases considered in~\cite{jagerSmallYoudenRectangles2023} except $5 \times 21$ Youden rectangles, which was too computationally demanding for our program.
In addition, we have enumerated $3 \times 14$, $3 \times 15$, $4 \times 14$ and $5 \times 12$ near Youden rectangles.
We present these new results in Table~\ref{tbl:nyr_enum}.
With some exceptions due to size restrictions, the data we generated is available at~\cite{near_youden_data}.

\subsection{Other row-column designs}\label{ssec:res_int}

Consider a binary equireplicate $r \times c$ row-column design on $v$ symbols and the following intersection conditions, where $\lrc$, $\lrr$ and $\lcc$ are from Lemma~\ref{lm:equiavg}:
\begin{enumerate}
    \customitem{(RC)}\label{RCD:rc} any row and column have $\lrc$ common symbols,
	\customitem{(RR)}\label{RCD:rr} any two rows have $\lrr$ common symbols,
	\customitem{(CC)}\label{RCD:cc} any two columns have $\lcc$ common symbols.
\end{enumerate}
Following the nomenclature in~\cite{jagerEnumerationRowColumnDesigns2024}, a binary equireplicate row-column design is called
\begin{itemize}
    \item a triple array if all three conditions~\ref{RCD:rc},~\ref{RCD:rr},~\ref{RCD:cc} hold;
    \item a \textit{double array} if~\ref{RCD:rr} and~\ref{RCD:cc} hold;
    \item a \textit{sesqui array} if~\ref{RCD:rc} and~\ref{RCD:rr} hold;
    \item a \textit{transposed mono array} if~\ref{RCD:rr} holds;
    \item an \textit{adjusted orthogonal array}, or an \textit{AO-array}, if~\ref{RCD:rc} holds.
\end{itemize}
Double arrays have historically been studied alongside triple arrays, for example, see~\cite{mcsorleyDoubleArraysTriple2005a}.
Sesqui arrays were introduced by Bailey, Cameron and Nilson~\cite{baileySesquiarraysGeneralisationTriple2018}, and
transposed mono arrays and AO-arrays were introduced more recently in~\cite{jagerEnumerationRowColumnDesigns2024}.

Clearly, any double array is also a transposed mono array, any sesqui array is also a transposed mono array and an AO-array, and any triple array belongs to all the other four classes of designs.
In each case the design is called \textit{proper} if those conditions~\ref{RCD:rc},~\ref{RCD:rr},~\ref{RCD:cc} not in the definition of the design are explicitly forbidden to hold.
For example, a transposed mono array is proper if it does not satisfy conditions~\ref{RCD:rc} and~\ref{RCD:cc}, or, equivalently, if it is not a sesqui array and not a double array.

Conditions~\ref{RCD:rc},~\ref{RCD:rr},~\ref{RCD:cc} can clearly only hold when, respectively, $\lrc$, $\lrr$, $\lcc$ are integers.
These conditions, together with the property of being equireplicate, restrict which parameters $(r \times c, v)$ are \textit{admissible} for all five mentioned classes of designs.
We note that any $(r \times c, v)$-near triple array on a parameter set that is admissible for one of these classes will actually belong to the class.

J\"ager, Markstr\"om, Shcherbak and \"Ohman~\cite{jagerEnumerationRowColumnDesigns2024} enumerated these types of designs for small parameter sets.
We independently replicated their results using a modified version of our program.
Note that our approach to enumeration is very different from theirs.
We further attempted to enumerate these types of designs in the same range of parameters we considered for near triple arrays and obtained complete enumeration in several new cases.
With some exceptions due to size restrictions, the data we generated is available at~\cite{row_col_data}.
We present these results in Table~\ref{tbl:RCDupto14} and Table~\ref{tbl:RCDfrom15}.
Table~\ref{tbl:RCDupto14} contains results for all non-trivial admissible parameter sets for binary equireplicate row-column designs on $v \leq 14$ symbols.
Table~\ref{tbl:RCDfrom15} contains results for some parameter sets with $v > 15$ which we previously considered for near triple arrays.
Note that the counts are given for proper designs of each type.
For comparison, we also give counts of near triple arrays for the same parameter sets.
Notably, in all cases considered, except of course for those corresponding to triple arrays, near triple arrays form a proper subset of every other class.
In several cases listed in Table~\ref{tbl:RCDfrom15}, we established existence of proper designs even though the complete enumeration has not been finished.
In each such case, this was done by the pigeonhole principle: for example, even though the search for $4 \times 12$ transposed mono arrays on $24$ symbols has not been completed, the number of non-isotopic arrays found already exceeded the total number of $4 \times 12$ sesqui arrays on $24$ symbols.

\section{Constructions}\label{sec:exist}

Throughout this section, we use the identities for parameters $\lrc$, $\lrr$, $\lcc$ from Lemma~\ref{lm:avg} without referring to them explicitly.

\subsection{Constructions with at most one common symbol between two columns}\label{ssec:ex_lcc}

We start with a lemma relating large $v$ to small $\lrc$ and $\lcc$.

\begin{lemma}\label{lm:smalllambda}
    Let $e, \lrc, \lrr, \lcc$ be the parameters of an $(r \times c, v)$-near triple array.
    Then the following are equivalent:
\begin{enumerate}[(a)]
    \item\label{sml:param} $e \leq 2$, $\lrc \leq 2$ and $\lcc \leq 1$;
    \item\label{sml:v} $v \geq rc - c \cdot \min(r, c - 1) / 2$.
\end{enumerate}
\end{lemma}

\begin{proof}
If $c > r$, then \ref{sml:v} turns into $v \geq rc / 2$, which is equivalent to $e\leq 2$.
This in turn implies the other inequalities of \ref{sml:param}: if $e \leq 2$, then $\lrc \leq 2$ and $\lcc = r(\lrc - 1) / (c - 1) \leq r / (c - 1) \leq 1$.

If $c \leq r$, then $\lcc = r(\lrc - 1) / (c - 1) \leq 1$ implies $\lrc \leq 1 + (c - 1) / r < 2$ and thus also $e < 2$.
Both items of the lemma hold trivially when $e = 1$, so we may assume $e^- = 1$ and $e^+ = 2$.
Then
\[
\lrc = e^- + e^+ - \frac{e^-e^+}{e} = 3 - \frac{2v}{rc}, \text{ so } \lcc = \frac{r(\lrc - 1)}{c - 1} = \frac{r}{c - 1} \left( 2 - \frac{2v}{rc} \right) =  \frac{2(rc - v)}{c(c - 1)}.
\]
Then $\lcc \leq 1$ is equivalent to $v \geq rc - c(c - 1) / 2$, i.e. to \ref{sml:v}.
\end{proof}

Given a near triple array satisfying the conditions of Lemma~\ref{lm:smalllambda}, we can construct new near triple arrays from it by either replacing some of the repeating symbols with new ones, or by adding a new column.

\begin{lemma}\label{lm:incv}
Let $T$ be an $(r \times c, v)$-near triple array with $v \geq rc - c \cdot \min(r, c - 1) / 2$.
For each $0 < i \leq rc - v$, there exists an $(r \times c, v + i)$-near triple array.
\end{lemma}
\begin{proof}
Let $e, \lrc, \lrr, \lcc$ be the parameters of $T$.
By Lemma~\ref{lm:smalllambda}, we have $e, \lrc \leq 2$ and $\lcc \leq 1$.
Choose two rows of $T$ that have $\lrr^+$ symbols in common, and let $a$ be one of their common symbols.
Since $e \leq 2$, $a$ doesn't appear in other rows of $T$.
Let $T'$ be an $r \times c$ array on $v + 1$ symbols obtained from $T$ by replacing one of the occurrences of $a$ with a new unique symbol.
Applying Proposition~\ref{prop:altdefNTA} with $x, x_{rc} = 1$, $x_{rr} = \lrr^+ - 1$ and $x_{cc} = 0$, we see that $T'$ is an $(r \times c, v + 1)$-near triple array.
Repeating this operation $i$ times, we get an $(r \times c, v + i)$-near triple array.
\end{proof}

\begin{lemma}\label{lm:addclm}
Let $T$ be an $(r \times c, v)$-near triple array with $v \geq rc - c \cdot \min(r, c - 1) / 2$.
Let $T'$ be an $r \times (c + 1)$ array on $v + r$ symbols obtained from $T$ by adding a new column with $r$ new symbols in it.
Then $T'$ is an $(r \times (c + 1), v + r)$-near triple array.
\end{lemma}
\begin{proof}
Let $e, \lrc, \lrr, \lcc$ be the parameters of $T$.
Due to Lemma~\ref{lm:smalllambda}, we have $e, \lrc \leq 2$ and $\lcc \leq 1$.
Applying Proposition~\ref{prop:altdefNTA} with $x, x_{rc} = 1$, $x_{rr} = \lrr^-$ and $x_{cc} = 0$, we see that $T'$ is a $(r \times (c + 1), v + r)$-near triple array.
\end{proof}

In the tables of Appendix~\ref{ap:nta_tables}, columns share more and more values at the bottom, the further to the right one looks.
For example, the last 8 numbers in the columns corresponding to $3 \times 14$ and $3 \times 15$ near triple arrays in Table~\ref{tbl:nta3} coincide.
In the next corollary we provide an explanation for this phenomenon.

\begin{corollary}
Let $\mathcal{N}(r \times c, v)$ be the number of isotopism classes of $(r \times c, v)$-near triple arrays.
\begin{enumerate}[(a)]
\item\label{tail:ineq} If $v \geq rc - c \cdot \min(r, c - 1) / 2$, then $\mathcal{N}(r \times c, v) \leq \mathcal{N}(r \times (c + 1), v + r)$.
\item\label{tail:eq} If $v \geq rc - c / 2$, then $\mathcal{N}(r \times c, v) = \mathcal{N}(r \times (c + 1), v + r)$.
\end{enumerate}
\end{corollary}

\begin{proof}
The operation of adding a column from Lemma~\ref{lm:addclm} preserves isotopism classes, so \ref{tail:ineq} holds.

If $v \geq rc - c / 2$, then any $(r \times (c + 1), v + r)$-near triple array $T'$ has at most
\[
    2 \cdot ( r(c + 1) - (v + r) ) \leq c < c + 1
\]
cells filled with symbols appearing more than once in the array, so, by the pigeonhole principle, $T'$ contains a column filled with unique symbols.
By deleting this column, we get an $(r \times c, v)$-near triple array, and since this operation also preserves isotopism classes, \ref{tail:eq} follows.
\end{proof}

For $v \geq rc-c/2$, we can show existence of near triple arrays by
an explicit construction, given in the proof of the following lemma.
The basic idea is that there are so few repeated symbols that the only complication is to distribute them in such a way that row-row intersections are balanced.

\begin{lemma}\label{lm:ex_tail}
There exists an $(r \times c, v)$-near triple array whenever $v \geq rc - c / 2$.
\end{lemma}

\begin{proof}
An $(r \times c, v)$-near triple array $T$ must have $t := rc - v \leq c / 2$ symbols occurring two times, and all other symbols occurring just once.
Let
\[
\left\{ P_1, \dots, P_{\binom{r}{2}} \right\} = \binom{[r]}{2} = \{ (i, j)\ :\ 1 \leq i < j \leq r\};\quad P_i := P_{i - \binom{r}{2}} \text{ for } i > \binom{r}{2}.
\]
For $1 \leq i \leq t$ put symbol $i$ into the two cells $(j_1, 2i - 1)$ and 
$(j_2, 2i)$ of $T$, where $(j_1, j_2) = P_i$, and put the remaining unique symbols into the remaining cells in an arbitrary way.
The resulting array is clearly binary, so by Proposition~\ref{prop:altdefNTA} applied with $x, x_{rc} = 1$, $x_{rr} = \left\lfloor\frac{2t}{r(r - 1)}\right\rfloor$ and $x_{cc} = 0$, $T$ is a near triple array.
\end{proof}

Another type of construction of near triple arrays is juxtaposing
two near triple arrays with suitable parameters, as in the following lemma.

\begin{lemma}\label{lm:concat}
    Let $T_1$ and $T_2$ be an $(r \times c_1, v_1)$ and an $(r \times c_2, v_2)$-near triple arrays on disjoint sets of symbols with parameters $e_1, \lrr^1, \lcc^1, \lrc^1$ and $e_2, \lrr^2, \lcc^2, \lrc^2$, respectively.
    Let
    \begin{enumerate}
        \item $\max(e_1^+, e_2^+) - \min(e_1^-, e_2^-) \leq 1$,
        \item $\lrr^1 \in \Z$ or $\lrr^2 \in \Z$,
        \item $\lcc^1, \lcc^2 \leq 1$,
    \end{enumerate}
    and let $T$ be a concatenation of $T_1$ and $T_2$ (each row of $T$ is a concatenation of two corresponding rows of $T_1$ and $T_2$).
    Then $T$ is an $(r \times (c_1 + c_2), v_1 + v_2)$-near triple array.
\end{lemma}
\begin{proof}
Without loss of generality, we may assume that $\lrr^1 \in \Z$.
Each symbol appears in $T$ either $\min(e_1^-, e_2^-)$ or $\max(e_1^+, e_2^+)$ times, and, due to Lemma~\ref{lm:avg}~\ref{avg:lrc+-}, any row and column of $T$ have $\min(e_1^-, e_2^-)$ or $\max(e_1^+, e_2^+)$ symbols in common.
Each pair of rows of $T$ has $\lrr^1 + (\lrr^2)^-$ or $\lrr^1 + (\lrr^2)^+$ symbols in common, and each pair of columns of $T$ has $0$ or $1$ symbols in common.
Then, by Proposition~\ref{prop:altdefNTA}, $T$ is an $(r \times (c_1 + c_2), v_1 + v_2)$-near triple array.
\end{proof}

\subsection{Connections with (near) Youden rectangles and Latin squares}\label{ssec:ex_nyr_latin}

In all cases we considered for enumeration (see tables in Appendix~\ref{ap:nta_tables}), we were always able to find examples of $(r \times n, n)$ and $(r \times (n - 1), n)$-near triple arrays with $r \leq n$.
This is not true in general: $(r \times n, n)$-near triple arrays are $r \times n$ (near) Youden rectangles, and it follows from results by Brown~\cite{brown} that there are no $6 \times 17$ near Youden rectangles.
Nevertheless, in the next lemma we prove that the existence of $(r \times n, n)$-near triple arrays implies the existence of $(r \times (n - 1), n)$-near triple arrays.
As can be seen in Table~\ref{tbl:nonta}, there are no $(r \times (n - 2), n)$-near triple arrays for quite a few values of $n$ and $r$, so there is no hope of extending this construction one step further.

\begin{lemma}\label{lm:from_Youden}
    For $r \leq n$, let $T$ be an $(r \times n, n)$-near triple array, that is, a (near) Youden rectangle, and let $T'$ be a row-column design obtained from $T$ by deleting any column.
    Then $T'$ is an $(r \times (n - 1), n)$-near triple array.
\end{lemma}
\begin{proof}
Every symbol appears in $T'$ either $r$ or $r - 1$ times.
Each row of $T'$ contains all symbols except one, so it has either $r - 1$ or $r$ common symbols with any remaining column.
Moreover, since the deleted column contained $r$ different symbols, the missing symbol is different for each row, thus any pair of rows has $n - 2$ symbols in common.
The number of common symbols between any pair of columns of $T'$ is one of two consecutive integers because this holds in $T$.
By Proposition~\ref{prop:altdefNTA}, $T'$ is an $(r \times (n - 1), n)$-near triple array.
\end{proof}

It is well known that removing a single row from an $n \times n$ Latin square will produce an $(n - 1) \times n$ Youden rectangle.
We will now show that starting from a suitably chosen Latin square, it is always possible to remove two rows to get an $(n - 2) \times n$ near Youden rectangle.
Combining this with Lemma~\ref{lm:from_Youden}, we get a few more explicit constructions of near triple arrays.

\begin{lemma}\label{lm:from_latin}
    For any $n \geq 4$ and $k \leq 2$ there exist $((n - k) \times n, n)$ and $((n - k) \times (n - 1), n)$-near triple arrays.
\end{lemma}

\begin{proof}
Let $T = T_0$ be the Latin square corresponding to the addition table in $\Z_n$, and let $T_1$ and $T_2$ be the Latin rectangles obtained from $T$ by deleting the last row and the last two rows, respectively.
For $k = 0, 1, 2$, $T_k$ is an $((n - k) \times n, n)$-near triple array: it is clear for $k = 0, 1$, and for $k = 2$ it is true because the addition table in $\Z_n$ has no intercalates, so no two columns lost the same two symbols after the deletion, and consequently column intersection sizes are either $n-3$ or $n-4$.
The existence of $((n - k) \times (n - 1), n)$-near triple arrays follows from Lemma~\ref{lm:from_Youden}.
\end{proof}

Note that a $2\times n$ array on $n$ symbols without intercalates is in fact a near triple array. 
Therefore, a wider perspective on Lemma~\ref{lm:from_latin} is that starting with a suitable Latin square, and removing a set of rows forming a $(k \times n, n)$-near triple array, the remaining array (the complement) is an $((n-k) \times n, n)$-near triple array. We formulate this as a lemma.

\begin{lemma}\label{lm:complement}
    Suppose there exists a $(k \times n, n)$-near triple array. Then there exists an $((n - k) \times n, n)$-near triple array and an $((n - k) \times (n-1), n)$-near triple array.
\end{lemma}

\begin{proof}
    Let $T$ be a $(k \times n, n)$-near triple array. Viewing it as a Latin rectangle, it is well-known that $T$ may be completed to a Latin square $L$. We claim that the complement $T'$ of $T$ in $L$, that is, $L$ with the rows from $T$ removed, is an $((n - k) \times n, n)$-near triple array.
    
    The only near triple array property of $T'$ that is not immediate is the column intersection property.
    Suppose that columns $j_1$ and $j_2$ of $T$ share $a$ common symbols, then they contain in total $2k - a$ different symbols.
    The columns $j_1$ and $j_2$ of $T'$ then share precisely $n - 2k + a$ symbols which do not appear in columns $j_1$ and $j_2$ of $T$.
    Since column intersection sizes in $T$ take at most two values, the same holds for $T'$, so $T'$ is a near triple array.
    The existence of an $((n - k) \times (n-1), n)$-near triple array then follows from Lemma~\ref{lm:from_Youden}.
\end{proof}

\subsection{Near triple arrays with 3 rows}\label{ssec:ex_3с}

In Table~\ref{tbl:nta3}, we see that there are no $(3 \times 3, 6)$, $(3 \times 4, 6)$ or $(3 \times 5, 8)$-near triple arrays, but there are examples of $(3 \times c, v)$-near triple arrays in all considered cases with $c \geq 6$.
Existence turns out to hold for all larger $c$ as well.

\begin{theorem}\label{thm:ex_3c}
    There exist $(3 \times c, v)$-near triple arrays for any $c \geq 6$ and $v \geq c$.
\end{theorem}
\begin{proof}
The proof is by induction on $c$, cases with $6 \leq c \leq 13$ form the base of induction; in all corresponding cases examples of near triple arrays were found by computer search (see Table~\ref{tbl:nta3}).

Note that any $(3 \times c, v)$-near triple array with $c \geq 7$ has $\lcc = \frac{3(\lrc - 1)}{c - 1} \leq \frac{6}{c - 1} \leq 1$.

Let $c \geq 14$.
If $v \geq 3c - c / 2$, a $(3 \times c, v)$-near triple array exists by Lemma~\ref{lm:ex_tail}.
Otherwise, if $v \geq c + 4$, then consider near triple arrays with parameters $(3 \times 6, 9)$ and $(3 \times (c - 6), v - 9)$, which exist by the induction hypothesis.
The first one has $e = 2$, $\lrr = 3 \in \Z$ and $\lcc < 1$, the second has either $e^- = 2$ or $e^+ = 2$ and $\lcc \leq 1$.
Thus, due to Lemma~\ref{lm:concat}, there is a $(3 \times c, v)$-near triple array. 

Finally, if $v \leq c + 3$, consider near triple arrays with parameters $(3 \times 7, 7)$ and $(3 \times (c - 7), v - 7)$, which again exist by the induction hypothesis.
The first one has $e = 3$, $\lrr = 7 \in \Z$ and $\lcc = 1$, the second has $e^+ = 3$ and $\lcc \leq 1$, so, by Lemma~\ref{lm:concat}, there is a $(3 \times c, v)$-near triple array.
\end{proof}

\begin{figure}[!ht]
\begin{center}
    \begin{tabular}{c c c}
	\begin{tabular}{|*{6}{c}|}
        \hline
        0 & 1 & 2 & 3 & 4 & 5\\
        1 & 6 & 3 & 7 & 5 & 8\\
        7 & 2 & 8 & 4 & 6 & 0\\
        \hline
    \end{tabular}
	& &
    \begin{tabular}{|*{7}{c}|}
        \hline
        0 & 1 & 2 & 3 & 4 & 5 & 6\\
        1 & 2 & 3 & 4 & 5 & 6 & 0\\
        3 & 4 & 5 & 6 & 0 & 1 & 2\\
        \hline
    \end{tabular}
    \end{tabular}
	\end{center}
\caption{A $(3 \times 6, 9)$-near triple array and a $(3 \times 7, 7)$-near triple array used in the proof of Theorem~\ref{thm:ex_3c}. Both are unique up to isotopism as can be seen in Table~\ref{tbl:nta3}.}
\label{fig:ex_3c}
\end{figure}

\begin{corollary}\label{cor:ex_n-3}
    There exist $((n-3) \times n, n)$-near triple arrays and
    $((n-3) \times (n-1), n)$-near triple arrays
    for any $n \geq 6$.
\end{corollary}
\begin{proof}
    This follows immediately from Theorem~\ref{thm:ex_3c} with
    $c=v=n$ and Lemma~\ref{lm:complement}.
\end{proof}

The existence of $4 \times c$ near triple arrays for large $c$ can potentially be proved by a similar approach, though Table~\ref{tbl:nta4} would need to be extended much further to get an appropriate induction base.

Combining the next lemma with the construction from the proof of Lemma~\ref{lm:incv}, we can construct $(3 \times c, v)$-near triple arrays with any $c \geq 6$ and $v \geq 3c / 2$ more directly.

\begin{lemma}\label{lm:constr_3c}
Let $k \geq 3$.
\begin{enumerate}
\item Let $T$ be a $3 \times 2k$ row-column design on $3k$ symbols, with symbol $(3a + b)$, $0 \leq a < k$, $0 \leq b < 3$, appearing in cells $(b + 1, a + 1)$ and $((b + 1) \bmod 3 + 1, k + (a + b) \bmod 3 + 1)$ of $T$.
Then $T$ is a $(3 \times 2k, 3k)$-near triple array.
\item Let $T'$ be a $3 \times (2k + 1)$ row-column design on $3k + 2$ symbols obtained from $T$ by replacing symbol $0$ in cell $(1, 1)$ with symbol $3k$, and adding a new column with symbols $T'_{1, 2k + 1} = 0$, $T'_{2, 2k + 1} = 3k$ and $T'_{3, 2k + 1} = 3k + 1$.
Then $T'$ is a $(3 \times (2k + 1), 3k + 2)$-near triple array.
\end{enumerate}
\end{lemma}
\begin{proof}
By construction, $T$ is equireplicate, and any two rows of $T$ have $k$ symbols in common.
Additionally, for any 3 symbols appearing in the same column, the second occurrences of these 3 symbols are in pairwise different rows and columns, which implies that any two columns share at most one symbol, and that any row and column share 1 or 2 symbols.
Thus, $T$ is a near triple array by Proposition~\ref{prop:altdefNTA}. This establishes the first claim.

For the second claim, note first that the new column shares 1 symbol with columns 1 and $k + 1$, no symbols with other columns, 2 symbols (i.e. symbols $0$ and $3k$) with the first two rows and 1 symbol with row 3.
Finally, the first two rows have $k + 1$ symbols in common, and other pairs of rows still have $k$ common symbols.
Again, due to Proposition~\ref{prop:altdefNTA}, $T'$ is a near triple array. \qedhere
\end{proof}

\begin{figure}[!ht]
\begin{center}
    \begin{tabular}{c c c}
	\begin{tabular}{|cccc|cccc|}
        \hline
        0 & 3 & 6 & 9 & 8 & 11 & 2 & 5\\
        1 & 4 & 7 & 10 & 0 & 3 & 6 & 9\\
        2 & 5 & 8 & 11 & 10 & 1 & 4 & 7\\
        \hline
    \end{tabular}
	& &
    \begin{tabular}{|cccc|cccc|c|}
        \hline
        12 & 3 & 6 & 9 & 8 & 11 & 2 & 5 & 0\\
        1 & 4 & 7 & 10 & 0 & 3 & 6 & 9 & 12\\
        2 & 5 & 8 & 11 & 10 & 1 & 4 & 7 & 13\\
        \hline
    \end{tabular}
    \end{tabular}
	\end{center}
\caption{Constructions from Lemma~\ref{lm:constr_3c} for $k = 4$.}
\label{fig:construct_3c}
\end{figure}

\section{Component designs}\label{sec:comp}

Let $T$ be an $r \times c$ binary row-column design on $v$ symbols.
For a symbol $x$ of $T$, consider the set of columns of $T$ which contain $x$.
Treating all $v$ such sets as blocks, and the $c$ columns of $T$ as points, we get a block design which we will call the \textit{column design} $CD_T$ of $T$.
Similarly, we define the \textit{row design} $RD_T$ of $T$ with rows of $T$ as points, and blocks corresponding to sets of rows containing a fixed symbol.
We will focus mainly on the column design, but the row design has similar properties as it is the column design of the transposed array.

Now, suppose that $T$ is an $(r \times c, v)$-near triple array with parameters $e, \lrc, \lrr, \lcc$.
Then every point of $CD_T$ occurs in $r$ blocks, each block contains $e^-$ or $e^+$ points, and every pair of points is covered by $\lcc^-$ or $\lcc^+$ blocks.
Then the column design $CD_T$ belongs to a class of designs studied by Bofill and Torras~\cite{MBMUD}, called \textit{maximally balanced maximally uniform designs (MBMUDs)}. MBMUDs are block designs in which the block sizes are either all the same, or each equal to one of two adjacent integers, and the same condition holds for occurrence numbers of points, and for covering numbers of pairs of points. We see then that $CD_T$ is a MBMUD with fixed occurrence number.

When one or both of the parameters $e$ and $\lcc$ of $T$ are integers, $CD_T$ also belongs to other well-studied classes of designs.
Namely, if $e$ is an integer, then $CD_T$ has fixed block sizes and occurrence numbers, and covering numbers $\lcc^-$ or $\lcc^+$.
Such block designs are known as \textit{regular graph designs}, and were first introduced by John and Mitchell~\cite{johnOptimalIncompleteBlock1977}.
If $\lcc$ is an integer, then $CD_T$ has fixed occurrence and covering numbers, and block sizes $e^-$ or $e^+$, and $CD_T$ is a \textit{pairwise balanced design} $(c, \{e^-, e^+\}, \lcc)$-PBD.
For an introduction to pairwise balanced designs, see Part IV of the Handbook of Combinatorial Designs~\cite{colbournHandbookCombinatorialDesigns2007}.
Finally, if both $e$ and $\lcc$ are integers, which happens, for example, when $T$ is a triple array, then each block of $CD_T$ contains $e$ points, and any pair of points is covered by $\lcc$ blocks, so $CD_T$ is a \textit{balanced incomplete block design} $(c, e, \lcc)$-BIBD.

For an integer $n$ and a real $m$, we will use the notation $m^- := \lfloor m \rfloor$, $m^+ := \lceil m \rceil$ and define
\[
S(n, m) := n\binom{m}{2} + \frac{n(m - m^-)(m^+ - m)}{2},
\]
where $\binom{m}{2}$ for a possibly non-integer $m$ is defined as $m(m - 1) / 2$.
Note that, by definition, $S(n, m) \geq n\binom{m}{2}$, with equality precisely when $m$ is integer.
In our further treatment of the column design, we will use the following technical lemma.

\begin{lemma}\label{lm:bin_sum}
Let $\{a_i\}_{i\in I}$ be a collection of $n = |I|$ non-negative integers, and let $m := \frac{1}{n}\sum_{i \in I} a_i$ be the average value of the $a_i$.
Then
\begin{equation}\label{eq:bin_sum}
\sum_{i \in I} \binom{a_i}{2} \geq S(n, m),
\end{equation}
with equality holding if and only if $a_i \in \{m^-, m^+\}$ for all $i \in I$.
\end{lemma}
\begin{proof}
We will first prove that, for fixed $n$ and $m$, the left hand side of~\eqref{eq:bin_sum} is minimized precisely when $a_i \in \{m^-, m^+\}$ for all $i \in I$, and then show that in the latter case the equality holds.

Suppose that $a_i - a_j > 1$ for some $i \neq j$, then the left hand side of~\eqref{eq:bin_sum} would strictly decrease if we replaced $a_i$ and $a_j$ with $a_i - 1$ and $a_j + 1$ since
\[
\binom{a_i - 1}{2} + \binom{a_j + 1}{2} - \binom{a_i}{2} - \binom{a_j}{2} = a_j - (a_i - 1) < 0.
\]
It follows that the sum is minimal if and only if $|a_i - a_j| \leq 1$ for all $i \neq j$, which is equivalent to $a_i \in \{m^-, m^+\}$ for all $i \in I$.

Now, suppose that $a_i \in \{m^-, m^+\}$ for all $i \in I$.
If $m$ is an integer, then all $a_i$ are equal to $m$, and so
\[
\sum_{i \in I} \binom{a_i}{2} = n\binom{m}{2} = S(n, m).
\]
If $m$ is not an integer, then, similar to Lemma~\ref{lm:k-k+}, there are $n(m^+ - m)$ elements equal to $m^-$ and $n(m - m^-)$ elements equal to $m^+$.
Thus
\[
\sum_{i \in I} \binom{a_i}{2} = n(m^+ - m)\binom{m^-}{2} + n(m - m^-)\binom{m^+}{2} = S(n, m).\qedhere
\]
\end{proof}

The next theorem gives a necessary condition for the existence of the column design of a near triple array, and thus for the existence of the near triple array itself.
Note that the theorem can be rephrased as a more general necessary condition for the existence of a block design, without connection to a near triple array.
In particular, it can be viewed as a special case of Theorem 9 from~\cite{MBMUD}.
Alternatively, it can be derived from an even more general identity proven by Tsuji~\cite{tsujiIntersectionNumbers0Designs1994}. Nevertheless, we give a self-contained elementary proof, adapted to this particular situation.

\begin{theorem}\label{thm:dual_bd}
Let $T$ be an $(r \times c, v)$-near triple array with parameters $e, \lrc, \lrr, \lcc$.
Let $\mu_c$ be the average number of times a pair of symbols appears together in a column, that is,
\[
\mu_c := \frac{\binom{r}{2}c}{\binom{v}{2}} = \frac{e(r - 1)}{v - 1}.
\]
Define $\mu_c^- := \lfloor \mu_c \rfloor$, $\mu_c^+ := \lceil \mu_c \rceil$.
Then
\begin{equation}\label{eq:dual_bd}
S( \binom{c}{2}, \lcc ) \geq S( \binom{v}{2}, \mu_c ).
\end{equation}
Moreover, equality holds in~\eqref{eq:dual_bd} if and only if each pair of symbols is covered by $\mu_c^-$ or $\mu_c^+$ columns.
\end{theorem}
\begin{proof}
For $1 \leq i < j \leq c$, let $l_{ij}$ be the number of common symbols between columns $i$ and $j$ of $T$, and for $1 \leq a < b \leq v$, let $m_{ab}$ be the number of columns containing both symbols $a$ and $b$.
Counting in two ways the number of times a pair of symbols appears together in a pair of columns, we get 
\begin{equation}\label{eq:dblcnt_bd}
\sum_{1 \leq i < j \leq c} \binom{l_{ij}}{2} = \sum_{1 \leq a < b \leq v} \binom{m_{ab}}{2}.
\end{equation}
Applying Lemma~\ref{lm:bin_sum} two times to 
both sides of~\ref{eq:dblcnt_bd},
and taking into account that $l_{ij} \in \{\lcc^-, \lcc^+\}$ for all $i < j$ since $T$ is a near triple array, we get
\begin{equation}\label{eq:2_bin_sums}
\sum_{1 \leq i < j \leq c} \binom{l_{ij}}{2} = S(\binom{c}{2}, \lcc)\quad \text{and} \quad \sum_{1 \leq a < b \leq v} \binom{m_{ab}}{2} \geq S( \binom{v}{2}, \mu_c ).
\end{equation}
Combining~\eqref{eq:dblcnt_bd} and~\eqref{eq:2_bin_sums} together gives~\eqref{eq:dual_bd}.
Finally, again by Lemma~\ref{lm:bin_sum}, equality holds in~\ref{eq:dual_bd} if and only if $m_{ab} \in \{\mu_c^-, \mu_c^+\}$ for all $a < b$.
\end{proof}

\begin{table}[!ht]
\begin{center}
\begin{tabular}{|c|c||l|l|l|l|l||r|r||c|}
\hline
\rule{0pt}{10pt}
$r \times c$ & $v$ & $e$ & $\lrc$ & $\lrr$ & $\lcc$ & $\mu_c$ & $S( \binom{c}{2}, \lcc )$ & $S( \binom{v}{2}, \mu_c )$ & Comment \\[2pt]
\hline
$5 \times 9$ & $11$ & $4.09...$ & $4.11...$ & $7$ & $1.94...$ & $1.63...$ & $34$ & $35$ & \ref{nonta_bd_2k-1},\ref{nonta_bd_s}\\
\hline
$6 \times 9$ & $11$ & $4.90...$ & $4.92...$ & $7.06...$ & $2.94...$ & $2.45...$ & $104$ & $105$ & \ref{nonta_bd_2k}\\
\hline
$6 \times 13$ & $16$ & $4.875$ & $4.89...$ & $10.13...$ & $1.94...$ & $1.625$ & $74$ & $75$ & \\
\hline
$6 \times 14$ & $16$ & $5.25$ & $5.28...$ & $12$ & $1.97...$ & $1.75$ & $89$ & $90$ & \ref{nonta_bd_s}\\
\hline
$7 \times 13$ & $15$ & $6.06...$ & $6.07...$ & $11$ & $2.96...$ & $2.6$ & $228$ & $231$ & \ref{nonta_bd_2k-1}\\
\hline
$7 \times 19$ & $22$ & $6.04...$ & $6.05...$ & $16$ & $1.96...$ & $1.72...$ & $165$ & $168$ & \ref{nonta_bd_s}\\
\hline
$7 \times 20$ & $22$ & $6.36...$ & $6.4$ & $18$ & $1.98...$ & $1.81...$ & $188$ & $189$ & \ref{nonta_bd_s}\\
\hline
$8 \times 13$ & $15$ & $6.93...$ & $6.94...$ & $11.03...$ & $3.96...$ & $3.46...$ & $459$ & $462$ & \ref{nonta_bd_2k}\\
\hline
$8 \times 18$ & $20$ & $7.2$ & $7.22...$ & $16$ & $2.92...$ & $2.65...$ & $437$ & $438$ & \\
\hline
$9 \times 17$ & $19$ & $8.05...$ & $8.05...$ & $15$ & $3.97...$ & $3.57...$ & $804$ & $810$ & \ref{nonta_bd_2k-1}\\
\hline
$10 \times 17$ & $19$ & $8.94...$ & $8.95...$ & $15.02...$ & $4.97...$ & $4.47...$ & $1344$ & $1350$ & \ref{nonta_bd_2k}\\
\hline
$15 \times 19$ & $22$ & $12.95...$ & $12.95...$ & $16.22...$ & $9.96...$ & $8.63...$ & $7641$ & $7644$ & \\
\hline
\end{tabular}
\caption{Parameter sets with $r \leq c \leq 20$ for which no near triple arrays exist due to Theorem~\ref{thm:dual_bd}.
The column Comment contains (x) if the parameter set belongs to series (x) of Corollary~\ref{cor:nonta_bd}.}
\label{tbl:nonta_dual_bd}
\end{center}
\end{table}

Among the parameter sets with $3 \leq r, c \leq 20$ and $\max(r, c) \leq v \leq rc$, there are just 12 for which the inequality~\eqref{eq:dual_bd} does not hold.
We present these parameter sets in Table~\ref{tbl:nonta_dual_bd}.
Among the remaining parameters sets in this range, there are 18777 for which the inequality~\eqref{eq:dual_bd} is strict, and 19689 for which equality holds.
It should be mentioned though that for all but 418 of the latter 19689 parameter sets, both sides of~\eqref{eq:dual_bd} are equal to zero, that is, both $\lcc \leq 1$ and $\mu_c \leq 1$ hold.

In the following corollary, we present a collection of examples of families of parameter sets where Theorem~\ref{thm:dual_bd} implies non-existence of near triple arrays.
This collection of examples is by no means exhaustive.

\begin{corollary}\label{cor:nonta_bd}
There are no near triple arrays for the following parameter sets:
\begin{enumerate}[(a)]
    \item\label{nonta_bd_2k-1} $((2k - 1) \times (4k - 3), 4k - 1)$ for $k \geq 3$,
    \item\label{nonta_bd_2k} $(2k \times (4k - 3), 4k - 1)$ for $k \geq 3$,
    \item\label{nonta_bd_s} $(k \times (\binom{k}{2} - s), \binom{k}{2} + 1)$ for positive integers $k, s$ with $1 \leq s \leq \frac{(k - 1)(k - 2)}{2k}$.
\end{enumerate}
\end{corollary}
\begin{proof}[Proof sketch]
In each case we use Lemma~\ref{lm:avg} to calculate the parameters, and then apply Theorem~\ref{thm:dual_bd}.
Since the calculations are rather tedious, we omit the details and for each identity show only the end result.
All identities were obtained via symbolic computation in SageMath~\cite{sagemath}.
\begin{enumerate}[(a)]
\item We calculate $e = 2k - 2 + \frac{1}{4k - 1}$, thus $e^- = 2k - 2$, $e^+ = 2k - 1$ and $\lrc = 2k - 2 + \frac{1}{4k - 3}$.
Then $\lcc = k - 1 - \frac{1}{2(4k - 3)}$, so $\lcc^- = k - 2$, $\lcc^+ = k - 1$.
Also, $\mu_c = k - 2 + \frac{2k + 1}{4k - 1}$, so $\mu_c^- = k - 2$, $\mu_c^+ = k - 1$.
Finally, $S( \binom{c}{2}, \lcc ) - S( \binom{v}{2}, \mu_c ) = -\binom{k - 1}{2} < 0$.
\item We calculate $e = 2k - 1 - \frac{1}{4k - 1}$, thus $e^- = 2k - 2$, $e^+ = 2k - 1$ and $\lrc = 2k - 1 - \frac{k - 1}{(4k - 3)k}$.
Then $\lcc = k - \frac{1}{2(4k - 3)}$, so $\lcc^- = k - 1$, $\lcc^+ = k$.
Also, $\mu_c = k - 1 + \frac{2k + 1}{4k - 1}$, so $\mu_c^- = k - 1$, $\mu_c^+ = k$.
Finally, $S( \binom{c}{2}, \lcc ) - S( \binom{v}{2}, \mu_c ) = -\binom{k - 1}{2} < 0$.
\item We calculate $e = k - \frac{2k(s + 1)}{k^2 - k + 2} \geq k - 1$, with equality precisely when $s = \frac{(k - 1)(k - 2)}{2k}$.
If the equality holds, then $e = \lrc = k - 1$, and otherwise $e^- = k - 1$, $e^+ = k$ and $\lrc = k - \frac{2(k - 1)(s + 1)}{k^2 - k - 2s}$.
In both cases, $\lcc = 2 - \frac{8s(s + 1)}{(k^2 - 2k - 2s)(k^2 - 2k - 2s - 2)}$.
Also, $\mu_c = 2 - \frac{4(s + 1)}{k^2 - k + 2}$.
These expressions for $\lcc$ and $\mu_c$ are monotone in $s$ and are equal to, respectively, $2 - \frac{2}{k^2 + 1} > 1$ and $2 - \frac{2}{k} > 1$, when $s = \frac{(k - 1)(k - 2)}{2k}$, thus $\lcc^- = 1$, $\lcc^+ = 2$, $\mu_c^- = 1$, $\mu_c^+ = 2$.
Finally, $S( \binom{c}{2}, \lcc ) - S( \binom{v}{2}, \mu_c ) = -\binom{s + 1}{2} < 0$.\qedhere
\end{enumerate}
\end{proof}

In Table~\ref{tbl:nonta}, we observed that there is no $(4 \times 11, 13)$-near triple array.
In Theorem~\ref{thm:nonta_column}, we generalize this non-existence result to a larger family of parameter sets.
Note that this result, similar to Theorem~\ref{thm:dual_bd}, can be rephrased more generally: what we actually prove is that for the given parameters there does not exist a block design corresponding to the column design of a near triple array.
The condition $r \geq 4$ is necessary, as illustrated by the existence of the two near triple arrays given in Figure~\ref{fig:nonta_column_tight}.

\begin{figure}[!ht]
\begin{center}
    \begin{tabular}{c c c}
	\begin{tabular}{|*{5}{c}|}
    \hline
    0 & 1 & 2 & 3 & 4\\
    1 & 2 & 5 & 4 & 6\\
    5 & 6 & 3 & 1 & 0\\
    \hline
    \end{tabular}
    & &
    \begin{tabular}{|*{5}{c}|}
    \hline
    0 & 1 & 2 & 3 & 4\\
    1 & 2 & 3 & 5 & 6\\
    6 & 5 & 0 & 4 & 2\\
    \hline
    \end{tabular}
    \end{tabular}
	\end{center}
    \caption{The two non-isotopic $(3 \times 5, 7)$-near triple arrays.}
	\label{fig:nonta_column_tight}
\end{figure}

\begin{theorem}\label{thm:nonta_column}
For $r \geq 4$, there are no $(r \times (r(r - 1) - 1), r(r - 1) + 1)$-near triple arrays.
\end{theorem}

\begin{proof}
Let us assume that $(r \times (r(r - 1) - 1), r(r - 1) + 1)$-near triple arrays exist, and let $T$ be an example of such array.
Then $T$ has
\[
e = \frac{rc}{v} = \frac{r^2(r - 1) - r}{r(r - 1) + 1} = r - \frac{2r}{r(r - 1) + 1}.
\]
For $r \geq 3$ this implies that $r - 1 < e < r$, so since we assume that $r\geq4$ we get that $e^- = r - 1$ and $e^+ = r$.
By Lemma~\ref{lm:k-k+}, we get $v_- = 2r$ and $v_+ = r(r - 3) + 1$.
Then the number of pairs of cells in $T$ containing the same symbol is
\begin{align*}
v_-\binom{e^-}{2} + v_+\binom{e^+}{2} &= \frac{2r(r - 1)(r - 2)}{2} + \frac{(r(r - 3) + 1)r(r - 1)}{2}\\
&= \frac{(r(r - 1) - 1)(r(r - 1) - 2) - 2}{2} = \binom{c}{2} - 1,
\end{align*}
thus
\[
    \lcc = \frac{\binom{c}{2} - 1}{\binom{c}{2}} < 1.
\]
Moreover, there is exactly one pair of columns in $T$ with no common symbols, with every other pair of columns sharing exactly one symbol.
Without loss of generality, let the first two columns of $T$ have no symbols in common.
The first column shares one symbol with each of the remaining $c - 2$ columns.
Since $r \geq 4$, we have
\[
    c - 2 = r(r - 1) - 3 > r(r - 2) = r(e^- - 1),
\]
thus at least one symbol in the first column must have replication number $e^+$.
Without loss of generality, let this symbol be $0$, and let the remaining symbols in the first column be $1, 2, \ldots, r - 1$.
Consider the set $X$ of $e^+ - 1 = r - 1$ other columns containing the symbol $0$.
Since $\lcc < 1$, the $(r - 1)^2$ symbols other than $0$ in $X$ are all distinct and do not appear in the first column, so these symbols are in fact all the remaining symbols $r, \dots, r(r - 1) = v - 1$, each appearing once.

Now consider the $r$ symbols appearing in the second column. 
By the pigeonhole principle, one of the $r-1$ columns of $X$ must contain two of the $r$ symbols of the second column.
This contradicts $\lcc < 1$, so our assumption is wrong and the near triple array $T$ does not exist.
\end{proof}

In Theorem~\ref{thm:ex_3c}, we proved that for $c\geq6 = 3(3-1)$ there exist $(3\times c, v)$-near triple arrays for any $v\geq c$. Theorem~\ref{thm:nonta_column} shows that the bound $c \geq r(r - 1)$ is actually a necessary condition for any $r\geq4$ for the existence of $(r \times c, v)$-near triple arrays for all $v \geq c$.

\section{Near balanced grids}\label{sec:near_balance}

For a binary $r \times c$ row-column design $T$ on $v$ symbols, let $\mu$ be the average number of times a pair of symbols appears together in a row or a column, that is
\[
\mu := \frac{\binom{c}{2}r + \binom{r}{2}c}{\binom{v}{2}} = \frac{c(c - 1)r + r(r - 1)c}{v(v - 1)} = \frac{e(r + c - 2)}{v - 1}.
\]
Define $\mu^- := \lfloor \mu \rfloor$ and $\mu^+ := \lceil \mu \rceil$.
A design $T$ is called a \textit{balanced grid} if $\mu$ is an integer and for each pair of symbols the number of rows and columns containing the pair is equal to $\mu$.
Balanced grids were introduced by McSorley, Phillips, Wallis and Yucas in~\cite{mcsorleyDoubleArraysTriple2005a}.
In particular, they showed that any balanced grid is equireplicate.

Similar to how we defined near triple arrays, we define an $(r \times c, v)$-\textit{near balanced grid} as a binary (near) equireplicate $r \times c$ row-column design on $v$ symbols, in which the number of times each pair of symbols appears in a row or a column is $\mu^-$ or $\mu^+$.
Clearly, a balanced grid is a near balanced grid.

Let $e, \lrc, \lrr, \lcc$ correspond to the parameter set $(r \times c, v)$, as described in Section~\ref{sec:defs}, let $S(n,m)$ be defined as in Section~\ref{sec:comp}, and let $\mu$ be defined as above. The following theorem describes the relation between near balanced grids and near triple arrays, in terms of these quantities.

\begin{theorem}\label{thm:dual_nta-nbg}
Consider a $(r \times c, v)$ binary row column design, and define
\[
\SNTA := S( \binom{c}{2}, \lcc ) + S( \binom{r}{2}, \lrr) + S(rc, \lrc),\quad \SNBG := S( \binom{v}{2}, \mu ).
\]
Then the following holds:
\begin{enumerate}[(a)]
    \item\label{nta-nbg:nonta} If $\SNTA < \SNBG$, then there are no $(r \times c, v)$-near triple arrays, but there may be $(r \times c, v)$-near balanced grids.
    \item If $\SNTA > \SNBG$, then there are no $(r \times c, v)$-near balanced grids, but there may be $(r \times c, v)$-near triple arrays.
    \item\label{nta-nbg:duality} If $\SNTA = \SNBG$, then any $(r \times c, v)$-near triple array is an $(r \times c, v)$-near balanced grid and vice versa.
\end{enumerate}
\end{theorem}
\begin{proof}
Let $T$ be a binary (near) equireplicate $r \times c$ row-column design on $v$ symbols.
Let $l^{cc}_{ij}$, $1 \leq i < j \leq c$, be the number of common symbols between columns $i$ and $j$ of $T$, let $l^{rr}_{ij}$, $1 \leq i < j \leq r$, be the number of common symbols between rows $i$ and $j$, and let $l^{rc}_{ij}$, $1 \leq i \leq r$, $1 \leq j \leq c$, be the number of common symbols between row $i$ and column $j$.
Finally, let $m_{ab}$, $1 \leq a < b \leq v$, be the number of times a row or a column contains both symbols $a$ and $b$.
Counting in two ways the number of times a pair of symbols appears together in a pair of rows and/or columns, we get
\begin{equation}\label{eq:dblcnt_nta-nbg}
\sum_{1 \leq i < j \leq c} \binom{l^{cc}_{ij}}{2} + \sum_{1 \leq i < j \leq r} \binom{l^{rr}_{ij}}{2} + \sum_{\substack{1 \leq i \leq r\\1 \leq j \leq c}} \binom{l^{rc}_{ij}}{2} = \sum_{1 \leq a < b \leq v} \binom{m_{ab}}{2}.
\end{equation}
Applying Lemma~\ref{lm:bin_sum} three times to the terms in the left hand side of~\ref{eq:dblcnt_nta-nbg}
we see that this is greater than or equal to $\SNTA$, and that it is equal to $\SNTA$ if and only if
\[
l^{cc}_{ij} \in \{\lcc^-, \lcc^+\},\quad l^{rr}_{ij} \in \{\lrr^-, \lrr^+\},\quad l^{rc}_{ij} \in \{\lrc^-, \lrc^+\}
\]
for all elements of these collections, which is equivalent to $T$ being a near triple array.
Similarly, applying Lemma~\ref{lm:bin_sum} to the right hand side of~\eqref{eq:dblcnt_nta-nbg}, this is greater than or equal to $\SNBG$, with equality precisely when $m_{ab} \in \{\mu^-, \mu^+\}$ for all $a < b$, that is, when $T$ is a near balanced grid.

This implies that if $T$ is a near triple array, then $\SNTA \geq \SNBG$, if $T$ is a near balanced grid, then $\SNTA \leq \SNBG$, and if $\SNTA = \SNBG$, then $T$ being a near triple array is equivalent to it being a near balanced grid.
\end{proof}

One nice consequence of Theorem~\ref{thm:dual_nta-nbg} is the next corollary, which provides a unified approach to several known properties of triple arrays and balanced grids.

\begin{corollary}\label{cor:ta-bg}
The following claims hold:
\begin{enumerate}[(a)]
\item\label{ta-bg:ta_lower}
Any $(r \times c, v)$-triple array with $v > \max(r, c)$ has $v \geq r + c - 1$.
\item\label{ta-bg:bd_upper}
Any $(r \times c, v)$-balanced grid has $v \leq r + c - 1$.
\item\label{ta-bg:duality}
Any $(r \times c, r + c - 1)$-triple array is an $(r \times c, r + c - 1)$-balanced grid and vice versa.
\end{enumerate}
\end{corollary}
\begin{proof}
The claims are trivial if $r = 1$ or $c = 1$, so from now on we assume that $r, c \geq 2$.
Let $e, \lrc, \lrr, \lcc, \mu$ correspond to the parameter set $(r \times c, v)$, and
let $\SNTA$ and $\SNBG$ be as in Theorem~\ref{thm:dual_nta-nbg}.
Define
\[
\STA := \binom{c}{2}\binom{\lcc}{2} + \binom{r}{2}\binom{\lrr}{2} + rc\binom{e}{2},\quad \SBG := \binom{v}{2}\binom{\mu}{2}.
\]
It can be shown by elementary algebraic manipulations that
\[
\STA - \SBG = \frac{(v - (r + c - 1))(r + c - 2)(v - c)(v - r)e^2}{4(v - 1)(r - 1)(c - 1)}.
\]
We verified this identity using symbolic computation in SageMath~\cite{sagemath}.
In particular, if $v > \max(r, c)$, then the sign of $\STA - \SBG$ coincides with the sign of $v - (r + c - 1)$.

Recall that, due to Remark~\ref{rem:elrc}, $\lrc \geq e$ with equality precisely when $e$ is integer.
It follows that $\SNTA \geq \STA$, with equality when $\lrr$, $\lcc$ and $e$ are integers, that is, when parameters $(r \times c, v)$ are admissible for triple arrays.
Similarly, $\SNBG \geq \SBG$, with equality when $\mu$ is integer.
If $T$ is an $(r \times c, v)$-triple array with $v > \max(r, c)$, then, using Theorem~\ref{thm:dual_nta-nbg},
\[
\STA = \SNTA \geq \SNBG \geq \SBG,
\]
so $\STA - \SBG \geq 0$, which is equivalent to $v - (r + c - 1) \geq 0$.
If additionally $v = r + c - 1$, then $\STA = \SBG$. 
In particular, $\SNBG = \SBG$ and thus $\mu$ is an integer, and $\SNTA = \SNBG$, so due to Theorem~\ref{thm:dual_nta-nbg}~\ref{nta-nbg:duality}, $T$ is a balanced grid.

Conversely, if $T$ is an $(r \times c, v)$-balanced grid, then, using Theorem~\ref{thm:dual_nta-nbg},
\[
\SBG = \SNBG \geq \SNTA \geq \STA,
\]
so $\STA - \SBG \leq 0$, which is equivalent to $v \in \{r, c\}$ or $v - (r + c - 1) \leq 0$.
If additionally $v = r + c - 1$, then $\STA = \SBG$.
In particular, $\SNTA = \STA$ and thus $\lrr$ and $\lcc$ are integers, and $\SNTA = \SNBG$, so, due to Theorem~\ref{thm:dual_nta-nbg}~\ref{nta-nbg:duality}, $T$ is a triple array.
\end{proof}

Corollary~\ref{cor:ta-bg}~\ref{ta-bg:ta_lower} corresponds to Theorem~3.2 in~\cite{mcsorleyDoubleArraysTriple2005a}, but it has also been shown earlier by Bagchi (Corollary~1 in~\cite{bagchiTwowayDesigns1998}) in 1998 and in unpublished work of Bailey and Heidtmann~\cite{bailey1994extremal} in 1994.
Corollary~\ref{cor:ta-bg}~\ref{ta-bg:bd_upper} corresponds to Theorem~4.2 in~\cite{mcsorleyDoubleArraysTriple2005a}.
We note that previous proofs of these two claims used linear algebra.
Finally, Corollary~\ref{cor:ta-bg}~\ref{ta-bg:duality} corresponds to Theorem~6.1 in~\cite{mcsorleyDoubleArraysTriple2005a} and Theorem~2.5 in~\cite{mcsorleyDoubleArraysTriple2005} combined.
Our approach to all three claims is uniform and thus more directly highlights a certain duality between triple arrays and balanced grids.
A concrete family of parameters for which this kind of duality is at play for near triple arrays and near balanced grids is given in the next corollary.

\begin{corollary}\label{cor:nta=nbg}
For $k \geq 3$ and $s \geq k - 1$, $(k \times (k - 1)s, ks)$-near triple arrays are $(k \times (k - 1)s, ks)$-near balanced grids and vice versa.
\end{corollary}
\begin{proof}[Proof sketch]
We use Lemma~\ref{lm:avg} throughout the proof, without specific references. We omit the details of the calculations and show only the end results. All identities were obtained via symbolic computation in SageMath~\cite{sagemath}.

We have $e = k - 1$, and so $\lrc = e = k - 1$, $\lrr = k - 2$.
Additionally, $\lcc = \frac{k(k - 2)}{(k - 1)s - 1}$, $\mu = k - 1 - \frac{(k - 1)(s - k + 1)}{ks - 1}$.
If $s = k - 1$, then $\lcc = 1$, i.e. $(r \times c, v)$ are admissible for triple arrays.
Moreover, in this case $v = r + c - 1$, so the claim follows from Corollary~\ref{cor:ta-bg}~\ref{ta-bg:duality}.
Otherwise, $s > k - 1$, in which case $\lcc^- = 0$, $\lcc^+ = 1$, $\mu^- = k - 2$, $\mu^+ = k - 1$.
A direct calculation then shows that $\SNTA = \SNBG$, so the claim follows by Theorem~\ref{thm:dual_nta-nbg}~\ref{nta-nbg:duality}.
\end{proof}

Theorem~\ref{thm:dual_nta-nbg} can also be used to derive more direct non-existence results, as in the following corollary, where integrality of the parameters involved is crucial. The parameter sets covered by the corollary for $r \leq c \leq 500$ are given in Table~\ref{tbl:nonta=bg+neareq}.

\begin{corollary}\label{cor:nonta=bg+neareq}
Let $e, \lrc, \lrr, \lcc, \mu$ correspond to the parameter set $(r \times c, v)$.
Suppose that $\mu$ is an integer, $e$ is not an integer, and $\SNTA = \SNBG$, where $\SNTA$ and $\SNBG$ are defined as in Theorem~\ref{thm:dual_nta-nbg}.
Then there are no $(r \times c, v)$-near triple arrays.
\end{corollary}
\begin{proof}
Any $(r \times c, v)$-near triple array would be a balanced grid due to Theorem~\ref{thm:dual_nta-nbg} and the fact that $\mu$ is an integer.
But a balanced grid is always equireplicate (see Theorem~4.1 in~\cite{mcsorleyDoubleArraysTriple2005a}), which is not possible when $e$ is not an integer.
\end{proof}

\begin{table}[!ht]
\begin{center}
\begin{tabular}{|c|c||l|l|l|l|l||r|}
\hline
$r \times c$ & $v$ & $e$ & $\lrc$ & $\lrr$ & $\lcc$ & $\mu$ & $\SNTA$ \\
\hline
$27 \times 66$ & $78$ & $22.84...$ & $22.85...$ & $55.47...$ & $9.07...$ & $27$ & $1054053$\\
\hline
$40 \times 40$ & $65$ & $24.61...$ & $24.625$ & $24.23...$ & $24.23...$ & $30$ & $904800$\\
\hline
$42 \times 56$ & $64$ & $36.75$ & $36.75...$ & $48.83...$ & $27.30...$ & $56$ & $3104640$\\
\hline
$45 \times 45$ & $55$ & $36.81...$ & $36.82...$ & $36.63...$ & $36.63...$ & $60$ & $2628450$\\
\hline
\end{tabular}
\caption{Parameter sets with $r \leq c \leq 500$ for which there are no near triple arrays due to Corollary~\ref{cor:nonta=bg+neareq}.}
\label{tbl:nonta=bg+neareq}
\end{center}
\end{table}

\begin{table}[!ht]
\begin{center}
\begin{tabular}{|c|c||l|l|l|l|l||r|r||c|}
\hline
$r \times c$ & $v$ & $e$ & $\lrc$ & $\lrr$ & $\lcc$ & $\mu$ & $\SNTA$ & $\SNBG$ & Comment \\
\hline
$4 \times 5$ & $10$ & $2$ & $2$ & $1.66...$ & $1$ & $1.55...$ & $24$ & $25$ & \ref{nonta_nbg_4k+2},\ref{nonta_nbg_binom}\\
\hline
$4 \times 6$ & $8$ & $3$ & $3$ & $4$ & $1.6$ & $3.42...$ & $117$ & $120$ & \ref{nonta_nbg_2k}\\
\hline
$5 \times 8$ & $10$ & $4$ & $4$ & $6$ & $2.14...$ & $4.88...$ & $426$ & $430$ & \ref{nonta_nbg_2k}\\
\hline
$5 \times 9$ & $15$ & $3$ & $3$ & $4.5$ & $1.25$ & $2.57...$ & $224$ & $225$ & \ref{nonta_nbg_k-2sq},\ref{nonta_nbg_binom}\\
\hline
$6 \times 7$ & $14$ & $3$ & $3$ & $2.8$ & $2$ & $2.53...$ & $186$ & $189$ & \ref{nonta_nbg_4k+2}\\
\hline
$6 \times 8$ & $12$ & $4$ & $4$ & $4.8$ & $2.57...$ & $4.36...$ & $486$ & $492$ & \\
\hline
$6 \times 9$ & $11$ & $4.90...$ & $4.92...$ & $7.06...$ & $2.94...$ & $6.38...$ & $950$ & $951$ & \\
\hline
$6 \times 10$ & $12$ & $5$ & $5$ & $8$ & $2.66...$ & $6.36...$ & $1125$ & $1134$ & \ref{nonta_nbg_2k}\\
\hline
$7 \times 10$ & $14$ & $5$ & $5$ & $6.66...$ & $3.11...$ & $5.76...$ & $1249$ & $1260$ & \\
\hline
$8 \times 9$ & $12$ & $6$ & $6$ & $6.42...$ & $5$ & $8.18...$ & $1932$ & $1944$ & \\
\hline
$8 \times 9$ & $18$ & $4$ & $4$ & $3.85...$ & $3$ & $3.52...$ & $696$ & $702$ & \ref{nonta_nbg_4k+2}\\
\hline
$8 \times 10$ & $16$ & $5$ & $5$ & $5.71...$ & $3.55...$ & $5.33...$ & $1390$ & $1400$ & \\
\hline
$9 \times 10$ & $15$ & $6$ & $6$ & $6.25$ & $5$ & $7.28...$ & $2394$ & $2415$ & \\
\hline
\end{tabular}
\caption{Parameter sets with $r \leq c \leq 10$ for which no near triple arrays exist due to Theorem~\ref{thm:dual_nta-nbg}~\ref{nta-nbg:nonta}.
The column Comment contains (x) if the parameter set belongs to series (x) of Corollary~\ref{cor:nonta_nbg}.}
\label{tbl:nonta_dual_nta-nbg}
\end{center}
\end{table}

\begin{table}[!ht]
\begin{center}
\begin{tabular}{|c|c||l|l|l|l|l||r|r|}
\hline
$r \times c$ & $v$ & $e$ & $\lrc$ & $\lrr$ & $\lcc$ & $\mu$ & $\SNTA$ & $\mathrm{NTA}$ \\
\hline
$3 \times 6$ & $9$ & $2$ & $2$ & $3$ & $0.6$ & $1.75$ & $27$ & $$1$$\\
\hline
$3 \times 8$ & $12$ & $2$ & $2$ & $4$ & $0.42...$ & $1.63...$ & $42$ & $$2$$\\
\hline
$3 \times 10$ & $15$ & $2$ & $2$ & $5$ & $0.33...$ & $1.57...$ & $60$ & $$3$$\\
\hline
$4 \times 4$ & $8$ & $2$ & $2$ & $1.33...$ & $1.33...$ & $1.71...$ & $20$ & $$1$$\\
\hline
$4 \times 6$ & $12$ & $2$ & $2$ & $2$ & $0.8$ & $1.45...$ & $30$ & $$2$$\\
\hline
$6 \times 6$ & $9$ & $4$ & $4$ & $3.6$ & $3.6$ & $5$ & $360$ & $$696$$\\
\hline
$6 \times 6$ & $12$ & $3$ & $3$ & $2.4$ & $2.4$ & $2.72...$ & $162$ & $$48$$\\
\hline
$8 \times 8$ & $16$ & $4$ & $4$ & $3.42...$ & $3.42...$ & $3.73...$ & $624$ & $$?$$\\
\hline
$9 \times 9$ & $16$ & $5.06...$ & $5.07...$ & $4.58...$ & $4.58...$ & $5.4$ & $1440$ & $$?$$\\
\hline
$9 \times 10$ & $13$ & $6.92...$ & $6.93...$ & $7.41...$ & $5.93...$ & $9.80...$ & $3375$ & $$?$$\\
\hline
$10 \times 10$ & $20$ & $5$ & $5$ & $4.44...$ & $4.44...$ & $4.73...$ & $1700$ & $$?$$\\
\hline
\end{tabular}
\caption{Parameter sets with $r \leq c \leq 10$, not including those corresponding to triple arrays, for which $\SNTA = \SNBG$, i.e. near triple arrays are equivalent to near balanced grids due to Theorem~\ref{thm:dual_nta-nbg}~\ref{nta-nbg:duality}.
Column NTA contains the number of non-isotopic near triple arrays or $?$ if the complete enumeration has not been done.}
\label{tbl:nta=nbg}
\end{center}
\end{table}

Among all parameter sets with $3 \leq r, c \leq 20$ and $\max(r, c) \leq v \leq rc$, there are 333 with $\SNTA < \SNBG$, 37411 with $\SNTA > \SNBG$ and 734 with $\SNTA = \SNBG$.
Among the latter 734 parameter sets, for 410 of them $\SNTA > 0$ holds.
In Tables~\ref{tbl:nonta_dual_nta-nbg} and~\ref{tbl:nta=nbg}, we give examples of small parameter sets with $\SNTA < \SNBG$ and $\SNTA = \SNBG$, respectively.
In the latter table, one parameter set, $(6 \times 6, 9)$, is of particular interest.
Balanced grids with $r = c$ are also known as \textit{binary pseudo-Youden designs}.
McSorley and Phillips~\cite{mcsorleyCompleteEnumerationProperties2007} enumerated $(6 \times 6, 9)$ binary pseudo-Youden designs and reported that, up to isotopism, there are 696 of them.
Our search likewise found 696 non-isotopic near triple arrays on this parameter set, and
Theorem~\ref{thm:dual_nta-nbg} explains this.
This duality does not hold for binary pseudo-Youden designs in general: Cheng~\cite{chengFamilyPseudoYouden1981} gave a construction of $(\binom{4k + 4}{2} \times \binom{4k + 4}{2}, (4k + 3)^2)$ binary pseudo-Youden designs whenever $4k + 3$ is a prime of a prime power.
In Corollary~\ref{cor:nonta_nbg}~\ref{nonta_nbg_pyd} below we show that this construction produces near triple arrays only when $k = 0$, that is, for the parameter set $(6 \times 6, 9)$.

Looking at Corollary~\ref{cor:nta=nbg}, curiously $\SNTA < \SNBG$ holds for all parameter sets $(k \times ((k - 1)s, ks)$ with $3 \leq k \leq 1000$ and $2 \leq s < k - 1$, i.e. there are no near triple arrays on these parameters.
This seems much harder to prove in general, but we give a proof for the case $s = 2$ in the next corollary, along with further examples of families of parameters sets where Theorem~\ref{thm:dual_nta-nbg}~\ref{nta-nbg:nonta} implies non-existence of near triple arrays. As for Corollary~\ref{cor:nonta_bd}, the list is not exhaustive.

\begin{corollary}\label{cor:nonta_nbg}
There are no near triple arrays for the following parameter sets:
\begin{enumerate}[(a)]
    \item\label{nonta_nbg_2k} $(k \times 2(k - 1), 2k)$ for $k \geq 4$,
    \item\label{nonta_nbg_k-2sq} $(k \times (k - 2)^2, k(k - 2))$ for $k \geq 5$,
    \item\label{nonta_nbg_4k+2} $(2k \times (2k + 1), 4k + 2)$ for $k \geq 2$,
    \item\label{nonta_nbg_binom} $(k \times (\binom{k}{2} - 1), \binom{k + 1}{2})$ for $k \geq 4$,
    \item\label{nonta_nbg_pyd} $(\binom{4k + 4}{2} \times \binom{4k + 4}{2}, (4k + 3)^2)$ for $k \geq 1$.
\end{enumerate}
\end{corollary}
\begin{proof}[Proof sketch]
Similar to Corollary~\ref{cor:nonta_bd}, we calculate parameter values using Lemma~\ref{lm:avg} and apply Theorem~\ref{thm:dual_nta-nbg}~\ref{nta-nbg:nonta}.
We again omit the details and for each identity show only the end result.
All identities were obtained via symbolic computation in SageMath~\cite{sagemath}.
\begin{enumerate}[(a)]
\item We calculate $\lrc = e = k - 1$, $\lrr = 2(k - 2)$.
Suppose that $k$ is even, so $k = 2q$ for some $q \geq 2$.
Then $\lcc  = q - 1 + \frac{3q - 3}{4q - 3}$, so $\lcc^- = q - 1$, $\lcc^+ = q$, and $\mu = 3q - 3 + \frac{q + 1}{4q - 1}$, so $\mu^- = 3q - 3$, $\mu^+ = 3q - 2$.
Then $\SNTA - \SNBG = -3\binom{q}{2} < 0$.
Now suppose that $k = 2q + 1$ for some $q \geq 2$.
Then $\lcc = q + \frac{q - 1}{4q - 1}$, so $\lcc^- = q$, $\lcc^+ = q + 1$, and $\mu = 3q - 2 + \frac{3q + 2}{4q + 1}$, so $\mu^- = 3q - 2$, $\mu^+ = 3q - 1$.
Then $\SNTA - \SNBG = -\frac{(3q + 2)(q - 1)}{2} < 0$.
\item We calculate $\lrc = e = k - 2$, $\lrr = k^2 - 6k + 9 + \frac{k - 3}{k - 1}$, so $\lrr^- = k^2 - 6k + 9$, $\lrr^+ = k^2 - 6k + 10$.
Next, $\lcc = 1 + \frac{1}{k - 1}$, i.e. $\lcc^- = 1$, $\lcc^+ = 2$, and $\mu = k - 3 + \frac{3k - 7}{k^2 - 2k - 1}$, so $\mu^- = k - 3$ and $\mu^+ = k - 2$.
Then $\SNTA - \SNBG = -\binom{k - 3}{2} < 0$.
\item We calculate $\lrc = e = k$, $\lrr = k - \frac{1}{2k - 1}$, so $\lrr^- = k - 1$, $\lrr^+ = k$, $\lcc = k - 1$, $\mu = k - 1 + \frac{2k + 1}{4k + 1}$, so $\mu^- = k - 1$, $\mu^+ = k$.
Finally, $\SNTA - \SNBG = -\binom{k}{2} < 0$.
\item We calculate $\lrc = e = k - 2$, $\lrr = \frac{k(k - 3)}{2} - 1 + \frac{2}{k - 1}$, so $\lrr^- = \frac{k(k - 3)}{2} - 1$, $\lrr^+ = \frac{k(k - 3)}{2}$.
Next, $\lcc = 1 + \frac{k^2 - 5k + 4}{k^2 - k - 4}$, so $\lcc^- = 1$, $\lcc^+ = 2$, $\mu = k - 3 + \frac{k^2 - 3k + 6}{k^2 + k - 2}$, so $\mu^- = k - 3$, $\mu^+ = k - 2$.
Finally, $\SNTA - \SNBG = -1 < 0$.
\item We calculate $\lrc = e = 4(k + 1)^2$, $\lrr = \lcc = 4(k + 1)^2 - \frac{2k + 2}{4k + 5}$, so $\lrr^- = \lcc^- = 4(k + 1)^2 - 1$, $\lrr^+ = \lcc^+ = 4(k + 1)^2$, $\mu = (4k + 5)(k + 1)$.
Finally, $\SNTA - \SNBG = -(4k + 3)(4k + 2)(k + 1)^2k < 0$.
\qedhere
\end{enumerate}
\end{proof}

\section{Correctness checks}\label{sec:correct}

In this section we describe various measures we took to verify the correctness of our computations.

We wrote a separate implementation of the search algorithm in SageMath~\cite{sagemath}, which we used to perform an independent complete enumeration of near triple arrays for all parameter sets from Tables~\ref{tbl:nta3} through~\ref{tbl:nta5} with $rc \leq 30$.
On the same parameters we also ran a simplified version of our C++ implementation with the completability procedure described in Section~\ref{ssec:compl} replaced with a simple check that the current partially filled array does not already have too large intersections of rows and/or columns, and with the canonicity checking procedure described in Section~\ref{ssec:canon} replaced with a brute-force search over all permutations of rows and columns.
In both cases the results matched the results obtained by our main C++ implementation.

Note that transposes of $(r \times c, v)$-near triple arrays are exactly $(c \times r, v)$-near triple arrays. For all parameter sets from Tables~\ref{tbl:nta3} through~\ref{tbl:nta5} with $rc \leq 30$, as well as for all triple array parameters in Table~\ref{tbl:triple_enum}, we ran the search both for arrays and their transposes and the results matched, as predicted.

Throughout our C++ implementation we included various sanity checks to ensure the program behaves as expected.
In particular, the completability checks described in Section~\ref{ssec:compl} ensure that any completed array obtained by the program is a near triple array.
Nevertheless, since checking that a completed array satisfies all the desired conditions is computationally cheap, we separately checked this for every found completed array.

Our computational results agree with our theoretical results: the search found examples of near triple arrays in all cases where constructions from Section~\ref{sec:exist} are applicable, and found no near triple arrays in all cases where non-existence follows from results of Sections~\ref{sec:comp} and~\ref{sec:near_balance}.

Additionally, for various classes of designs related to near triple arrays, there are previous enumeration results in the literature, and we checked that they matched the results obtained by our program, namely:

\textbf{Triple arrays.}
See Section~\ref{ssec:res_ta}.

\textbf{(Near) Youden rectangles.}
See Section~\ref{ssec:res_nyr}.

\textbf{Double arrays, sesqui arrays, transposed mono arrays and AO-arrays.}
See Section~\ref{ssec:res_int}.

\textbf{$(6 \times 6, 9)$-binary pseudo-Youden designs.} As we already discussed in Section~\ref{sec:near_balance}, Theorem~\ref{thm:dual_nta-nbg} implies that $(6 \times 6, 9)$-binary pseudo-Youden designs are precisely $(6 \times 6, 9)$-near triple arrays.
Our search found 696 non-isotopic near triple arrays on these parameters, which matches counts of binary pseudo-Youden designs given by McSorley and Phillips~\cite{mcsorleyCompleteEnumerationProperties2007}.

\textbf{Latin squares.}
Our counts of $(n \times n, n)$-near triple arrays, $n \leq 8$, match counts of $n \times n$ Latin squares by McKay~\cite{mckayIsomorphFreeExhaustiveGeneration1998} (see Table 3) and by McKay, Meynert and Myrvold~\cite{mckaySmallLatinSquares2007} (see Table II).
In addition to the results reported in Appendix~\ref{ap:nta_tables}, 
we also ran our program for the case $n=8$, and indeed found 1676267 non-isotopic $(8 \times 8, 8)$-near triple arrays, 
which matches the result in~\cite{mckaySmallLatinSquares2007}.

\textbf{Latin rectangles.}
Any $(n - 1) \times n$ Latin rectangle is a Youden rectangle, a triple array, and therefore also a near triple array.
Again, our counts of $((n - 1) \times n, n)$-near triple arrays, $n \leq 8$, match counts given by McKay~\cite{mckayIsomorphFreeExhaustiveGeneration1998} (see Table 3), in all cases reported in Appendix~\ref{ap:nta_tables}, and in the case $n = 8$, where we also found 13302311 $(7 \times 8, 8)$-near triple arrays.

\section{Concluding remarks}\label{sec:concl}

In the present paper, we gave several partial results on the following question, both positive and negative, but the picture is far from complete.
\begin{question}
    For which combinations of $r, c, v$ do near triple arrays exist?
\end{question}

We proved in Theorem~\ref{thm:ex_3c} that $(3 \times c, v)$-near triple arrays exist for any $c \geq 6 = 3(3 - 1)$ and $v \geq c$.
On the other hand, Theorem~\ref{thm:nonta_column} shows that for any given $r \geq 4$, there are no $(r \times c, c + 2)$-near triple arrays for $c = r(r - 1) - 1$. We therefore propose the following conjecture.

\begin{conjecture}
    $(r \times c, v)$-near triple arrays exist for any $c \geq r(r - 1)$ and $v \geq c$.
\end{conjecture}

As we already mentioned earlier, due to Theorem~\ref{thm:dual_nta-nbg}, among all parameter sets with $3 \leq r, c \leq 20$ and $\max(r, c) \leq v \leq rc$, there are 734 for which near balanced grids are near triple arrays, and only 333 for which near balanced grids may exist which are not triple arrays. 
For comparison, on the remaining 37411 parameter sets there may exist near triple arrays, but there are no near balanced grids.
Theorem~\ref{thm:dual_nta-nbg} thus seems, empirically, to imply non-existence of near balanced grids for the majority of parameter sets.
It would be interesting to sharpen these results to a complete resolution of the following question.
\begin{question}
    For which combinations of $r, c, v$ do near balanced grids exist?
\end{question}

A more specialized question is motivated by the unresolved entries in Table \ref{tbl:nta=nbg}.
\begin{question}
    For which $r\geq 4$ do $(r \times r, 2r)$-near triple arrays exist?
\end{question}

For those parameter sets where near triple arrays do not exist, we investigated generalized triple arrays. In relation to this, we ask the following question.
\begin{question}
    Let $\omega(r \times c, v)$ be the smallest integer $\omega$ such that there exists an $(r \times c, v; \omega_{rc}, \omega_{rr}, \omega_{cc})$-generalized triple array with $\omega_{rc}, \omega_{rr}, \omega_{cc} \leq \omega$.
    How can $\omega(r \times c, v)$ be bounded in terms of $r, c, v$?
\end{question}

As an intermediate goal, we pose a related problem for the column design of a generalized triple array.
\begin{question}
    Let $\omega_c(r, c, v)$ be the smallest integer $\omega_c$ such that there exists an integer $x$ and a block design with $v$ blocks and $c$ points, with each point occurring in $r$ blocks, each block containing $\lfloor\frac{rc}{v}\rfloor$ or $\lceil\frac{rc}{v}\rceil$ points, and each pair of points covered by at least $x$ and at most $x + \omega_c - 1$ blocks.
    How can $\omega_c(r, c, v)$ be bounded in terms of $r, c, v$?
\end{question}

In Section~\ref{ssec:res_ta} we describe the structure of two rather symmetric triple arrays given in Figures~\ref{fig:ex120} and~\ref{fig:Fano}.
Our hope is that further study of these examples can lead to new constructions of triple array series.

Finally, an important application of the combinatorial objects studied in the present paper is as experimental designs. In this context, their statistical efficiency is a key property. We hope to return to this question in future work.

\section*{Acknowledgements}

The computational work was performed on resources provided by the Swedish National Infrastructure for Computing (SNIC) at High Performance Computing Center North (HPC2N).
Alexey Gordeev was supported by Kempe foundation grant JCSMK23-0058.

\bibliographystyle{abbrv}
\bibliography{main}
\newpage

\appendix

\clearpage
\section{Full description of the completability checking procedure}\label{ap:compl}

Let $T$ be a partially filled array, and let $(i_0, j_0)$ be the empty cell of $T$ to be filled in the next step of the search.
Our goal is to either deduce that $T$ cannot be completed to an $(r \times c, v)$-near triple array, or, otherwise, to find a set of symbols which may be put into the cell $(i_0, j_0)$ without directly violating any of the near triple array conditions.

Denote the set of all symbols by $V := \{0, 1, \dots, v - 1\}$.
For $i \geq 0$, let $V_i$ be the set of symbols used $i$ times in $T$.
Let $V_{may}$ be the set of symbols which may be used to fill empty cells of $T$: more precisely, $V_{may} = \bigcup_{i = 0}^{e - 1} V_i$ in the equireplicate case.
In the near-equireplicate case, using notation from Lemma~\ref{lm:k-k+}, if $|V_{e^+}| < v_+$, then $V_{may} = \bigcup_{i = 0}^{e^-} V_i$, otherwise, $V_{may} = \bigcup_{i = 0}^{e^- - 1}V_i$.

Let $v_{max}$ be the maximum symbol used so far in $T$, that is, $v_{max} \geq w$ for any symbol $w$ used in $T$, and set $V_{next} := V_{may} \cap \{0, \dots, v_{max} + 1\}$.
In order for the array obtained from $T$ by filling the cell $(i_0, j_0)$ to be canonical, the symbol put into the cell $(i_0, j_0)$ clearly has to come from $V_{next}$.

Denote by $R_i$ and $C_j$ the sets of symbols used so far in row $i$ and column $j$ of $T$, respectively.
At the start of the completability checking procedure, we define a collection of auxiliary sets as follows. For $1 \leq i \leq r$, we define 
\[
A_{R_i} := R_i,\quad B_{R_i} := \begin{cases}
R_i \cup V_{may} &\text{if row } i \text{ of } T \text{ has empty cells,}\\
R_i &\text{otherwise.}
\end{cases}
\]
Similarly, for $1 \leq j \leq c$ we define
\[
A_{C_j} := C_j,\quad B_{C_j} := \begin{cases}
C_j \cup V_{may} &\text{if column } j \text{ of } T \text{ has empty cells,}\\
C_j &\text{otherwise.}
\end{cases}
\]
Finally, for $1 \leq i \leq r$ and $1 \leq j \leq c$ we define
\[
B_{i, j} := \begin{cases}
V_{next} \setminus R_i \setminus C_j &\text{if } (i, j) = (i_0, j_0),\\
V_{may} \setminus R_i \setminus C_j &\text{if cell } (i, j) \text{ of } T \text{ is empty, but } (i, j) \neq (i_0, j_0),\\
\{T_{i, j}\} &\text{if cell } (i, j) \text{ of } T \text{ is not empty}.
\end{cases}
\]

The sets $A_{R_i}, B_{R_i}, A_{C_j}, B_{C_j}, B_{i, j}$ are the \textit{estimate sets}.
Throughout the completability checking procedure, the estimate sets may be changed, but the point of them is that we maintain the following properties throughout: if $T'$ is a near triple array completed from $T$, $R'_i$ and $C'_j$ are contents of row $i$ and column $j$ of $T'$, then
\begin{equation}\label{eq:estprop}
A_{R_i} \subseteq R'_i \subseteq B_{R_i},\quad A_{C_j} \subseteq C'_j \subseteq B_{C_j},\quad T'_{i, j} \in B_{i, j}.    
\end{equation}

In other words, $A_{R_i}$ is a set of symbols that will all appear in row $R_i$ of any completion $T'$ of $T$, and conversely, $B_{R_i}$ is a set of symbols that we cannot yet rule out as possible candidates for use in row $R_i$ of some completion $T'$ of $T$. The sets $A_{C_j}$ and $B_{C_j}$ are the analogues for columns. Finally, the set $B_{i,j}$ is a set of symbols that cannot yet be ruled out as possible candidates for use in cell $(i,j)$ in some completion $T'$ of $T$.

As long as at least one of the rules listed below can be used to either deduce the non-completability of $T$ or to update the contents of some of the estimate sets, one such rule is chosen and applied.
If the chosen rule implies that $T$ is non-completable, we stop the procedure and remove $T$ from further consideration in the search.
When none of the rules can be applied, we stop the procedure and return the estimate set $B_{i_0, j_0}$ corresponding to the empty cell $(i_0,j_0)$ of $T$ as the list of options to put into this cell in the next step of the search.

\begin{enumerate}
\item\textbf{Full content of row $i$/column $j$ determined by $A_{R_i}$/$A_{C_j}$}

\begin{enumerate}
\item If $|A_{R_i}| > c$ for some $1 \leq i \leq r$, then $T$ is non-completable.
If $|A_{R_i}| = c$, then~\eqref{eq:estprop} implies that $R'_i = A_{R_i}$ for any near triple array $T'$ completed from $T$, so we may update $B_{R_i}$ to be equal to $A_{R_i}$.
\item If $|A_{C_j}| > r$ for some $1 \leq j \leq c$, then $T$ is non-completable. 
If $|A_{C_j}| = r$, then, similarly to the previous case, we may update $B_{C_j}$ to be equal to $A_{C_j}$.
\end{enumerate}
\item\textbf{Full content of row $i$/column $j$ determined by $B_{R_i}$/$B_{C_j}$}

\begin{enumerate}
\item If $|B_{R_i}| < c$ for some $1 \leq i \leq r$, then $T$ is non-completable.
If $|B_{R_i}| = c$, then, due to~\eqref{eq:estprop}, we have $R'_i = B_{R_i}$ for any near triple array $T'$ completed from $T$.
Thus we may update $A_{R_i}$ to be equal to $B_{R_i}$.
\item If $|B_{C_j}| < r$ for some $1 \leq j \leq c$, then $T$ is non-completable.
If $|B_{C_j}| = r$, then, similarly to the previous case, we may update $A_{C_j}$ to be equal to $B_{C_j}$.
\end{enumerate}
\item\textbf{Sets $B_{R_i}$ and $B_{C_j}$ restrict set $B_{i, j}$}

For any empty cell $(i, j)$ of $T$, we may replace $B_{i, j}$ with $B_{i, j} \cap (B_{R_i} \setminus R_i) \cap (B_{C_j} \setminus C_j)$.
\item\textbf{Sets $B_{i,j}$ restrict sets $B_{R_i}$ and $B_{C_j}$}

For $1 \leq i \leq r$, we may replace $B_{R_i}$ with $B_{R_i} \cap \bigcup_{j = 1}^c B_{i, j}$.
For $1 \leq j \leq c$, we may replace $B_{C_j}$ with $B_{C_j} \cap \bigcup_{i = 1}^r B_{i, j}$.

\item\textbf{Row-column intersection and empty cells}

Let $(i, j)$ be an empty cell of $T$. 

\begin{enumerate}
\item
If $|R_i \cap C_j| + 1 > \lrc^+$, then $T$ is non-completable, since filling $(i,j)$ contributes $+1$ to the intersection between row $i$ and 
column $j$. 
If $|R_i \cap C_j| + 1 = \lrc^+$, then the symbol used in cell $(i, j)$ is the only further common symbol between row $i$ and column $j$ in any completion of $T$.
Then for any empty cell $(i, j') \neq (i, j)$ of $T$ we may replace $B_{i, j'}$ with $B_{i, j'} \setminus A_{C_j}$, and for any empty cell $(i', j) \neq (i, j)$ of $T$ we may replace $B_{i', j}$ with $B_{i', j} \setminus A_{R_i}$.

\item 
Setting
\[
l_{max} := |R_i \cap C_j| + (c - |R_i|) + (r - |C_j|) - 1,
\]
we also see that if $l_{max} < \lrc^-$, then $T$ is non-completable, since there is not enough room in the empty cells of row $R_i$ and $C_j$ to reach the required intersection size.
If $l_{max} = \lrc^-$, then each empty cell $(i, j') \neq (i, j)$ in any completion of $T$ must be filled with a symbol from $C_j$.
In other words, we may replace $B_{i, j'}$ with $B_{i, j'} \cap C_j$.
Similarly, for any empty cell $(i', j) \neq (i, j)$ of $T$ we may replace $B_{i', j}$ with $B_{i', j} \cap R_i$.
\end{enumerate}

\item\textbf{Lower bound on row-row intersection size}

Consider two rows $i$ and $j$ of $T$.
The sets
\[
X_{R_i} := (B_{R_i} \setminus A_{R_i}) \cap A_{R_j},\quad Y_{R_i} := (B_{R_i} \setminus A_{R_i} \setminus A_{R_j}) \cap B_{R_j},\quad Z_{R_i} := B_{R_i} \setminus A_{R_i} \setminus A_{R_j} \setminus B_{R_j}
\]
together with $A_{R_i}$ are pairwise disjoint and form a partition of $B_{R_i}$ as follows:
\begin{equation}\label{eq:BRipart}
B_{R_i} = A_{R_i} \sqcup X_{R_i} \sqcup Y_{R_i} \sqcup Z_{R_i}.
\end{equation}
Note that the sets $X_{R_i}$, $Y_{R_i}$, $Z_{R_i}$ depend on $A_{R_j}$ and $B_{R_j}$.
Similarly, the sets
\[
X_{R_j} := (B_{R_j} \setminus A_{R_j}) \cap A_{R_i},\quad Y_{R_j} := (B_{R_j} \setminus A_{R_j} \setminus A_{R_i}) \cap B_{R_i},\quad Z_{R_j} := B_{R_j} \setminus A_{R_j} \setminus A_{R_i} \setminus B_{R_i}
\]
together with $A_{R_j}$ are pairwise disjoint and form a partition of $B_{R_j}$:
\begin{equation}\label{eq:BRjpart}
B_{R_j} = A_{R_j} \sqcup X_{R_j} \sqcup Y_{R_j} \sqcup Z_{R_j}.    
\end{equation}
By symmetry $Y_{R_i} = Y_{R_j}$, so from now on we denote both these sets by $Y$.

Suppose $T$ is completable, and let $T'$ be a near triple array completed from $T$.
Let $R'_i$ and $R'_j$ be contents of rows $i$ and $j$ of $T'$.
Using conditions~\eqref{eq:estprop}, partitions~\eqref{eq:BRipart} and~\eqref{eq:BRjpart} give partitions of $R'_i$ and $R'_j$:
\[
R'_i = A_{R_i} \sqcup X'_i \sqcup Y'_i \sqcup Z'_i,\quad R'_j = A_{R_j} \sqcup X'_j \sqcup Y'_j \sqcup Z'_j,
\]
where
\begin{align*}
X'_i &:= R'_i \cap X_{R_i},&&Y'_i := R'_i \cap Y, &&Z'_i := R'_i \cap Z_{R_i},\\
X'_j &:= R'_j \cap X_{R_j},&&Y'_j := R'_j \cap Y, &&Z'_j := R'_j \cap Z_{R_j}.
\end{align*}
These partitions in turn give the following partition of the common intersection of $R'_i$ and $R'_j$:
\begin{equation}\label{eq:RRpart}
R'_i \cap R'_j = (A_{R_i} \cap A_{R_j}) \sqcup X'_i \sqcup X'_j \sqcup (Y'_i \cap Y'_j).    
\end{equation}
Moreover, since each row of $T'$ contains $c$ symbols,
\[
|X'_i \cup Y'_i| \geq f_i := \max(0, c - |A_{R_i}| - |Z_{R_i}|), \quad |X'_j \cup Y'_j| \geq f_j := \max(0, c - |A_{R_j}| - |Z_{R_j}|),
\]
so
\begin{equation}\label{eq:lmin}
|R'_i \cap R'_j| \geq l_{min} := |A_{R_i} \cap A_{R_j}| + \max(0, f_i + f_j - |Y|).
\end{equation}
The quantity $l_{min}$ does not depend on the choice of $T'$, and can be calculated using the estimate sets.
If $l_{min} > \lrr^+$, $T'$ cannot be a near triple array, so our assumption is wrong and $T$ is non-completable.
If $l_{min} = \lrr^+$, then the inequality~\eqref{eq:lmin} must be tight, in which case we can update the estimate sets as follows:
\begin{enumerate}[(a)]
\item If $|Y| > f_i + f_j$ or if $f_i, f_j = 0$, then we must have $|R'_i \cap R'_j| = |A_{R_i} \cap A_{R_j}|$, which together with~\eqref{eq:RRpart} implies
\[
R'_i \subseteq B_{R_i} \setminus X_{R_i},\quad R'_j \subseteq B_{R_j} \setminus X_{R_j}.
\]
Thus we may replace $B_{R_i}$ and $B_{R_j}$ with $B_{R_i} \setminus X_{R_i}$ and $B_{R_j} \setminus X_{R_j}$, respectively.
In all remaining cases we may assume that at least one of $f_i$, $f_j$ is positive and that
\begin{equation}\label{eq:ppff}
|Y| \leq f_i + f_j.
\end{equation}
\item If $f_i > 0$ and $f_j = 0$, then~\eqref{eq:ppff} turns into $|Y| \leq f_i$, and~\eqref{eq:lmin} is tight only when
\[
|R'_i \cap R'_j| = |A_{R_i} \cap A_{R_j}| + f_i - |Y|.
\]
Looking at~\eqref{eq:RRpart}, this is only possible when $X'_j, Y'_j = \emptyset$, $Y'_i = Y$ and $Z'_i = Z_{R_i}$, i.e.
\[
A_{R_i} \cup Y \cup Z_{R_i} \subseteq R'_i,\quad R'_j \subseteq B_{R_j} \setminus X_{R_j} \setminus Y,
\]
so we may replace $A_{R_i}$ and $B_{R_j}$ with $A_{R_i} \cup Y \cup Z_{R_i}$ and $B_{R_j} \setminus X_{R_j} \setminus Y$, respectively.
The case $f_i = 0$ and $f_j > 0$ is similar.
\item If $f_i, f_j > 0$, then~\eqref{eq:lmin} being tight means that
\[
|R'_i \cap R'_j| = |A_{R_i} \cap A_{R_j}| + f_i + f_j - |Y|.
\]
This is only possible when $Z'_i = Z_{R_i}$ and $Z'_j = Z_{R_j}$, i.e.
\[
A_{R_i} \cup Z_{R_i} \subseteq R'_i,\quad A_{R_j} \cup Z_{R_j} \subseteq R'_j,
\]
so we may replace the estimate sets $A_{R_i}$ and $A_{R_j}$ with $A_{R_i} \cup Z_{R_i}$ and $A_{R_j} \cup Z_{R_j}$, respectively.

Additionally, if in this case we have $|Y| = f_i + f_j$, then $|R'_i \cap R'_j| = |A_{R_i} \cap A_{R_j}|$, so $X'_i, X'_j = \emptyset$, that is
\[
R'_i \subseteq B_{R_i} \setminus X_{R_i},\quad R'_j \subseteq B_{R_j} \setminus X_{R_j},
\]
so we may also replace $B_{R_i}$ and $B_{R_j}$ with $B_{R_i} \setminus X_{R_i}$ and $B_{R_j} \setminus X_{R_j}$, respectively.
\end{enumerate}

\item\textbf{Upper bound on row-row intersection size}

We use the same notation as in the previous rule.
In particular, assuming that $T$ is completable to a near triple array $T'$ with $R'_i$ and $R'_j$ contents of rows $i$ and $j$ of $T'$, we again have a partition~\eqref{eq:RRpart} of the common intersection of $R'_i$ and $R'_j$.
Since each row of $T'$ contains $c$ symbols,
\[
|Y'_i \cup Z'_i| \geq s_i := \max(0, c - |A_{R_i}| - |X_{R_i}|), \quad |Y'_j \cup Z'_j| \geq s_j := \max(0, c - |A_{R_j}| - |X_{R_j}|),
\]
so
\begin{align}
\begin{split}\label{eq:lmax}
|R'_i \cap R'_j| \leq l_{max} :=&\ |A_{R_i} \cap A_{R_j}|\\
+& \min(c - |A_{R_i}|, |X_{R_i}|) + \min(c - |A_{R_j}|, |X_{R_j}|)\\
+& \min(|Y|, s_i, s_j).    
\end{split}
\end{align}
The quantity $l_{max}$ does not depend on the choice of $T'$ and can be calculated using the estimate sets.
If $l_{max} < \lrr^-$, then $T'$ cannot be a near triple array, so our assumption is wrong and $T$ is non-completable.
If $l_{max} = \lrr^-$, then the inequality~\eqref{eq:lmax} must be tight, in which case we can update the sets as follows:
\begin{enumerate}[(a)]
\item If $s_i, s_j = 0$, for~\eqref{eq:lmax} to be tight we must have $Y'_i, Z'_i, Y'_j, Z'_j = \emptyset$, so
\[
R'_i \subseteq B_{R_i} \setminus Y \setminus Z_{R_i},\quad R'_j \subseteq B_{R_j} \setminus Y \setminus Z_{R_j}.
\]
Then we can replace estimate sets $B_{R_i}$ and $B_{R_j}$ with $B_{R_i} \setminus Y \setminus Z_{R_i}$ and $B_{R_j} \setminus Y \setminus Z_{R_j}$, respectively.
In the remaining cases we may assume that at least one of $s_i$, $s_j$ is positive.
\item If $s_i, s_j \geq |Y|$, then~\eqref{eq:lmax} is tight only if $X'_i = X_{R_i}$, $X'_j = X_{R_j}$ and $Y'_i, Y'_j = Y$, so
\[
A_{R_i} \cup X_{R_i} \cup Y \subseteq R'_i,\quad A_{R_j} \cup X_{R_j} \cup Y \subseteq R'_j.
\]
Thus we can replace $A_{R_i}$ and $A_{R_j}$ with $A_{R_i} \cup X_{R_i} \cup Y$ and $A_{R_j} \cup X_{R_j} \cup Y$, respectively.
In the remaining cases we may assume that at least one of $s_i$, $s_j$ is less than $|Y|$.
\item If $s_i > s_j$, then~\eqref{eq:lmax} being tight implies that $X'_i = X_{R_i}$ and $Z'_j = \emptyset$, so
\[
A_{R_i} \cup X_{R_i} \subseteq R'_i,\quad R'_j \subseteq B_{R_j} \setminus Z_{R_j}.
\]
Thus we can replace estimate sets $A_{R_i}$ and $B_{R_j}$ with $A_{R_i} \cup X_{R_i}$ and $B_{R_j} \setminus Z_{R_j}$, respectively.
The case $s_j > s_i$ is similar.
\item Finally, if $s_i = s_j$, then
\[
|Y| > s_i = s_j > 0.
\]
In this case,~\eqref{eq:lmax} is tight only when $X'_i = X_{R_i}$, $X'_j = X_{R_j}$ and $Z'_i, Z'_j = \emptyset$, thus
\[
A_{R_i} \cup X_{R_i} \subseteq R'_i \subseteq B_{R_i} \setminus Z_{R_i},\quad A_{R_j} \cup X_{R_j} \subseteq R'_j \subseteq B_{R_j} \setminus Z_{R_j}.
\]
It means we can replace estimate sets $A_{R_i}$, $B_{R_i}$, $A_{R_j}$, $B_{R_j}$ with, respectively, $A_{R_i} \cup X_{R_i}$, $B_{R_i} \setminus Z_{R_i}$, $A_{R_j} \cup X_{R_j}$, $B_{R_j} \setminus Z_{R_j}$.
\end{enumerate}
\item\textbf{Lower and upper bounds on column-column and row-column intersection sizes}

By performing a similar analysis, we get rules analogous to the previous two for the intersection of two columns, and for the intersection of a row and a column.
\end{enumerate}

\clearpage
\section{Counts of near triple arrays}\label{ap:nta_tables}
\begin{table}[!ht]
\setlength\tabcolsep{2pt}
\begin{center}
\begin{tabular}{|r|r||*{13}{r|}}
\hline
\multicolumn{2}{|r||}{$c$} & $3$ & $4$ & $5$ & $6$ & $7$ & $8$ & $9$ & $10$ & $11$ & $12$ & $13$ & $14$ & $15$\\
\hline
\hline
$v$ & $3$ & $1$ & & & & & & & & & & & &\\
\hhline{~--------------}
 & $4$ & $2$ & $2$ & & & & & & & & & & &\\
\hhline{~--------------}
 & $5$ & \cellcolor{gray!30}$1$ & $5$ & $2$ & & & & & & & & & &\\
\hhline{~--------------}
 & $6$ & \cellcolor{gray!30}$\mathbf{0}$ & \cellcolor{gray!30}$\mathbf{0}$ & $8$ & $2$ & & & & & & & & &\\
\hhline{~--------------}
 & $7$ & \cellcolor{gray!30}$1$ & \cellcolor{gray!30}$1$ & $2$ & $1$ & $1$ & & & & & & & &\\
\hhline{~--------------}
 & $8$ & \cellcolor{gray!30}$1$ & \cellcolor{gray!30}$3$ & \cellcolor{gray!30}$\mathbf{0}$ & $10$ & $12$ & $4$ & & & & & & &\\
\hhline{~--------------}
 & $9$ & $1$ & \cellcolor{gray!30}$3$ & \cellcolor{gray!30}$3$ & \cellcolor{gray!30}$1$ & $122$ & $129$ & $11$ & & & & & &\\
\hhline{~--------------}
 & $10$ & & \cellcolor{gray!30}$2$ & \cellcolor{gray!30}$11$ & \cellcolor{gray!30}$2$ & $6$ & $1439$ & $1690$ & $80$ & & & & &\\
\hhline{~--------------}
 & $11$ & & \cellcolor{gray!30}$1$ & \cellcolor{gray!30}$14$ & \cellcolor{gray!30}$16$ & \cellcolor{gray!30}$1$ & $184$ & $21474$ & $23160$ & $852$ & & & &\\
\hhline{~--------------}
 & $12$ & & $1$ & \cellcolor{gray!30}$4$ & \cellcolor{gray!30}$26$ & \cellcolor{gray!30}$9$ & \cellcolor{gray!30}$2$ & $6441$ & $338331$ & $328862$ & $11598$ & & &\\
\hhline{~--------------}
 & $13$ & & & \cellcolor{gray!30}$2$ & \cellcolor{gray!30}$51$ & \cellcolor{gray!30}$80$ & \cellcolor{gray!30}$8$ & $27$ & $193776$ & $5614315$ & $4927142$ & $169262$ & &\\
\hhline{~--------------}
 & $14$ & & & \cellcolor{gray!30}$1$ & \cellcolor{gray!30}$26$ & \cellcolor{gray!30}$205$ & \cellcolor{gray!30}$97$ & \cellcolor{gray!30}$4$ & $2481$ & $5272193$ & $97892858$ & $77912954$ & $2636564$ &\\
\hhline{~--------------}
 & $15$ & & & $1$ & \cellcolor{gray!30}$5$ & \cellcolor{gray!30}$89$ & \cellcolor{gray!30}$206$ & \cellcolor{gray!30}$42$ & \cellcolor{gray!30}$3$ & $182564$ & $136305671$ & $1790055841$ & $1300585623$ & $43373610$\\
\hhline{~--------------}
 & $16$ & & & & \cellcolor{gray!30}$2$ & \cellcolor{gray!30}$91$ & \cellcolor{gray!30}$866$ & \cellcolor{gray!30}$614$ & \cellcolor{gray!30}$42$ & $184$ & $10524953$ & $3454555697$ & $34294920533$ & $22894088053$\\
\hhline{~--------------}
 & $17$ & & & & \cellcolor{gray!30}$1$ & \cellcolor{gray!30}$30$ & \cellcolor{gray!30}$797$ & \cellcolor{gray!30}$2376$ & \cellcolor{gray!30}$665$ & \cellcolor{gray!30}$20$ & $30576$ & $503196513$ & $87419635977$ & $687642840827$\\
\hhline{~--------------}
 & $18$ & & & & $1$ & \cellcolor{gray!30}$5$ & \cellcolor{gray!30}$164$ & \cellcolor{gray!30}$1603$ & \cellcolor{gray!30}$1843$ & \cellcolor{gray!30}$258$ & \cellcolor{gray!30}$14$ & $3916524$ & $21262916079$ & $2233771041861$\\
\hhline{~--------------}
 & $19$ & & & & & \cellcolor{gray!30}$2$ & \cellcolor{gray!30}$111$ & \cellcolor{gray!30}$3185$ & \cellcolor{gray!30}$12181$ & \cellcolor{gray!30}$5211$ & \cellcolor{gray!30}$273$ & $1470$ & $377392364$ & $830377660598$\\
\hhline{~--------------}
 & $20$ & & & & & \cellcolor{gray!30}$1$ & \cellcolor{gray!30}$31$ & \cellcolor{gray!30}$1503$ & \cellcolor{gray!30}$18000$ & \cellcolor{gray!30}$26838$ & \cellcolor{gray!30}$5320$ & \cellcolor{gray!30}$133$ & $382611$ & $29140453434$\\
\hhline{~--------------}
 & $21$ & & & & & $1$ & \cellcolor{gray!30}$5$ & \cellcolor{gray!30}$200$ & \cellcolor{gray!30}$5411$ & \cellcolor{gray!30}$25987$ & \cellcolor{gray!30}$18090$ & \cellcolor{gray!30}$1990$ & \cellcolor{gray!30}$52$ & $74659793$\\
\hhline{~--------------}
 & $22$ & & & & & & \cellcolor{gray!30}$2$ & \cellcolor{gray!30}$115$ & \cellcolor{gray!30}$6091$ & \cellcolor{gray!30}$81679$ & \cellcolor{gray!30}$162605$ & \cellcolor{gray!30}$49207$ & \cellcolor{gray!30}$2079$ & $13514$\\
\hhline{~--------------}
 & $23$ & & & & & & \cellcolor{gray!30}$1$ & \cellcolor{gray!30}$31$ & \cellcolor{gray!30}$1923$ & \cellcolor{gray!30}$59477$ & \cellcolor{gray!30}$340938$ & \cellcolor{gray!30}$314182$ & \cellcolor{gray!30}$48283$ & \cellcolor{gray!30}$1027$\\
\hhline{~--------------}
 & $24$ & & & & & & $1$ & \cellcolor{gray!30}$5$ & \cellcolor{gray!30}$213$ & \cellcolor{gray!30}$10069$ & \cellcolor{gray!30}$149724$ & \cellcolor{gray!30}$399150$ & \cellcolor{gray!30}$195274$ & \cellcolor{gray!30}$17593$\\
\hhline{~--------------}
 & $25$ & & & & & & & \cellcolor{gray!30}$2$ & \cellcolor{gray!30}$116$ & \cellcolor{gray!30}$7920$ & \cellcolor{gray!30}$258111$ & \cellcolor{gray!30}$1756572$ & \cellcolor{gray!30}$2193436$ & \cellcolor{gray!30}$510576$\\
\hhline{~--------------}
 & $26$ & & & & & & & \cellcolor{gray!30}$1$ & \cellcolor{gray!30}$31$ & \cellcolor{gray!30}$2063$ & \cellcolor{gray!30}$112193$ & \cellcolor{gray!30}$1838668$ & \cellcolor{gray!30}$5993253$ & \cellcolor{gray!30}$3883956$\\
\hhline{~--------------}
 & $27$ & & & & & & & $1$ & \cellcolor{gray!30}$5$ & \cellcolor{gray!30}$215$ & \cellcolor{gray!30}$13196$ & \cellcolor{gray!30}$440424$ & \cellcolor{gray!30}$3566678$ & \cellcolor{gray!30}$6114257$\\
\hhline{~--------------}
 & $28$ & & & & & & & & \cellcolor{gray!30}$2$ & \cellcolor{gray!30}$116$ & \cellcolor{gray!30}$8625$ & \cellcolor{gray!30}$482725$ & \cellcolor{gray!30}$8688039$ & \cellcolor{gray!30}$34802693$\\
\hhline{~--------------}
 & $29$ & & & & & & & & \cellcolor{gray!30}$1$ & \cellcolor{gray!30}$31$ & \cellcolor{gray!30}$2097$ & \cellcolor{gray!30}$149310$ & \cellcolor{gray!30}$5274747$ & \cellcolor{gray!30}$48674447$\\
\hhline{~--------------}
 & $30$ & & & & & & & & $1$ & \cellcolor{gray!30}$5$ & \cellcolor{gray!30}$216$ & \cellcolor{gray!30}$14462$ & \cellcolor{gray!30}$803966$ & \cellcolor{gray!30}$15849563$\\
\hhline{~--------------}
 & $31$ & & & & & & & & & \cellcolor{gray!30}$2$ & \cellcolor{gray!30}$116$ & \cellcolor{gray!30}$8804$ & \cellcolor{gray!30}$647187$ & \cellcolor{gray!30}$23915420$\\
\hhline{~--------------}
 & $32$ & & & & & & & & & \cellcolor{gray!30}$1$ & \cellcolor{gray!30}$31$ & \cellcolor{gray!30}$2102$ & \cellcolor{gray!30}$165868$ & \cellcolor{gray!30}$9596949$\\
\hhline{~--------------}
 & $33$ & & & & & & & & & $1$ & \cellcolor{gray!30}$5$ & \cellcolor{gray!30}$216$ & \cellcolor{gray!30}$14832$ & \cellcolor{gray!30}$1073706$\\
\hhline{~--------------}
 & $34$ & & & & & & & & & & \cellcolor{gray!30}$2$ & \cellcolor{gray!30}$116$ & \cellcolor{gray!30}$8840$ & \cellcolor{gray!30}$724595$\\
\hhline{~--------------}
 & $35$ & & & & & & & & & & \cellcolor{gray!30}$1$ & \cellcolor{gray!30}$31$ & \cellcolor{gray!30}$2103$ & \cellcolor{gray!30}$171035$\\
\hhline{~--------------}
 & $36$ & & & & & & & & & & $1$ & \cellcolor{gray!30}$5$ & \cellcolor{gray!30}$216$ & \cellcolor{gray!30}$14910$\\
\hhline{~--------------}
 & $37$ & & & & & & & & & & & \cellcolor{gray!30}$2$ & \cellcolor{gray!30}$116$ & \cellcolor{gray!30}$8845$\\
\hhline{~--------------}
 & $38$ & & & & & & & & & & & \cellcolor{gray!30}$1$ & \cellcolor{gray!30}$31$ & \cellcolor{gray!30}$2103$\\
\hhline{~--------------}
 & $39$ & & & & & & & & & & & $1$ & \cellcolor{gray!30}$5$ & \cellcolor{gray!30}$216$\\
\hhline{~--------------}
 & $40$ & & & & & & & & & & & & \cellcolor{gray!30}$2$ & \cellcolor{gray!30}$116$\\
\hhline{~--------------}
 & $41$ & & & & & & & & & & & & \cellcolor{gray!30}$1$ & \cellcolor{gray!30}$31$\\
\hhline{~--------------}
 & $42$ & & & & & & & & & & & & $1$ & \cellcolor{gray!30}$5$\\
\hhline{~--------------}
 & $43$ & & & & & & & & & & & & & \cellcolor{gray!30}$2$\\
\hhline{~--------------}
 & $44$ & & & & & & & & & & & & & \cellcolor{gray!30}$1$\\
\hhline{~--------------}
 & $45$ & & & & & & & & & & & & & $1$\\
\hline
\end{tabular}
\end{center}
\caption{The number of non-isotopic $(3 \times c, v)$-near triple arrays.
The cell is shaded if the corresponding parameter $e^+$ is even.}
\label{tbl:nta3}
\end{table}

\begin{table}[!ht]
\begin{center}
\begin{tabular}{|r|r||*{9}{r|}}
\hline
\multicolumn{2}{|r||}{$c$} & $4$ & $5$ & $6$ & $7$ & $8$ & $9$ & $10$ & $11$ & $12$\\
\hline
\hline
$v$ & $4$ & \cellcolor{gray!30}$2$ & & & & & & & &\\
\hhline{~----------}
 & $5$ & \cellcolor{gray!30}$6$ & \cellcolor{gray!30}$3$ & & & & & & &\\
\hhline{~----------}
 & $6$ & $20$ & \cellcolor{gray!30}$132$ & \cellcolor{gray!30}$34$ & & & & & &\\
\hhline{~----------}
 & $7$ & $8$ & $\mathbf{0}$ & \cellcolor{gray!30}$24$ & \cellcolor{gray!30}$6$ & & & & &\\
\hhline{~----------}
 & $8$ & \cellcolor{gray!30}$1$ & $144$ & $\mathbf{0}$ & \cellcolor{gray!30}$6310$ & \cellcolor{gray!30}$285$ & & & &\\
\hhline{~----------}
 & $9$ & \cellcolor{gray!30}$\mathbf{0}$ & $15$ & $255$ & \cellcolor{gray!30}$\mathbf{0}$ & \cellcolor{gray!30}$331625$ & \cellcolor{gray!30}$5342$ & & &\\
\hhline{~----------}
 & $10$ & \cellcolor{gray!30}$3$ & \cellcolor{gray!30}$\mathbf{0}$ & $43$ & $38$ & \cellcolor{gray!30}$1176$ & \cellcolor{gray!30}$1687883$ & \cellcolor{gray!30}$9722$ & &\\
\hhline{~----------}
 & $11$ & \cellcolor{gray!30}$7$ & \cellcolor{gray!30}$1$ & $1$ & $63$ & $8$ & \cellcolor{gray!30}$3062$ & \cellcolor{gray!30}$606102$ & \cellcolor{gray!30}$1598$ &\\
\hhline{~----------}
 & $12$ & \cellcolor{gray!30}$14$ & \cellcolor{gray!30}$9$ & \cellcolor{gray!30}$2$ & $51$ & $2$ & $1$ & \cellcolor{gray!30}$\mathbf{0}$ & \cellcolor{gray!30}$2600$ & \cellcolor{gray!30}$262$\\
\hhline{~----------}
 & $13$ & \cellcolor{gray!30}$14$ & \cellcolor{gray!30}$45$ & \cellcolor{gray!30}$6$ & $26$ & $1630$ & $101$ & \cellcolor{gray!30}$24$ & \cellcolor{gray!30}$\mathbf{0}$ & \cellcolor{gray!30}$192$\\
\hhline{~----------}
 & $14$ & \cellcolor{gray!30}$4$ & \cellcolor{gray!30}$29$ & \cellcolor{gray!30}$65$ & \cellcolor{gray!30}$1$ & $24025$ & $8322$ & $96$ & \cellcolor{gray!30}$1753$ & \cellcolor{gray!30}$57220$\\
\hhline{~----------}
 & $15$ & \cellcolor{gray!30}$1$ & \cellcolor{gray!30}$74$ & \cellcolor{gray!30}$547$ & \cellcolor{gray!30}$27$ & $724$ & $2838837$ & $214834$ & $144$ & \cellcolor{gray!30}$201373$\\
\hhline{~----------}
 & $16$ & $1$ & \cellcolor{gray!30}$60$ & \cellcolor{gray!30}$1674$ & \cellcolor{gray!30}$127$ & \cellcolor{gray!30}$5$ & $250125$ & $86711536$ & $1488902$ & $71$\\
\hhline{~----------}
 & $17$ & & \cellcolor{gray!30}$21$ & \cellcolor{gray!30}$1225$ & \cellcolor{gray!30}$1986$ & \cellcolor{gray!30}$156$ & $36393$ & $267557001$ & $774406930$ & $3624700$\\
\hhline{~----------}
 & $18$ & & \cellcolor{gray!30}$4$ & \cellcolor{gray!30}$201$ & \cellcolor{gray!30}$15697$ & \cellcolor{gray!30}$3798$ & \cellcolor{gray!30}$3$ & $101615015$ & $49175310984$ & $1894336935$\\
\hhline{~----------}
 & $19$ & & \cellcolor{gray!30}$1$ & \cellcolor{gray!30}$253$ & \cellcolor{gray!30}$40135$ & \cellcolor{gray!30}$19370$ & \cellcolor{gray!30}$357$ & $304711$ & $34374903385$ & $+$\\
\hhline{~----------}
 & $20$ & & $1$ & \cellcolor{gray!30}$101$ & \cellcolor{gray!30}$32472$ & \cellcolor{gray!30}$23784$ & \cellcolor{gray!30}$27757$ & \cellcolor{gray!30}$221$ & $4052892965$ & $+$\\
\hhline{~----------}
 & $21$ & & & \cellcolor{gray!30}$24$ & \cellcolor{gray!30}$8199$ & \cellcolor{gray!30}$219769$ & \cellcolor{gray!30}$437809$ & \cellcolor{gray!30}$4166$ & $5044186$ & $+$\\
\hhline{~----------}
 & $22$ & & & \cellcolor{gray!30}$4$ & \cellcolor{gray!30}$569$ & \cellcolor{gray!30}$595863$ & \cellcolor{gray!30}$2183763$ & \cellcolor{gray!30}$45490$ & \cellcolor{gray!30}$2453$ & $16765885394$\\
\hhline{~----------}
 & $23$ & & & \cellcolor{gray!30}$1$ & \cellcolor{gray!30}$403$ & \cellcolor{gray!30}$546980$ & \cellcolor{gray!30}$3318477$ & \cellcolor{gray!30}$1964596$ & \cellcolor{gray!30}$110894$ & $207132761$\\
\hhline{~----------}
 & $24$ & & & $1$ & \cellcolor{gray!30}$111$ & \cellcolor{gray!30}$183576$ & \cellcolor{gray!30}$1276477$ & \cellcolor{gray!30}$27462929$ & \cellcolor{gray!30}$3274452$ & \cellcolor{gray!30}$13778$\\
\hhline{~----------}
 & $25$ & & & & \cellcolor{gray!30}$24$ & \cellcolor{gray!30}$22422$ & \cellcolor{gray!30}$5024686$ & \cellcolor{gray!30}$128522759$ & \cellcolor{gray!30}$25397255$ & \cellcolor{gray!30}$723049$\\
\hhline{~----------}
 & $26$ & & & & \cellcolor{gray!30}$4$ & \cellcolor{gray!30}$882$ & \cellcolor{gray!30}$5911305$ & \cellcolor{gray!30}$214769083$ & \cellcolor{gray!30}$48593583$ & \cellcolor{gray!30}$42755237$\\
\hhline{~----------}
 & $27$ & & & & \cellcolor{gray!30}$1$ & \cellcolor{gray!30}$457$ & \cellcolor{gray!30}$2598146$ & \cellcolor{gray!30}$122072545$ & \cellcolor{gray!30}$783409971$ & \cellcolor{gray!30}$838121559$\\
\hhline{~----------}
 & $28$ & & & & $1$ & \cellcolor{gray!30}$113$ & \cellcolor{gray!30}$465518$ & \cellcolor{gray!30}$19265323$ & \cellcolor{gray!30}$4006629828$ & \cellcolor{gray!30}$5874053693$\\
\hhline{~----------}
 & $29$ & & & & & \cellcolor{gray!30}$24$ & \cellcolor{gray!30}$34714$ & \cellcolor{gray!30}$38180496$ & \cellcolor{gray!30}$7483581415$ & \cellcolor{gray!30}$13555343682$\\
\hhline{~----------}
 & $30$ & & & & & \cellcolor{gray!30}$4$ & \cellcolor{gray!30}$1002$ & \cellcolor{gray!30}$23940212$ & \cellcolor{gray!30}$5395821358$ & \cellcolor{gray!30}$8585291699$\\
\hhline{~----------}
 & $31$ & & & & & \cellcolor{gray!30}$1$ & \cellcolor{gray!30}$464$ & \cellcolor{gray!30}$6107278$ & \cellcolor{gray!30}$1442407076$ & \cellcolor{gray!30}$63316583808$\\
\hhline{~----------}
 & $32$ & & & & & $1$ & \cellcolor{gray!30}$113$ & \cellcolor{gray!30}$709486$ & \cellcolor{gray!30}$115646704$ & \cellcolor{gray!30}$148523119798$\\
\hhline{~----------}
 & $33$ & & & & & & \cellcolor{gray!30}$24$ & \cellcolor{gray!30}$39938$ & \cellcolor{gray!30}$133951426$ & \cellcolor{gray!30}$135060766907$\\
\hhline{~----------}
 & $34$ & & & & & & \cellcolor{gray!30}$4$ & \cellcolor{gray!30}$1027$ & \cellcolor{gray!30}$52211264$ & \cellcolor{gray!30}$50537462784$\\
\hhline{~----------}
 & $35$ & & & & & & \cellcolor{gray!30}$1$ & \cellcolor{gray!30}$465$ & \cellcolor{gray!30}$9095115$ & \cellcolor{gray!30}$7490105521$\\
\hhline{~----------}
 & $36$ & & & & & & $1$ & \cellcolor{gray!30}$113$ & \cellcolor{gray!30}$819858$ & \cellcolor{gray!30}$355086761$\\
\hhline{~----------}
 & $37$ & & & & & & & \cellcolor{gray!30}$24$ & \cellcolor{gray!30}$41110$ & \cellcolor{gray!30}$271914797$\\
\hhline{~----------}
 & $38$ & & & & & & & \cellcolor{gray!30}$4$ & \cellcolor{gray!30}$1029$ & \cellcolor{gray!30}$75685889$\\
\hhline{~----------}
 & $39$ & & & & & & & \cellcolor{gray!30}$1$ & \cellcolor{gray!30}$465$ & \cellcolor{gray!30}$10504476$\\
\hhline{~----------}
 & $40$ & & & & & & & $1$ & \cellcolor{gray!30}$113$ & \cellcolor{gray!30}$848391$\\
\hhline{~----------}
 & $41$ & & & & & & & & \cellcolor{gray!30}$24$ & \cellcolor{gray!30}$41266$\\
\hhline{~----------}
 & $42$ & & & & & & & & \cellcolor{gray!30}$4$ & \cellcolor{gray!30}$1030$\\
\hhline{~----------}
 & $43$ & & & & & & & & \cellcolor{gray!30}$1$ & \cellcolor{gray!30}$465$\\
\hhline{~----------}
 & $44$ & & & & & & & & $1$ & \cellcolor{gray!30}$113$\\
\hhline{~----------}
 & $45$ & & & & & & & & & \cellcolor{gray!30}$24$\\
\hhline{~----------}
 & $46$ & & & & & & & & & \cellcolor{gray!30}$4$\\
\hhline{~----------}
 & $47$ & & & & & & & & & \cellcolor{gray!30}$1$\\
\hhline{~----------}
 & $48$ & & & & & & & & & $1$\\
\hline
\end{tabular}
\end{center}
\caption{The number of non-isotopic $(4 \times c, v)$-near triple arrays.
Cells with $+$ correspond to cases where the complete enumeration was not finished, but some examples were found.
The cell is shaded if the corresponding parameter $e^+$ is even.}
\label{tbl:nta4}
\end{table}

\begin{table}[!ht]
\begin{center}
\begin{tabular}{|r|r||*{6}{r|}}
\hline
\multicolumn{2}{|r||}{$c$} & $5$ & $6$ & $7$ & $8$ & $9$ & $10$\\
\hline
\hline
$v$ & $5$ & $2$ & & & & &\\
\hhline{~-------}
 & $6$ & $109$ & $40$ & & & &\\
\hhline{~-------}
 & $7$ & \cellcolor{gray!30}$4680$ & $38246$ & $5205$ & & &\\
\hhline{~-------}
 & $8$ & \cellcolor{gray!30}$\mathbf{0}$ & \cellcolor{gray!30}$\mathbf{0}$ & $128886$ & $6688$ & &\\
\hhline{~-------}
 & $9$ & $58$ & \cellcolor{gray!30}$51248$ & \cellcolor{gray!30}$\mathbf{0}$ & $183128838$ & $2757904$ &\\
\hhline{~-------}
 & $10$ & $2$ & $7$ & \cellcolor{gray!30}$19596$ & \cellcolor{gray!30}$\mathbf{0}$ & $219802698$ & $1913816$\\
\hhline{~-------}
 & $11$ & $1064$ & $877$ & \cellcolor{gray!30}$1670$ & \cellcolor{gray!30}$\mathbf{0}$ & $\mathbf{0}$ & $866014$\\
\hhline{~-------}
 & $12$ & $12$ & $222152$ & $45$ & \cellcolor{gray!30}$518859$ & \cellcolor{gray!30}$\mathbf{0}$ & $\mathbf{0}$\\
\hhline{~-------}
 & $13$ & \cellcolor{gray!30}$3$ & $22655$ & $34847$ & \cellcolor{gray!30}$377$ & \cellcolor{gray!30}$165902231$ & \cellcolor{gray!30}$\mathbf{0}$\\
\hhline{~-------}
 & $14$ & \cellcolor{gray!30}$4$ & $186$ & $2561224$ & $462$ & \cellcolor{gray!30}$480486$ & \cellcolor{gray!30}$+$\\
\hhline{~-------}
 & $15$ & \cellcolor{gray!30}$4$ & \cellcolor{gray!30}$\mathbf{0}$ & $1628933$ & $2435740$ & $\mathbf{0}$ & \cellcolor{gray!30}$+$\\
\hhline{~-------}
 & $16$ & \cellcolor{gray!30}$20$ & \cellcolor{gray!30}$1$ & $217$ & $2668280$ & $703$ & \cellcolor{gray!30}$+$\\
\hhline{~-------}
 & $17$ & \cellcolor{gray!30}$378$ & \cellcolor{gray!30}$17$ & $1$ & $15461$ & $22436$ & $\mathbf{0}$\\
\hhline{~-------}
 & $18$ & \cellcolor{gray!30}$1492$ & \cellcolor{gray!30}$382$ & \cellcolor{gray!30}$2$ & $40482$ & $30551$ & $\mathbf{0}$\\
\hhline{~-------}
 & $19$ & \cellcolor{gray!30}$1962$ & \cellcolor{gray!30}$1356$ & \cellcolor{gray!30}$180$ & $3398$ & $13758974$ & $51273$\\
\hhline{~-------}
 & $20$ & \cellcolor{gray!30}$1051$ & \cellcolor{gray!30}$1051$ & \cellcolor{gray!30}$8274$ & \cellcolor{gray!30}$17$ & $1521373$ & $9582082$\\
\hhline{~-------}
 & $21$ & \cellcolor{gray!30}$241$ & \cellcolor{gray!30}$11651$ & \cellcolor{gray!30}$133685$ & \cellcolor{gray!30}$410$ & $6693245$ & $+$\\
\hhline{~-------}
 & $22$ & \cellcolor{gray!30}$31$ & \cellcolor{gray!30}$33065$ & \cellcolor{gray!30}$701262$ & \cellcolor{gray!30}$19914$ & $32848$ & $+$\\
\hhline{~-------}
 & $23$ & \cellcolor{gray!30}$4$ & \cellcolor{gray!30}$34799$ & \cellcolor{gray!30}$1259950$ & \cellcolor{gray!30}$886120$ & \cellcolor{gray!30}$3091$ & $+$\\
\hhline{~-------}
 & $24$ & \cellcolor{gray!30}$1$ & \cellcolor{gray!30}$15784$ & \cellcolor{gray!30}$750183$ & \cellcolor{gray!30}$16163512$ & \cellcolor{gray!30}$60592$ & $26784571$\\
\hhline{~-------}
 & $25$ & $1$ & \cellcolor{gray!30}$3377$ & \cellcolor{gray!30}$120800$ & \cellcolor{gray!30}$115784626$ & \cellcolor{gray!30}$414940$ & \cellcolor{gray!30}$41153$\\
\hhline{~-------}
 & $26$ & & \cellcolor{gray!30}$379$ & \cellcolor{gray!30}$401523$ & \cellcolor{gray!30}$333793857$ & \cellcolor{gray!30}$25971198$ & \cellcolor{gray!30}$2362118$\\
\hhline{~-------}
 & $27$ & & \cellcolor{gray!30}$35$ & \cellcolor{gray!30}$427529$ & \cellcolor{gray!30}$397490453$ & \cellcolor{gray!30}$673586835$ & \cellcolor{gray!30}$78105957$\\
\hhline{~-------}
 & $28$ & & \cellcolor{gray!30}$4$ & \cellcolor{gray!30}$194189$ & \cellcolor{gray!30}$195608765$ & \cellcolor{gray!30}$7216790813$ & \cellcolor{gray!30}$1020860674$\\
\hhline{~-------}
 & $29$ & & \cellcolor{gray!30}$1$ & \cellcolor{gray!30}$42662$ & \cellcolor{gray!30}$37391589$ & \cellcolor{gray!30}$32615487253$ & \cellcolor{gray!30}$4473909398$\\
\hhline{~-------}
 & $30$ & & $1$ & \cellcolor{gray!30}$5110$ & \cellcolor{gray!30}$2203445$ & \cellcolor{gray!30}$65764321020$ & \cellcolor{gray!30}$5579638990$\\
\hhline{~-------}
 & $31$ & & & \cellcolor{gray!30}$411$ & \cellcolor{gray!30}$3405980$ & \cellcolor{gray!30}$61329868571$ & \cellcolor{gray!30}$134724076378$\\
\hhline{~-------}
 & $32$ & & & \cellcolor{gray!30}$35$ & \cellcolor{gray!30}$1831437$ & \cellcolor{gray!30}$26809296536$ & \cellcolor{gray!30}$1128645501851$\\
\hhline{~-------}
 & $33$ & & & \cellcolor{gray!30}$4$ & \cellcolor{gray!30}$458386$ & \cellcolor{gray!30}$5402919445$ & \cellcolor{gray!30}$4038030856747$\\
\hhline{~-------}
 & $34$ & & & \cellcolor{gray!30}$1$ & \cellcolor{gray!30}$62637$ & \cellcolor{gray!30}$465771713$ & \cellcolor{gray!30}$6717790442263$\\
\hhline{~-------}
 & $35$ & & & $1$ & \cellcolor{gray!30}$5628$ & \cellcolor{gray!30}$13452549$ & \cellcolor{gray!30}$5485741819034$\\
\hhline{~-------}
 & $36$ & & & & \cellcolor{gray!30}$417$ & \cellcolor{gray!30}$11806850$ & \cellcolor{gray!30}$2261762679072$\\
\hhline{~-------}
 & $37$ & & & & \cellcolor{gray!30}$35$ & \cellcolor{gray!30}$3857004$ & \cellcolor{gray!30}$473080761333$\\
\hhline{~-------}
 & $38$ & & & & \cellcolor{gray!30}$4$ & \cellcolor{gray!30}$648895$ & \cellcolor{gray!30}$48965130731$\\
\hhline{~-------}
 & $39$ & & & & \cellcolor{gray!30}$1$ & \cellcolor{gray!30}$69217$ & \cellcolor{gray!30}$2314176175$\\
\hhline{~-------}
 & $40$ & & & & $1$ & \cellcolor{gray!30}$5690$ & \cellcolor{gray!30}$39085268$\\
\hhline{~-------}
 & $41$ & & & & & \cellcolor{gray!30}$417$ & \cellcolor{gray!30}$22523768$\\
\hhline{~-------}
 & $42$ & & & & & \cellcolor{gray!30}$35$ & \cellcolor{gray!30}$5264112$\\
\hhline{~-------}
 & $43$ & & & & & \cellcolor{gray!30}$4$ & \cellcolor{gray!30}$716516$\\
\hhline{~-------}
 & $44$ & & & & & \cellcolor{gray!30}$1$ & \cellcolor{gray!30}$70261$\\
\hhline{~-------}
 & $45$ & & & & & $1$ & \cellcolor{gray!30}$5696$\\
\hhline{~-------}
 & $46$ & & & & & & \cellcolor{gray!30}$417$\\
\hhline{~-------}
 & $47$ & & & & & & \cellcolor{gray!30}$35$\\
\hhline{~-------}
 & $48$ & & & & & & \cellcolor{gray!30}$4$\\
\hhline{~-------}
 & $49$ & & & & & & \cellcolor{gray!30}$1$\\
\hhline{~-------}
 & $50$ & & & & & & $1$\\
\hline
\end{tabular}
\end{center}
\caption{The number of non-isotopic $(5 \times c, v)$-near triple arrays.
Cells with $+$ correspond to cases where the complete enumeration was not finished, but some examples were found.
The cell is shaded if the corresponding parameter $e^+$ is even.}
\label{tbl:nta5}
\end{table}

\begin{table}[!ht]
\begin{center}
\begin{tabular}{|r|r||*{3}{r|}|r|}
\hline
\multicolumn{2}{|r||}{$r \times c$} & $6 \times 6$ & $6 \times 7$ & $6 \times 8$ & $7 \times 7$\\
\hline
\hline
$v$ & $6$ & \cellcolor{gray!30}$22$ & & &\\
\hhline{~-----}
 & $7$ & \cellcolor{gray!30}$23746$ & \cellcolor{gray!30}$3479$ & & $564$\\
\hhline{~-----}
 & $8$ & $25797136$ & \cellcolor{gray!30}$187768202$ & \cellcolor{gray!30}$21956009$ & $106228849$\\
\hhline{~-----}
 & $9$ & \cellcolor{gray!30}$696$ & $24788690$ & \cellcolor{gray!30}$18453776222$ & \cellcolor{gray!30}$+$\\
\hhline{~-----}
 & $10$ & \cellcolor{gray!30}$3552$ & $\mathbf{0}$ & $\mathbf{0}$ & $\mathbf{0}$\\
\hhline{~-----}
 & $11$ & \cellcolor{gray!30}$18194$ & \cellcolor{gray!30}$\mathbf{0}$ & $\mathbf{0}$ & $\mathbf{0}$\\
\hhline{~-----}
 & $12$ & $48$ & \cellcolor{gray!30}$59288586$ & \cellcolor{gray!30}$\mathbf{0}$ & $\mathbf{0}$\\
\hhline{~-----}
 & $13$ & $1292$ & \cellcolor{gray!30}$305875$ & \cellcolor{gray!30}$+$ & \cellcolor{gray!30}$+$\\
\hhline{~-----}
 & $14$ & $156$ & $\mathbf{0}$ & \cellcolor{gray!30}$+$ & \cellcolor{gray!30}$209276$\\
\hhline{~-----}
 & $15$ & $2256752$ & $32932$ & \cellcolor{gray!30}$+$ & \cellcolor{gray!30}$+$\\
\hhline{~-----}
 & $16$ & $11321683$ & $212352507$ & $85$ & \cellcolor{gray!30}$+$\\
\hhline{~-----}
 & $17$ & $49662$ & $3198380866$ & $51680$ & $+$\\
\hhline{~-----}
 & $18$ & \cellcolor{gray!30}$10$ & $29151126$ & $+$ & $+$\\
\hhline{~-----}
 & $19$ & \cellcolor{gray!30}$18$ & $39608203$ & $+$ & $+$\\
\hhline{~-----}
 & $20$ & \cellcolor{gray!30}$12$ & $38199$ & $+$ & $+$\\
\hhline{~-----}
 & $21$ & \cellcolor{gray!30}$10$ & \cellcolor{gray!30}$\mathbf{0}$ & $+$ & $+$\\
\hhline{~-----}
 & $22$ & \cellcolor{gray!30}$253$ & \cellcolor{gray!30}$14$ & $+$ & $+$\\
\hhline{~-----}
 & $23$ & \cellcolor{gray!30}$12274$ & \cellcolor{gray!30}$731$ & $308$ & $+$\\
\hhline{~-----}
 & $24$ & \cellcolor{gray!30}$215695$ & \cellcolor{gray!30}$20287$ & \cellcolor{gray!30}$70$ & $+$\\
\hhline{~-----}
 & $25$ & \cellcolor{gray!30}$1388636$ & \cellcolor{gray!30}$191755$ & \cellcolor{gray!30}$7416$ & \cellcolor{gray!30}$194$\\
\hhline{~-----}
 & $26$ & \cellcolor{gray!30}$3643800$ & \cellcolor{gray!30}$555724$ & \cellcolor{gray!30}$554459$ & \cellcolor{gray!30}$566$\\
\hhline{~-----}
 & $27$ & \cellcolor{gray!30}$4276477$ & \cellcolor{gray!30}$418852$ & \cellcolor{gray!30}$22065914$ & \cellcolor{gray!30}$379$\\
\hhline{~-----}
 & $28$ & \cellcolor{gray!30}$2445135$ & \cellcolor{gray!30}$9820703$ & \cellcolor{gray!30}$388697987$ & \cellcolor{gray!30}$142$\\
\hhline{~-----}
 & $29$ & \cellcolor{gray!30}$720821$ & \cellcolor{gray!30}$74786706$ & \cellcolor{gray!30}$2856419193$ & \cellcolor{gray!30}$33561$\\
\hhline{~-----}
 & $30$ & \cellcolor{gray!30}$113560$ & \cellcolor{gray!30}$240004321$ & \cellcolor{gray!30}$8790938224$ & \cellcolor{gray!30}$3252688$\\
\hhline{~-----}
 & $31$ & \cellcolor{gray!30}$10012$ & \cellcolor{gray!30}$362934810$ & \cellcolor{gray!30}$11245395929$ & \cellcolor{gray!30}$126061414$\\
\hhline{~-----}
 & $32$ & \cellcolor{gray!30}$579$ & \cellcolor{gray!30}$278730543$ & \cellcolor{gray!30}$5630028557$ & \cellcolor{gray!30}$2215766156$\\
\hhline{~-----}
 & $33$ & \cellcolor{gray!30}$40$ & \cellcolor{gray!30}$114999499$ & \cellcolor{gray!30}$875826406$ & \cellcolor{gray!30}$18609936195$\\
\hhline{~-----}
 & $34$ & \cellcolor{gray!30}$4$ & \cellcolor{gray!30}$26543461$ & \cellcolor{gray!30}$6278248215$ & \cellcolor{gray!30}$78666874054$\\
\hhline{~-----}
 & $35$ & \cellcolor{gray!30}$1$ & \cellcolor{gray!30}$3532905$ & \cellcolor{gray!30}$15947990679$ & \cellcolor{gray!30}$175829559343$\\
\hhline{~-----}
 & $36$ & $1$ & \cellcolor{gray!30}$280516$ & \cellcolor{gray!30}$18901668915$ & \cellcolor{gray!30}$216924726101$\\
\hhline{~-----}
 & $37$ & & \cellcolor{gray!30}$14491$ & \cellcolor{gray!30}$11705603557$ & \cellcolor{gray!30}$153271802938$\\
\hhline{~-----}
 & $38$ & & \cellcolor{gray!30}$626$ & \cellcolor{gray!30}$4044873951$ & \cellcolor{gray!30}$63993999295$\\
\hhline{~-----}
 & $39$ & & \cellcolor{gray!30}$40$ & \cellcolor{gray!30}$815339796$ & \cellcolor{gray!30}$16196868113$\\
\hhline{~-----}
 & $40$ & & \cellcolor{gray!30}$4$ & \cellcolor{gray!30}$99275901$ & \cellcolor{gray!30}$2534828720$\\
\hhline{~-----}
 & $41$ & & \cellcolor{gray!30}$1$ & \cellcolor{gray!30}$7593009$ & \cellcolor{gray!30}$249091015$\\
\hhline{~-----}
 & $42$ & & $1$ & \cellcolor{gray!30}$391222$ & \cellcolor{gray!30}$15628165$\\
\hhline{~-----}
 & $43$ & & & \cellcolor{gray!30}$15737$ & \cellcolor{gray!30}$651525$\\
\hhline{~-----}
 & $44$ & & & \cellcolor{gray!30}$635$ & \cellcolor{gray!30}$20544$\\
\hhline{~-----}
 & $45$ & & & \cellcolor{gray!30}$40$ & \cellcolor{gray!30}$677$\\
\hhline{~-----}
 & $46$ & & & \cellcolor{gray!30}$4$ & \cellcolor{gray!30}$40$\\
\hhline{~-----}
 & $47$ & & & \cellcolor{gray!30}$1$ & \cellcolor{gray!30}$4$\\
\hhline{~-----}
 & $48$ & & & $1$ & \cellcolor{gray!30}$1$\\
\hhline{~-----}
 & $49$ & & & & $1$\\
\hline
\end{tabular}
\end{center}
\caption{The number of non-isotopic $(6 \times c, v)$ and $(7 \times 7, v)$-near triple arrays.
Cells with $+$ correspond to cases where the complete enumeration was not finished, but some examples were found.
The cell is shaded if the corresponding parameter $e^+$ is even.}
\label{tbl:nta67}
\end{table}

\clearpage
\section{Counts of triple arrays}\label{ap:TA}

\begin{table}[!ht]
\begin{center}
\begin{tabular}{|r|r||r|r|r|r|r|r|}
\hline
\multicolumn{2}{|r||}{$v$} & $6$ & $10$ & $12$ & $14$ & $15$ & $20$\\
\hline
\multicolumn{2}{|r||}{$r \times c$} & $3 \times 4$ & $5 \times 6$ & $4 \times 9$ & $7 \times 8$ & $6 \times 10$ & $5 \times 16$\\
\hline
\multicolumn{2}{|r||}{$e$} & $2$ & $3$ & $3$ & $4$ & $4$ & $4$\\
\hline
\multicolumn{2}{|r||}{$\lrc$} & $2$ & $3$ & $3$ & $4$ & $4$ & $4$\\
\hline
\multicolumn{2}{|r||}{$\lrr$} & $2$ & $3$ & $6$ & $4$ & $6$ & $12$\\
\hline
\multicolumn{2}{|r||}{$\lcc$} & $1$ & $2$ & $1$ & $3$ & $2$ & $1$\\
\hline
\multicolumn{2}{|r||}{Total \#} & $0$ & $7$ & $1$ & $684782$ & $270119$ & $26804$\\
\hline
\hline
$|\Aut| $& $1$ & & & & $682054$ & $263790$ & $26714$\\
\hhline{~-------}
& $2$ & & & & $1266$ & $5280$ &\\
\hhline{~-------}
& $3$ & & $2$ & $1$ & $1277$ & $260$ & $90$\\
\hhline{~-------}
& $4$ & & $1$ & & $98$ & $579$ &\\
\hhline{~-------}
& $5$ & & & & & $1$ &\\
\hhline{~-------}
& $6$ & & $1$ & & $48$ & $69$ &\\
\hhline{~-------}
& $7$ & & & & $2$ & &\\
\hhline{~-------}
& $8$ & & & & $12$ & $88$ &\\
\hhline{~-------}
& $10$ & & & & & $2$ &\\
\hhline{~-------}
& $12$ & & $2$ & & $9$ & $17$ &\\
\hhline{~-------}
& $16$ & & & & & $11$ &\\
\hhline{~-------}
& $18$ & & & & & $1$ &\\
\hhline{~-------}
& $20$ & & & & & $4$ &\\
\hhline{~-------}
& $21$ & & & & $8$ & &\\
\hhline{~-------}
& $24$ & & & & $7$ & $9$ &\\
\hhline{~-------}
& $36$ & & & & & $2$ &\\
\hhline{~-------}
& $48$ & & & & & $4$ &\\
\hhline{~-------}
& $60$ & & $1$ & & & &\\
\hhline{~-------}
& $120$ & & & & & $1$ &\\
\hhline{~-------}
& $168$ & & & & $1$ & &\\
\hhline{~-------}
& $720$ & & & & & $1$ &\\
\hline
\end{tabular}
\caption{The number of triple arrays sorted by autotopism group order.}
\label{tbl:triple_enum}
\end{center}
\end{table}

\clearpage
\section{New counts of near Youden rectangles}\label{ap:NYR}

\begin{table}[!ht]
\begin{center}
\begin{tabular}{|r|r||r|r|r|r|}
\hline
\multicolumn{2}{|r||}{$r \times c$} & $3 \times 14$ & $3 \times 15$ & $4 \times 14$ & $5 \times 12$\\
\hline
\multicolumn{2}{|r||}{$\lcc$} & $0.46...$ & $0.42...$ & $0.92...$ & $1.81...$\\
\hline
\multicolumn{2}{|r||}{Total \#} & $2636564$ & $43373610$ & $825$ & $96671180$\\
\hline
\hline
$|\Aut| $& $1$ & $2629306$ & $43341502$ & $625$ & $96559477$\\
\hhline{~-----}
& $2$ & $6980$ & $31077$ & $136$ & $105305$\\
\hhline{~-----}
& $3$ & $189$ & $803$ & $23$ & $2817$\\
\hhline{~-----}
& $4$ & $47$ & & $24$ & $3007$\\
\hhline{~-----}
& $5$ & & $25$ & &\\
\hhline{~-----}
& $6$ & $31$ & $178$ & $4$ & $233$\\
\hhline{~-----}
& $7$ & $2$ & & &\\
\hhline{~-----}
& $8$ & & & $6$ & $225$\\
\hhline{~-----}
& $9$ & & $2$ & & $7$\\
\hhline{~-----}
& $10$ & & $12$ & &\\
\hhline{~-----}
& $12$ & $1$ & & $2$ & $60$\\
\hhline{~-----}
& $14$ & $6$ & & $1$ &\\
\hhline{~-----}
& $15$ & & $2$ & &\\
\hhline{~-----}
& $16$ & & & & $10$\\
\hhline{~-----}
& $18$ & & $2$ & & $5$\\
\hhline{~-----}
& $21$ & & $2$ & $1$ &\\
\hhline{~-----}
& $24$ & & & $2$ & $25$\\
\hhline{~-----}
& $30$ & & $3$ & &\\
\hhline{~-----}
& $36$ & & & & $2$\\
\hhline{~-----}
& $42$ & & & $1$ &\\
\hhline{~-----}
& $48$ & & & & $5$\\
\hhline{~-----}
& $56$ & & $1$ & &\\
\hhline{~-----}
& $72$ & & & & $1$\\
\hhline{~-----}
& $144$ & & & & $1$\\
\hhline{~-----}
& $168$ & & $1$ & &\\
\hhline{~-----}
& $294$ & $2$ & & &\\
\hline
\end{tabular}
\caption{The number of $r \times c$ near Youden rectangles, that is, $(r \times c, c)$-near triple arrays, sorted by autotopism group order.}
\label{tbl:nyr_enum}
\end{center}
\end{table}

\clearpage
\section{Counts of other row-column designs}\label{ap:RCD}

\begin{table}[!ht]
\begin{center}
\begin{tabular}{|r|r|r||r|r|r|r|r||r|}
\hline
$v$ & $e$ & $r \times c$ & $\mathrm{MA}^T$ & $\mathrm{SA}$ & $\mathrm{DA}$ & $\mathrm{TA}$ & $\mathrm{AO}$ & $\mathrm{NTA}$\\
\hline
\hline
\multirow{2}{*}{$6$} & \multirow{2}{*}{$2$} & $3 \times 4$ & $3$ & $2$ & $2$ & $0$ & $-$ & $0$\\
\hhline{~~-------}
& & $4 \times 3$ & $-$ & $-$ & $2$ & $0$ & $-$ & $0$\\
\hline
\hline
\multirow{2}{*}{$8$} & \multirow{1}{*}{$2$} & $4 \times 4$ &  &  &  &  & $20$ & \cellcolor{gray!30}$1$\\
\hhline{~--------}
& \multirow{1}{*}{$3$} & $4 \times 6$ & $12336$ & $113$ &  &  & $-$ & \cellcolor{gray!30}$0$\\
\hline
\hline
\multirow{2}{*}{$9$} & \multirow{1}{*}{$2$} & $3 \times 6$ & $104$ & $5$ &  &  & $-$ & \cellcolor{gray!30}$1$\\
\hhline{~--------}
& \multirow{1}{*}{$4$} & $6 \times 6$ &  &  &  &  & $53215$ & \cellcolor{gray!30}$696$\\
\hline
\hline
\multirow{4}{*}{$10$} & \multirow{1}{*}{$2$} & $5 \times 4$ & $189$ & $1$ &  &  & $45$ & \cellcolor{gray!30}$0$\\
\hhline{~--------}
& \multirow{2}{*}{$3$} & $5 \times 6$ & $8364560$ & $49$ & $24663$ & $7$ & $8707$ & $7$\\
\hhline{~~-------}
& & $6 \times 5$ & $362120$ & $0$ & $24663$ & $7$ & $8707$ & $7$\\
\hhline{~--------}
& \multirow{1}{*}{$4$} & $5 \times 8$ & \cellcolor{gray!30}$153420549948$ & $1549129$ &  &  & $-$ & \cellcolor{gray!30}$0$\\
\hline
\hline
\multirow{8}{*}{$12$} & \multirow{2}{*}{$2$} & $3 \times 8$ & $4367$ & $15$ &  &  & $-$ & \cellcolor{gray!30}$2$\\
\hhline{~~-------}
& & $4 \times 6$ & $29695$ & $20$ &  &  & $312$ & \cellcolor{gray!30}$2$\\
\hhline{~--------}
& \multirow{3}{*}{$3$} & $4 \times 9$ & \cellcolor{gray!30}$21951950407$ & $249625$ & $2893$ & $1$ & $-$ & $1$\\
\hhline{~~-------}
& & $9 \times 4$ & $-$ & $-$ & $2893$ & $1$ & $-$ & $1$\\
\hhline{~~-------}
& & $6 \times 6$ &  &  &  &  & \cellcolor{gray!30}$3381640$ & \cellcolor{gray!30}$48$\\
\hhline{~--------}
& \multirow{1}{*}{$4$} & $6 \times 8$ &  &  &  &  & \cellcolor{gray!30}$18333488962$ & \cellcolor{gray!30}$0$\\
\hhline{~--------}
& \multirow{1}{*}{$5$} & $6 \times 10$ & $+$ & $+$ &  &  & $-$ & \cellcolor{gray!30}$0$\\
\hhline{~--------}
& \multirow{1}{*}{$6$} & $9 \times 8$ & $+$ & $+$ &  &  & $+$ & $?$\\
\hline
\hline
\multirow{6}{*}{$14$} & \multirow{1}{*}{$2$} & $4 \times 7$ &  &  &  &  & $1632$ & \cellcolor{gray!30}$1$\\
\hhline{~--------}
& \multirow{1}{*}{$3$} & $7 \times 6$ & \cellcolor{gray!30}$2243780000083$ & $44602$ &  &  & \cellcolor{gray!30}$916201732$ & \cellcolor{gray!30}$0$\\
\hhline{~--------}
& \multirow{2}{*}{$4$} & $7 \times 8$ & $+$ & $+$ & $+$ & \cellcolor{gray!30}$684782$ & $+$ & \cellcolor{gray!30}$684782$\\
\hhline{~~-------}
& & $8 \times 7$ & $+$ & \cellcolor{gray!30}$28233890$ & $+$ & \cellcolor{gray!30}$684782$ & $+$ & \cellcolor{gray!30}$684782$\\
\hhline{~--------}
& \multirow{1}{*}{$5$} & $7 \times 10$ &  &  &  &  & $+$ & $?$\\
\hhline{~--------}
& \multirow{1}{*}{$6$} & $7 \times 12$ & $+$ & $+$ &  &  & $-$ & $?$\\
\hline
\end{tabular}
\end{center}
\caption{All non-trivial admissible parameter sets with $v \leq 14$, with existence or the number of non-isotopic proper transposed mono arrays ($\mathrm{MA}^T$), proper sesqui arrays ($\mathrm{SA}$), proper double arrays ($\mathrm{DA}$), triple arrays ($\mathrm{TA}$), proper AO-arrays ($\mathrm{AO}$), and near triple arrays ($\mathrm{NTA}$).
Shaded cells correspond to new results. 
The symbol $-$ indicates that the existence of a proper design is ruled out by Proposition 3.2 from~\cite{jagerEnumerationRowColumnDesigns2024}, $+$ indicates that there are examples but we have no complete enumeration, and $?$ indicates that the question of existence is open.
An empty cell indicates that the parameter set is not admissible.}
\label{tbl:RCDupto14}
\end{table}

\begin{table}[!ht]
\begin{center}
\begin{tabular}{|r|r|r||r|r|r|r|r||r|}
\hline
$v$ & $e$ & $r \times c$ & $\mathrm{MA}^T$ & $\mathrm{SA}$ & $\mathrm{DA}$ & $\mathrm{TA}$ & $\mathrm{AO}$ & $\mathrm{NTA}$\\
\hline
\hline
\multirow{5}{*}{$15$} & \multirow{2}{*}{$2$} & $3 \times 10$ & \cellcolor{gray!30}$274270$ & \cellcolor{gray!30}$44$ &  &  & $-$ & \cellcolor{gray!30}$3$\\
\hhline{~~-------}
& & $6 \times 5$ & \cellcolor{gray!30}$1241870$ & $3$ &  &  & \cellcolor{gray!30}$5761$ & \cellcolor{gray!30}$0$\\
\hhline{~--------}
& \multirow{1}{*}{$3$} & $5 \times 9$ &  &  &  &  & \cellcolor{gray!30}$2530173570$ & \cellcolor{gray!30}$0$\\
\hhline{~--------}
& \multirow{2}{*}{$4$} & $6 \times 10$ & $?$ & $+$ & \cellcolor{gray!30}$+$ & \cellcolor{gray!30}$270119$ & $+$ & \cellcolor{gray!30}$270119$\\
\hhline{~~-------}
& & $10 \times 6$ & $?$ & $0$ & \cellcolor{gray!30}$+$ & \cellcolor{gray!30}$270119$ & $+$ & \cellcolor{gray!30}$270119$\\
\hline
\hline
\multirow{3}{*}{$16$} & \multirow{1}{*}{$2$} & $4 \times 8$ &  &  &  &  & \cellcolor{gray!30}$13428$ & \cellcolor{gray!30}$5$\\
\hhline{~--------}
& \multirow{2}{*}{$3$} & $4 \times 12$ & \cellcolor{gray!30}$+$ & \cellcolor{gray!30}$4592589885$ &  &  & $-$ & \cellcolor{gray!30}$71$\\
\hhline{~~-------}
& & $6 \times 8$ &  &  &  &  & \cellcolor{gray!30}$450296427296$ & \cellcolor{gray!30}$85$\\
\hline
\hline
\multirow{3}{*}{$18$} & \multirow{3}{*}{$2$} & $3 \times 12$ & \cellcolor{gray!30}$22152485$ & \cellcolor{gray!30}$199$ &  &  & $-$ & \cellcolor{gray!30}$14$\\
\hhline{~~-------}
& & $4 \times 9$ & \cellcolor{gray!30}$5315487248$ & \cellcolor{gray!30}$1292$ &  &  & \cellcolor{gray!30}$115557$ & \cellcolor{gray!30}$3$\\
\hhline{~~-------}
& & $6 \times 6$ &  &  &  &  & \cellcolor{gray!30}$522950$ & \cellcolor{gray!30}$10$\\
\hline
\hline
\multirow{4}{*}{$20$} & \multirow{2}{*}{$2$} & $4 \times 10$ &  &  &  &  & \cellcolor{gray!30}$1262840$ & \cellcolor{gray!30}$221$\\
\hhline{~~-------}
& & $5 \times 8$ & \cellcolor{gray!30}$1539025721354$ & \cellcolor{gray!30}$7473$ &  &  & \cellcolor{gray!30}$4408695$ & \cellcolor{gray!30}$17$\\
\hhline{~--------}
& \multirow{2}{*}{$4$} & $5 \times 16$ & $?$ & $+$ & \cellcolor{gray!30}$+$ & \cellcolor{gray!30}$26804$ & $-$ & \cellcolor{gray!30}$26804$\\
\hhline{~~-------}
& & $16 \times 5$ & $-$ & $-$ & \cellcolor{gray!30}$+$ & \cellcolor{gray!30}$26804$ & $-$ & \cellcolor{gray!30}$26804$\\
\hline
\hline
\multirow{2}{*}{$21$} & \multirow{2}{*}{$2$} & $3 \times 14$ & \cellcolor{gray!30}$2147715121$ & \cellcolor{gray!30}$1069$ &  &  & $-$ & \cellcolor{gray!30}$52$\\
\hhline{~~-------}
& & $7 \times 6$ & \cellcolor{gray!30}$608103648334$ & \cellcolor{gray!30}$254$ &  &  & \cellcolor{gray!30}$53169721$ & \cellcolor{gray!30}$0$\\
\hline
\hline
\multirow{1}{*}{$22$} & \multirow{1}{*}{$2$} & $4 \times 11$ &  &  &  &  & \cellcolor{gray!30}$15149193$ & \cellcolor{gray!30}$2453$\\
\hline
\hline
\multirow{2}{*}{$24$} & \multirow{2}{*}{$2$} & $4 \times 12$ & \cellcolor{gray!30}$+$ & \cellcolor{gray!30}$1242243$ &  &  & \cellcolor{gray!30}$200484208$ & \cellcolor{gray!30}$13778$\\
\hhline{~~-------}
& & $6 \times 8$ &  &  &  &  & \cellcolor{gray!30}$7795266899$ & \cellcolor{gray!30}$70$\\
\hline
\hline
\multirow{1}{*}{$25$} & \multirow{1}{*}{$2$} & $5 \times 10$ &  &  &  &  & \cellcolor{gray!30}$7795567507$ & \cellcolor{gray!30}$41153$\\
\hline
\end{tabular}
\end{center}
\caption{A selection of admissible parameter sets with $v > 14$, with existence or the number of non-isotopic proper transposed mono arrays ($\mathrm{MA}^T$), proper sesqui arrays ($\mathrm{SA}$), proper double arrays ($\mathrm{DA}$), triple arrays ($\mathrm{TA}$), proper AO-arrays ($\mathrm{AO}$), and near triple arrays ($\mathrm{NTA}$).
Shaded cells correspond to new results. 
The symbol $-$ indicates that the existence of a proper design is ruled out by Proposition 3.2 from~\cite{jagerEnumerationRowColumnDesigns2024}, $+$ indicates that there are examples but we have no complete enumeration, and $?$ indicates that the question of existence is open.
An empty cell indicates that the parameter set is not admissible.}
\label{tbl:RCDfrom15}
\end{table}

\end{document}